   \pdfoutput=1

\documentclass{article}
\usepackage{excludeonly}

\usepackage{savesym}
\usepackage{scrextend}  
\usepackage{amsthm,amsmath}
\usepackage{amssymb,latexsym,graphicx}
\usepackage{amscd}
 \usepackage[all,cmtip]{xy}
\usepackage{tikz-cd}
  \usetikzlibrary{arrows}
\tikzset{
  commutative diagrams/.cd,
  arrow style=tikz
  }
\usepackage{tikz}

\usepackage{accents}
\usepackage{cite}
\usepackage{mathtools}
\usepackage{stmaryrd} 

\usepackage{calligra}
\DeclareMathAlphabet{\mathcalligra}{T1}{calligra}{m}{n}
\DeclareMathAlphabet{\mathpzc}{OT1}{pzc}{m}{it}
\usepackage{hycolor}
\usepackage{xcolor}
\usepackage[
                       colorlinks=true,
                       linkcolor=black, 
                       citecolor=black, 
                       urlcolor=blue,
                     ]{hyperref}
\usepackage{soul} 

\usepackage{makeidx}

\usepackage[intoc]{nomentbl}
\makenomenclature

\usepackage{stackengine} 
\usepackage{booktabs} 

\newtheorem{theoremABC}{Theorem}

\newtheorem{theorem}{Theorem}[section]

\newtheorem{corollary}[theorem]{Corollary}

\newtheorem{lemma}[theorem]{Lemma}
\newtheorem{proposition}[theorem]{Proposition}

\theoremstyle{definition}
\newtheorem{definition}[theorem]{Definition}

\newtheorem{convention}[theorem]{Convention}
\newtheorem{remark}[theorem]{Remark}

\theoremstyle{remark}
\newtheorem*{notation}{Notation}
\newcommand{\C}{{\mathbb{C}}}
\newcommand{\D}{{\mathbb{D}}}

\newcommand{\HH}{{\mathbb{H}}}

\newcommand{\R}{{\mathbb{R}}}
\renewcommand{\SS}{{\mathbb{S}}}

\newcommand{\Z}{{\mathbb{Z}}}
\newcommand{\Aa}{{\mathcal{A}}}   
\newcommand{\Bb}{{\mathcal{B}}}
\newcommand{\Cc}{{\mathcal{C}}}   

\newcommand{\Hh}{{\mathcal{H}}}

\newcommand{\Kk}{{\mathcal{K}}}
\newcommand{\Ll}{{\mathcal{L}}}   
\newcommand{\Mm}{{\mathcal{M}}}   

\newcommand{\Uu}{{\mathcal{U}}}
\newcommand{\Vv}{{\mathcal{V}}}

\newcommand{\AAA}{\mbf{A}}       
\newcommand{\aaa}{\mbf{a}}       %
\newcommand{\AAAdot}{\mbf{\dot A}} 
\newcommand{\bbb}{\mbf{b}}       %
\newcommand{\rot}{\mathrm{rot}}  
\newcommand{\im}{{\rm im\, }}             
\newcommand{\Id}{{\rm Id}}
\newcommand{\diag}{{\rm diag}}            
\newcommand{\grad}{\mathop{\mathrm{grad}}}    

\newcommand{\cgraph}[1]{\Gamma_{\kern-.5ex{}#1}}     
\renewcommand{\Re}{{\rm Re}}       
\newcommand{\Hess}{\mathrm{Hess}}          
\newcommand{\CAP}{\mathop{\cap}}           
 
\newcommand{\norm}{{\rm norm}}

\newcommand{\eps}{{\varepsilon}}

\newcommand{\Bfrak}{{\mathfrak B}}

\newcommand{\hfrak}{{\mathfrak h}}

\newcommand{\Qfrak}{{\mathfrak Q}}

\newcommand{\tfrak}{{\mathfrak t}}
\newcommand{\ufrak}{{\mathfrak u}}
\newcommand{\Ufrak}{{\mathfrak U}}
\newcommand{\Vfrak}{{\mathfrak V}}

\newcommand{\Zfrak}{{\mathfrak Z}}
\newcommand{\inner}[2]{\langle #1, #2\rangle}   
\newcommand{\INNER}[2]{\left\langle #1, #2\right\rangle}

\newcommand{\mbf}[1]{\text{\boldmath $#1$}}  

\def\Nablatop#1{\nabla^{#1}\kern-.5ex{}}

\def\NABLA#1{{\mathop{\nabla\kern-.5ex\lower1ex\hbox{$#1$}}}}
\def\Nabla#1{\nabla\kern-.5ex{}_{#1}}
\def\Tabla#1{\Tilde\nabla\kern-.5ex{}_{#1}}
\def\Babla#1{\widebar\nabla\kern-.5ex{}_{#1}}
\def\abs#1{\mathopen|#1\mathclose|}   
\def\Abs#1{\left|#1\right|}            
\def\norm#1{\mathopen\|#1\mathclose\|}
\def\Norm#1{\left\|#1\right\|}

\renewcommand{\Tilde}{\widetilde}

\newcommand{\p}{{\partial}}

\newcommand{\INTO}{\hookrightarrow}              

\renewcommand{\1}{{{\mathchoice {\rm 1\mskip-4mu l} {\rm 1\mskip-4mu l}
{\rm 1\mskip-4.5mu l} {\rm 1\mskip-5mu l}}}}

\newlength\eqshift
\renewcommand\theequation{\thesection.\arabic{equation}}
\let\savetheequation\theequation

\makeatletter
\renewcommand*\env@matrix[1][\arraystretch]{%
  \edef\arraystretch{#1}%
  \hskip -\arraycolsep
  \let\@ifnextchar\new@ifnextchar
  \array{*\c@MaxMatrixCols c}}
\makeatother

\makeatletter
\let\save@mathaccent\mathaccent
\newcommand*\if@single[3]{%
  \setbox0\hbox{${\mathaccent"0362{#1}}^H$}%
  \setbox2\hbox{${\mathaccent"0362{\kern0pt#1}}^H$}%
  \ifdim\ht0=\ht2 #3\else #2\fi
  }
\newcommand*\rel@kern[1]{\kern#1\dimexpr\macc@kerna}
\newcommand*\widebar[1]{\@ifnextchar^{{\wide@bar{#1}{0}}}{\wide@bar{#1}{1}}}
\newcommand*\wide@bar[2]{\if@single{#1}{\wide@bar@{#1}{#2}{1}}{\wide@bar@{#1}{#2}{2}}}
\newcommand*\wide@bar@[3]{%
  \begingroup
  \def\mathaccent##1##2{%
    \let\mathaccent\save@mathaccent
    \if#32 \let\macc@nucleus\first@char \fi
    \setbox\z@\hbox{$\macc@style{\macc@nucleus}_{}$}%
    \setbox\tw@\hbox{$\macc@style{\macc@nucleus}{}_{}$}%
    \dimen@\wd\tw@
    \advance\dimen@-\wd\z@
    \divide\dimen@ 3
    \@tempdima\wd\tw@
    \advance\@tempdima-\scriptspace
    \divide\@tempdima 10
    \advance\dimen@-\@tempdima
    \ifdim\dimen@>\z@ \dimen@0pt\fi
    \rel@kern{0.6}\kern-\dimen@
    \if#31
      \overline{\rel@kern{-0.6}\kern\dimen@\macc@nucleus\rel@kern{0.4}\kern\dimen@}%
      \advance\dimen@0.4\dimexpr\macc@kerna
      \let\final@kern#2%
      \ifdim\dimen@<\z@ \let\final@kern1\fi
      \if\final@kern1 \kern-\dimen@\fi
    \else
      \overline{\rel@kern{-0.6}\kern\dimen@#1}%
    \fi
  }%
  \macc@depth\@ne
  \let\math@bgroup\@empty \let\math@egroup\macc@set@skewchar
  \mathsurround\z@ \frozen@everymath{\mathgroup\macc@group\relax}%
  \macc@set@skewchar\relax
  \let\mathaccentV\macc@nested@a
  \if#31
    \macc@nested@a\relax111{#1}%
  \else
    \def\gobble@till@marker##1\endmarker{}%
    \futurelet\first@char\gobble@till@marker#1\endmarker
    \ifcat\noexpand\first@char A\else
      \def\first@char{}%
    \fi
    \macc@nested@a\relax111{\first@char}%
  \fi
  \endgroup
}
\makeatother

\def\XXint#1#2#3{{\setbox0=\hbox{$#1{#2#3}{\int}$}
     \vcenter{\hbox{$#2#3$}}\kern-.5\wd0}}

\long\def\symbolfootnote[#1]#2{\begingroup%
\def\thefootnote{\fnsymbol{footnote}}\footnote[#1]{#2}\endgroup}

\tikzset{
  symbol/.style={
    draw=none,
    every to/.append style={
      edge node={node [sloped, allow upside down, auto=false]{$#1$}}}
  }
}

\begin{document}
\sloppy

\author{\quad Urs Frauenfelder \quad \qquad\qquad
             Joa Weber\footnote{
  Email: urs.frauenfelder@math.uni-augsburg.de
  \hfill
  joa@unicamp.br
  }
        %
        %
    \\
    Universit\"at Augsburg \qquad\qquad
    UNICAMP
}

\title{Towards a Floer theory for Mars \\
         \Large II - Floer Hessian field almost extends}


\date{\today}

\maketitle 
%


%
%

%





\begin{abstract}
In part~I,~\cite{Frauenfelder:2026c}, we showed that
collisional periodic orbits of twisted Zeeman systems
can be detected variationally by a non-local Hamiltonian
action functional.

In this part~II we show that the linearized gradient flow
of this non-local functional is a Fredholm operator
and prove a non-local elliptic regularity result.

These results are obtained with the theory of almost extendability
of weak Hessian fields introduced in~\cite{Frauenfelder:2025g}.
\end{abstract}

\tableofcontents

\section{Introduction}

In~\cite{Frauenfelder:2026c} we introduce the notion of twisted Zeeman
system.
This system models an electron in the plane
attracted by a proton fixed in the origin and
subject to additional forces:
an electric force,
a Lorentz force,
and an Euler force.
All these additional forces are allowed to depend periodically on time.

A motivating example of such a system is the restricted elliptic
three-body-problem.
In the example of the restricted elliptic three-body-problem the
proton can be thought of as the planet Mars,
the electron is a space station in the vicinity of Mars.
The Lorentz force is given by the Coriolis force of the rotating frame.
The electric force is a combination of the gravitational force of the
sun and the centrifugal force.
The Euler force is an additional fictitious force, beside
the Coriolis force and the centrifugal force, due to acceleration
and deceleration of the rotating frame,
which is due to the eccentricity of the elliptic orbit of Mars.

There is the danger that the electron collides with the proton.
However two-body-collisions can be regularized.
For systems depending periodically on time
a new regularization technique was discovered by
Barutello-Ortega-Verzini~\cite{Barutello:2021b}
which blows up the loop space.
In part I,~\cite{Frauenfelder:2026c}, we showed that
this new regularization technique can be applied to
twisted Zeeman systems and gives rise to a variational approach
to collisional periodic orbits of a twisted Zeeman system.
In fact, we showed that there are two action functionals to
detect these periodic collisional orbits, a Lagrangian one and a
Hamiltonian one, related by a non-local Legendre transformation.
Since blowing up the loop space is a non-local technique,
both of these functionals are non-local.

\subsection{Main results}

In this article we prove a Fredholm result for
the linearized gradient flow of the Hamiltonian functional
for regularized twisted Zeeman systems.
\\
Given such a system, the second derivative of the regularized
Hamiltonian functional gives rise to a weak Hessian field
$u\mapsto A^u$, roughly as follows.
Namely, for the Hilbert space triple of maps
$$
   (H_0,H_1,H_0)
   =\left(L^2(\SS^1,\C^2),W^{1,2}(\SS^1,\C^2),W^{2,2}(\SS^1,\C^2)\right)
$$
there exists an open subset $U_1\subset H_1$
such that for any $u\in U_2:=U_1\cap H_2$ the Hessian operator
$A^u\colon H_1\to H_0$ is a Fredholm operator of index zero
which restricts to a Fredholm operator 
$A^u_2\colon H_2\to H_1$ also of index zero.
\\
Given $u_-,u_+\in U_2$ such that $A^{u_-}$ and $A^{u_+}$
are invertible as maps $H_1\to H_0$,
we are considering a connecting path
$u\colon\R\to U_2$ from $u_-$ to $u_+$
(Definition~\ref{def:connecting-paths}).
This gives rise to two linear operators 
\begin{equation*}
\begin{split}
   \D^u=\p_s+A^u\colon
   &W^{1,2}(\R,H_0)\cap L^2(\R,H_1)
   \to L^2 (\R,H_0)
\\
   \D^u_2=\p_s+A^u_2\colon 
   &W^{1,2}(\R,H_1)\cap L^2(\R,H_2)
   \to L^2 (\R,H_1)
   .
\end{split}
\end{equation*}
In the present article we show furthermore

\begin{theoremABC}\label{thm:A}
Both $\D^u$ and $\D^u_2$
are Fredholm and their Fredholm indices agree.
\end{theoremABC}

The proof of Theorem~\ref{thm:A} uses a technique we developed
in~\cite{Frauenfelder:2025g}.
In that article we introduced the notion of an \emph{almost extendable
weak Hessian field}.

\begin{theoremABC}\label{thm:B}
The weak Hessian field
$A=\{A^u\}_{u\in U_1}$ is almost extendable.
\end{theoremABC}

Theorem~\ref{thm:A} then follows from Theorem~\ref{thm:B}
in view of the abstract result~\cite[Thm.\,6.11]{Frauenfelder:2025g};
cf. \S\,\ref{sec:proof-AB}.

\medskip\noindent
\textsc{Non-local elliptic regularity.}
Theorem~\ref{thm:A} can be thought of as a
non-local elliptic regularity result.
Since blowing up the loop space is non-local,
elements of the kernel and cokernel of the operator $\D^u$
cannot be thought of as solutions of an elliptic PDE.\footnote{
  If the solutions in the kernel of $\D^u$ would satisfy an elliptic
  PDE, by elliptic regularity they would automatically lie in the kernel
  of $\D^u_2$, and therefore the kernel and cokernel of $\D^u$ could be
  identified  with the kernel and cokernel of $\D^u_2$, so that
  both operators would have the same index.  
  }
But in view of Theorem~\ref{thm:A} one might interpret them
as solutions of an 'elliptic' delay equation.

\medskip
In this paper we consider the Fredholm property
in the Hamiltonian setup.
The functional $\Aa$ considered in this paper can be interpreted as the
Legendre transform of a Lagrangian one, as we explained in the
introduction,
and therefore the Fredholm property can also be asked for
the Lagrangian functional $\Bb$.
In Appendix~\ref{sec:Lag-Kepler}
we show that in the Kepler case (no Coriolis/magnetic contribution)
the Hessian of the Lagrangian functional is almost extendable as well.

The relation between the Fredholm indices
in the Lagrangian case and in the Hamiltonian case
for these non-local functionals we plan to discuss in a forthcoming
part III~\cite{Frauenfelder:2026e}.

\subsection{Outline}

Throughout we allow for twisted-periodic one forms (vector potentials)
and also for twisted loops, as defined in~\S\,\ref{sec:twist}.

\smallskip
Section~\ref{sec:Zeeman} summarizes
notions and results of~\cite{Frauenfelder:2026b}
on which the present article is based.

\smallskip
Section~\ref{sec:weak-Hess-fields} summarizes
notions and results of~\cite{Frauenfelder:2025g}
which are needed to understand the statement in Theorem~\ref{thm:B}
and to see how Theorem~\ref{thm:A} follows from Theorem~\ref{thm:B}.

\smallskip
Section~\ref{sec:Hess-Ham} is the main part of this article.
We prove Theorems~\ref{thm:A} and~\ref{thm:B}.

\smallskip
Section~\ref{sec:L2-Hess-symmetry}
discusses the symmetry of weak Hessians
from an abstract point of view.

\smallskip
Appendix~\ref{sec:Hess-Lag} deals with the Hessian field of the
regularized Lagrangian action functional.
The Hessian is calculated in~\ref{sec:calc-Lag-Hess}.
The magnetic contributions also
play a crucial role in the main Section~\ref{sec:Hess-Ham} of this article.
Having all the machinery in place
we show in~\ref{sec:Lag-Kepler} that in the Kepler case
(no magnetic field) the Lagrangian Hessian field almost extends
(while the Hamiltonian one in~\S\,\ref{sec:Ham-Hess-Kepler-extends}
even extended).
The difference comes from the Hamiltonian equations being first
order versus second order of the Lagrangian ones.
Unfortunately, in the general Lagrangian case (including magnetic
contributions) doing the scale Lipschitz estimate seems practically
hopeless, because second order leads to an explosion of the number of
terms, even in comparison to the already very long calculation
in the Hamiltonian scenario~\S\,\ref{sec:sc-Lipschitz-Ham}.

\smallskip
Section~\ref{sec:ests-summands-Hess-mag-Ham}
carries out the technical estimates.

\medskip
\begin{notation}
Throughout $\inner{\cdot}{\cdot}$ denotes
$L^2$ inner products with induced norm $\norm{\cdot}$.
The induced norm of a Hilbert space $H_k$ is often denoted by
$\abs{\cdot}_{H_k}$ or $\abs{\cdot}_k$.

Working with functions on function spaces easily triggers
excesses of parentheses, which harms legibility.
Therefore we often write variables either as subscripts
$\Mm_z$ or in the form $\Mm|_z$, as opposed to $\Mm(z)$.
For time-dependence subscript has priority, for example if $t\mapsto
q(t)$ is a loop then $\theta_t|_{q_t} \dot q_t$ denotes a
time-dependent $1$-form at time $t$ and at the spatial point $q(t)$
evaluated on the velocity vector $\dot q(t)$.
\end{notation}

\medskip\noindent
{\bf Acknowledgements.}
UF~acknowledges 
support by DFG grant
FR~2637/5-1.

\section{Regularized twisted Zeeman systems}
\label{sec:Zeeman}

\boldmath
\subsection{Let's twist again}\label{sec:twist}
\unboldmath

\boldmath
\subsubsection*{Euclidean plane -- two models $\C$ and $\R^2$}
\unboldmath

Configuration spaces in this article are open subsets of the Euclidean
plane containing the origin.
There are two natural models for the plane,
the set $\C$ of complex numbers and the set $\R^2$ of pairs of real
numbers, each one endowed with the natural structure of a vector
space over the real numbers $\R$.
The natural isomorphism $\C\to\R^2$, $x+iy\mapsto (x,y)$,
identifies multiplication by $i$ viewed as linear map on $\C$
with the matrix $j_0\colon\R^2\to\R^2$ of counter-clockwise rotation
by $\pi/2$.
As each model has its advantage,
\emph{we freely change from one to the other}.
The complex side is extremely effective
with respect to notation, complex multiplication
$zw=(z+iy)(u+iv)$ encodes composition of linear maps
as well as application of a linear map $z$ to a vector $w$,
on the $\R^2$-side this is
$$
\small
   \begin{pmatrix}
      x&-y\\y&x
   \end{pmatrix}
   \begin{pmatrix}
      u&-v\\v&u
   \end{pmatrix}
   =
   \begin{pmatrix}
      xu-yv&-(yu+xv)\\yu+xv&xu-yv
   \end{pmatrix}
   ,\quad
   \begin{pmatrix}
      x&-y\\y&x
   \end{pmatrix}
   \begin{pmatrix}
      u\\v
   \end{pmatrix}
   =
   \begin{pmatrix}
      xu-yv\\yu+xv
   \end{pmatrix}
   .
$$
Instead of dealing with inverse matrices one just multiplies and
divides by a real number, namely
$z w^{-1}=\frac{z}{w}=\frac{z\bar w}{w\bar w}$
where $w\bar w=u^2+v^2=:\abs{w}^2$.
The $\R^2$-side might appeal more to ones geometric intuition,
or just ones customs.

Concerning the Euclidean inner product, we feel free to write
$$
   \inner{z}{w}_0=xu+yv=\Re(\bar z w)
$$
using on the left pairs $(x,y)$, $(u,v)$ and
on the right sums $x+iy$ and $u+iv$.

\boldmath
\subsubsection*{Complex squaring map and sign involution}
\unboldmath

Throughout $0\in\Qfrak\subset\C$ is an open subset which contains the
origin, also called singularity or collision locus.
To exclude the origin we write $\Qfrak^\times:=\Qfrak\setminus\{0\}$.

\begin{remark}[complex square root is not continuous]
The complex squaring map $\varsigma\colon\C^\times\to\C^\times$,
$z\mapsto z^2=(-z)^2$, is not a bijection, but only 2:1 due to the
sign ambiguity.
Reverting the perspective, let us define the complex square root
$\sqrt{re^{i\varphi}}$, in analogy to the real case,
as that one of the two candidates $\pm \sqrt{r}e^{i\phi/2}$
which lies in the upper half plane $\HH^+$ or on the positve half axis
$\R^+$.
One gets to a bijection, see~(\ref{eq:square-root}), by dividing out
the sign $\pm$.
Unfortunately, the square root definition is not
continuous, as illustrated in Figure~\ref{fig:fig-square-root}:
while the unit circle elements $w$ and $1$ are close, their square roots
are not.
\end{remark}
\begin{figure}[h]
  \centering
  \includegraphics
                             {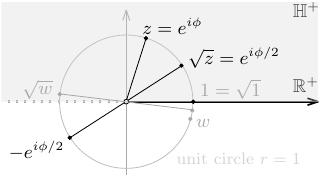}
  \caption{Square root $\sqrt{z}\in\HH^+\CAP\R^+$
                of non-zero $z\in\C$ -- not continuous}
   \label{fig:fig-square-root}
\end{figure} 

The pre-image $\Zfrak^\times:=\varsigma^{-1}(\Qfrak)$
is invariant under sign involution 
and complex squaring is a double cover and invariant under sign
involution, in symbols
\begin{equation}\label{eq:I-Zfrak-i}
   \mathfrak{i}\colon \Zfrak^\times\to \Zfrak^\times
   ,\;
   z\mapsto-z
    ,\qquad
   \varsigma\colon\Zfrak^\times
   \stackrel{2 : 1}{\longrightarrow}
   \Qfrak^\times
   ,\;
   z\mapsto z^2
   ,\qquad
   \varsigma\circ \mathfrak{i}=\varsigma
   .
\end{equation}
Dividing out by sign involution,
complex squaring becomes a bijection
whose inverse is the complex square root, in symbols
%
\begin{equation}\label{eq:square-root}
\begin{tikzcd} [row sep=tiny] 
\overline{\Zfrak^\times}=\Zfrak^\times/\mathfrak{i}
    \arrow[rr, shift left=2, "1:1"']
  &&
     \Qfrak^\times
     \arrow[ll, shift left=2]
\\
  \quad\;\;\pm z\quad\;\;
    \arrow[rr, mapsto, "\varsigma"]
  &&
    z^2
\\
  \pm\sqrt{r}e^{i\phi/2}
  &&
    re^{i\phi}=q
    \arrow[ll, mapsto, "\sqrt{\cdot}"']
\end{tikzcd}
\end{equation}
%
To ease notation we use for space and quotient space elements
the same notation.

\boldmath
\subsubsection*{Twisted loops}
\unboldmath

While loops $q\colon\SS^1\to\Qfrak^\times\subset\C$ of even winding
number around the origin lift to loops ($+$) in $\Zfrak^\times$,
loops of odd winding number lift to twisted ($-$) loops in $\Zfrak^\times$.
This is illustrated in Figure~\ref{fig:fig-twisted}.
A map $\gamma$ with domain $\R$ such that
$\gamma_{\tau+1}=\gamma_\tau$ for all
$\tau$ is called \textbf{periodic}, domain notation $\SS^1=\R/\Z$.

\medskip\noindent
The spaces of \textbf{(periodic) loops} ($+$) and
\textbf{twisted loops}  ($-$),
still called loops, 
\begin{equation}\label{eq:twisted-loop-space}
   \Ll_\pm^\times\Zfrak
   :=\{z\in C^\infty (\R,\Zfrak)\mid\forall \tau\in\R\colon
   z_{\tau+1}=\pm z_\tau\}
   \setminus\{0\}
   ,
\end{equation}
are disjoint and
invariant under \textbf{sign involution}
%
%
$I\colon \Ll_\pm\Zfrak\to \Ll_\pm\Zfrak$, $z\mapsto -z$
which acts freely.
Set $\Ll^\times\Zfrak=\Ll_+^\times\Zfrak\cup \Ll_-^\times\Zfrak$.
The elements $\tau_*$ of the set $z^{-1}(0)$
are called \textbf{collision times} or simply \textbf{collisions}.

\medskip\noindent
The two cotangent bundles are disjoint
\begin{equation}\label{eq:twisted-loop-space-*}
\small
\begin{split}
   T^*\Ll_\pm^\times\Zfrak:
   &=\{(z,\eta)\in C^\infty (\R,\Zfrak\times\C)\mid
   z\not\equiv0,\,
   \left(z_{\tau+1},\eta_{\tau+1}\right)=\pm\left(z_{\tau},\eta_{\tau}\right)
   ,\forall \tau
   \}
\\
   &=\Ll_\pm^\times\Zfrak \times \Ll_\pm\C   
\end{split}
\end{equation}
and invariant under the \textbf{sign involution}
$T^*I(z,\eta)=-(z,\eta)$ which acts freely.
The \textbf{base point projection} 
$$
   \pi\colon T^*\Ll_\pm^\times\Zfrak\to \Ll_\pm^\times\Zfrak
   ,\quad
   \Upsilon=(z,\eta)\mapsto z
   ,\qquad
   \pi\circ T^*I=I\circ \pi
   ,
$$
is sign involution equivariant.
The tangent spaces $T\Ll_\pm^\times\Zfrak$
are given by the same formulas (since
Euclidean $\C$ and its dual space are canonically isomorphic). 
The uppercase greek letter $\Xi=(z,\xi)$ is a ``Xi''
and $\Upsilon=(z,\eta)$ is an ``Upsilon''.
\begin{figure}[h]
  \centering
  \includegraphics
                             {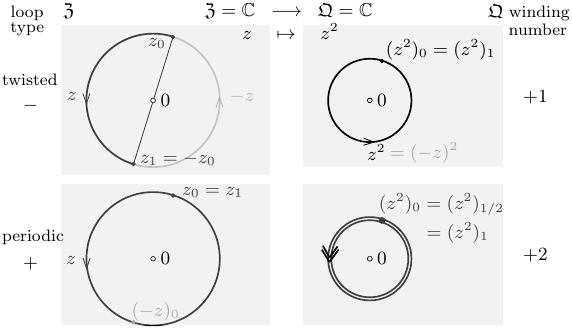}
  \caption{Loop correspondence -- $\pm$ loops $z$ in $\Zfrak$ and
                loops $q=z^2$ in $\Qfrak$}
   \label{fig:fig-twisted}
\end{figure}

\begin{remark}
While 
in part I,~\cite{Frauenfelder:2026c}, it was important to build
a model for the paths in $\Qfrak$, in particular to have
a $1:1$ correspondence, and therefore we had to quotient out by sign
involution $\pm$, the present part II is of analytic nature and
the analysis is done on representatives anyway.
Hence, for simplicity, in the present text we work directly on
loop spaces, not on quotient spaces.
\end{remark}

\boldmath
\subsubsection*{Twisted-periodic one-forms}
\unboldmath

If one allows the $1$-form to be twisted-periodic
and not just periodic, then
additional (external) electric forces for an electric potential depending
periodically on time can be modeled by the twist,
see~\cite[\S\,3.3]{Frauenfelder:2026a}.

\begin{definition}[twisted-periodic]
\label{def:tw-per-1-form}
A \textbf{twisted-periodic \boldmath$1$-form} $\theta$
is a smooth family of $1$-forms $\{\theta_t\}_{t\in\R}$ on $\Qfrak$
such that a)~the time-derivative is periodic
and b)~$\forall t$ $\exists$ a smooth \textbf{twist function}
$f_t\colon\Qfrak\to \R$, 
that is
\begin{equation}\label{eq:df_t}
   \text{a) $\dot\theta_{t+1}=\dot\theta_t$}
   ,\quad
   \text{b) $\theta_{t+1}=\theta_t+df_t$}
   ,\qquad
{\color{gray}
   \text{c) $\AAAdot_{t+1}\stackrel{\text{a)}}{=}\AAAdot_t$}
   ,\quad
   \text{d) $d\theta_{t+1}\stackrel{\text{b)}}{=} d\theta_t$}
   .
}
\end{equation}
The time-dependent vector field
$$
   \AAA_t:=\left(A_t^1,A_t^2\right)\colon\Qfrak\to\R^2
   ,\qquad
   \{\theta_t=A_t^1\, dq_1+A_t^2\, dq_2\}_{t\in\R}
   ,
$$
is called a \textbf{vector potential} of the magnetic field $d\theta$.
\end{definition}

\begin{remark}\label{rmk:tw-per}
Let $\theta$ be twisted-periodic along $\Qfrak$ with twist function $f$.
Then the following is true
by~\cite[\S5]{Frauenfelder:2026a}$_{\text{(i--ii)}}$
and~\cite[\S4.1]{Frauenfelder:2026b}$_{\text{(iii)}}$.

\begin{itemize}\setlength\itemsep{0ex}
\item[\rm(i)]
  The time-slice $f:=f_0\colon\Qfrak\to\R$ is a twist function.
\item[\rm(ii)]
   In the periodic case ($\theta_{t+1}=\theta_t$, $\forall t$)  the
   twist $f=0$ vanishes.
\item[\rm(iii)]
  The pull-back $1$-form
  $\vartheta:=\varsigma^*\theta$ under complex
  squaring~(\ref{eq:I-Zfrak-i}) is twisted-periodic along
  $\Zfrak=\varsigma^{-1}(\Qfrak)$ and
  $\varsigma^* f$ is a twist function.
  \\
  The vector potential $\aaa$ of $\vartheta$ along $\Zfrak$, notation
  $$
     \aaa_t:=\left(a_t^1,a_t^2\right)\colon\Zfrak\to\R^2 
     ,\qquad
     \{\vartheta_t=a_t^1\, dx+a_t^2\, dy\}_{t\in\R}
     ,
  $$
  and $\AAA$ of $\theta$ along $\Qfrak$
  satisfy, at any point $z\in\Zfrak$, the identities
  \begin{equation*}
  \begin{split}
     \rot\,\aaa_t|_z :
     &=(\p_x a_t^2-\p_y a_t^1)_z
     =4\Abs{z}^2 \rot\,\AAA_t|_{z^2}
  \\
     (d\vartheta_t)_z
     &=(\rot\,\aaa_t|_z)\,\INNER{j_0\cdot}{\cdot}_0
     =4\Abs{z}^2(\rot\,\AAA_t|_{z^2})\, \INNER{j_0\cdot}{\cdot}_0 .
  \end{split}
  \end{equation*}
\end{itemize}
\end{remark}

\boldmath
\subsubsection*{Barutello-Ortega-Verzini reparametrization}
\unboldmath

Given $z\in \Ll^\times\Zfrak$,
the variable $\tau$ of $z\colon\R\to\Zfrak^\times$
is \textbf{regularized time},
\textbf{classical time} are the values of the map
$t_z\colon\SS^1\to\SS^1$ defined by (cf. App.\,~\ref{sec:t_z})
\begin{equation}\label{eq:class-time-tau-C}
   \forall\tau\in\R\colon
   \quad
   t_z(\tau)
   :=\frac{\int_0^\tau \Abs{z(s)}^2\, ds}{\norm{z}^2}
   ,\quad
   t_z(0)=0
   ,\quad
   t_z(1)=1
   .
\end{equation}

\boldmath
\subsection{Lagrangian action functional $\Bb$}
\unboldmath


For a twisted-periodic $1$-form $\theta$
along $\Qfrak$ with twist function $f$, see~(\ref{eq:df_t}),
we defined in~\cite[\S 4.2]{Frauenfelder:2026b}
the \textbf{non-local Lagrangian action functional}
\begin{equation}\label{eq:Bb}
\begin{split}
   \Bb
   \colon \Ll^\times\Zfrak
    &\to\R
{\color{gray}
   \qquad\qquad \qquad\qquad \qquad\qquad\;\;
   \text{\small
   $\Bb=\Kk-\Uu+\Mm$
   }
}
\\
   z
   &\mapsto
   \underbrace{2\norm{z}^2\norm{z^\prime}^2}_{\Kk(z)}      
   -\underbrace{{\color{cyan}\frac{-1}{\norm{z}^2}}}_{\Uu(z)}                  
   +\underbrace{\int_0^1\vartheta_{t_z(\tau)}|_{z_\tau}         
   z^\prime_\tau\, d\tau
   -f(z_0^2)}_{\Mm(z)}
   .
\end{split}
\end{equation}
Here $\vartheta:=\varsigma^*\theta$ is twisted-periodic along $\Zfrak$
with twist function $f^*\sigma$ and $\varsigma\colon\Zfrak\to\Qfrak$
is the complex squaring map~(\ref{eq:I-Zfrak-i}).
Observe that $\Bb=\Kk-\Uu+\Mm$ is the sum of 
three terms, kinetic, potential, and magnetic non-local action.

\boldmath
\subsubsection*{$L^2$-gradient}
\unboldmath

We showed in~\cite[\S 4.2]{Frauenfelder:2026b} that the $L^2$-gradient
is given by the formula
\begin{equation}\label{eq:grad-Bb}
\begin{split}
   \grad\Bb|_z
   &=
{\color{brown}\; 
   \grad\Kk|_z
}
   -
{\color{cyan}\;
   \grad\Uu|_z
}
   +\grad\Mm|_z
\\
   &=
{\color{brown}\; 
   4\norm{z^\prime}^2z-4\norm{z}^2z^{\prime\prime}
}
   -
{\color{cyan}\;
   \frac{2z}{\norm{z}^4}
}
\\
   &\quad
   -\frac{2z}{\norm{z}^4}
   \int_0^1
{\textstyle
   \int_0^\sigma\Abs{z_\rho}^2 d\rho
}
   \cdot
   \inner{\dot{\aaa}_{t_z(\sigma)}|_{z_\sigma}}{z^\prime_\sigma}_0\; d\sigma
   -\frac{\Abs{z}^2}{\norm{z}^2}\dot{\aaa}_{t_z}|_z
   \\
   &\quad
   +\frac{2 z}{\norm{z}^2}
   \int_{\sigma=\tau}^1
   \inner{\dot{\aaa}_{t_z(\sigma)}|_{z_\sigma}}{z^\prime_\sigma}_0\;
   d\sigma
   -\bigl(\rot\,\aaa_{t_z}|_z\bigr)\; j_0 z^\prime
\end{split}
\end{equation}
whenever $z\in \Ll^\times\Zfrak$.
Note that all integrands are periodic due to~(\ref{eq:df_t})
and so is, using Remark~\ref{rmk:tw-per}\,(iii), the final summand
$\bigl(\rot\,\aaa_{t_z}|_z\bigr)\; \inner{j_0 z^\prime}{\cdot}_0
=d\vartheta_{t_z}|_z(z^\prime,\cdot)$

\boldmath
\subsection{Hamiltonian action functional $\Aa$}
\label{sec:Aa_Hh}
\unboldmath

On the cotangent bundle $T^*\Ll^\times\Zfrak$,
see~(\ref{eq:twisted-loop-space-*}),
define the \textbf{mechanic Hamiltonian}
\begin{equation}\label{eq:Hh}
\begin{split}
   \Hh=\colon
   T^*\Ll^\times\Zfrak
   &\to\R
   ,\quad
   (z,\eta)
   \mapsto
   \tfrac12 \INNER{\eta}{\eta}^z
   -\tfrac{1}{\norm{z}^2}
   =
\underbrace{
   \tfrac{\norm{\eta}^2}{8\norm{z}^2}
}_{\Kk^*(z,\eta)}
\underbrace{
{\;\color{cyan}
   -\tfrac{1}{\norm{z}^2}
}
}_{+\;{\color{cyan}\Uu(z)}}
   .
\end{split}
\end{equation}
In~\cite[\S 5.6]{Frauenfelder:2026b}
the Hamiltonian is obtained as the Legendre dual of 
the natural extension of $\Bb$ to the tangent bundle.
The \textbf{\boldmath$L^2$-gradient of $\Hh$}
is determined by
$$
   d\Hh|_{(z,\eta)}=\INNER{\grad\Hh|_{(z,\eta)}}{\cdot}
   .
$$

\begin{lemma}[$L^2$-gradient]
\label{le:grad-Hh}
At $(z,\eta)\in T^*\Ll^\times\Zfrak$ the value
of $\grad(\Kk^*+\pi^*\Uu)$ is
\begin{equation}\label{eq:grad-Hh}
\begin{split}
   \grad\Hh|_{(z,\eta)}
   &=\tfrac{1}{4\norm{z}^2}
   \begin{pmatrix}
      \tfrac{8-\norm{\eta}^2}{\norm{z}^2}\, z
      \\
      \eta
   \end{pmatrix}
   =
   -\tfrac{1}{4\norm{z}^2}
   \begin{pmatrix}
      \tfrac{\norm{\eta}^2}{\norm{z}^2}\, z
      \\
      \eta
   \end{pmatrix}
   +
   \begin{pmatrix}
      \tfrac{2z}{\norm{z}^4}
      \\
      0
   \end{pmatrix}
   .
\end{split}
\end{equation}
\end{lemma}

\begin{proof}
The gradient is the pair $\left(\p_z\Hh,\p_\eta\Hh\right)$.
Component one is obvious.
Observe that $\p_z(-\norm{z}^{-2})=-\p_z(\norm{z}^2)^{-1}
=-(-1)(\norm{z}^2)^{-2} 2z=
{\color{cyan}
\frac{2z}{\norm{z}^4}
=\grad\Uu_z
}
$.
\end{proof}


We define the symplectic action ad-hoc\footnote{
  Abstractly one would twist $\Lambda$ by adding the pull-back
  of a $1$-form $\Theta$ on loop space.
  This works in the periodic case in which such $\Theta$
  exists as discussed in~\cite[App.\,C]{Frauenfelder:2026b}.
  Calculating $d\Theta$ the formula obtained makes sense
  in the general, twisted-periodic, case defining
  a $2$-form $\Sigma$ on loop
  space~\cite[App.\,B]{Frauenfelder:2026b}, just not an exact one.
  Similarly, while $\Mm$ is not defined on loop space,
  since $\vartheta$ is not necessarily periodic,
  its gradient~(\ref{eq:grad-Bb}) is, hence~(\ref{eq:grad-Aa}) is.
  }
in Definition~\ref{def:Aa_Hh} below
where $\Vv(z,\eta):=(z^\prime,\eta^\prime)$
is the canonical loop space vector field and $\Lambda$
arises by integrating the Liouville form $\lambda$ associated to
$T^*\Zfrak$; for details see~\cite[\S 5]{Frauenfelder:2026b}.

\begin{definition}\label{def:Aa_Hh}
By definition
$\Aa:=i_\Vv\Lambda+\pi^*\Mm-\Kk^*{\color{cyan}\,-\,\pi^*\Uu}$,
that is
\begin{equation}\label{eq:Aa_Hh}
\begin{split}
   \Aa
   \colon T^*\Ll^\times\Zfrak
   &\to\R
\\
   (z,\eta)
   &\mapsto
   \INNER{\eta}{z^\prime}
   +\Mm(z)
   -\tfrac{\norm{\eta}^2}{8\norm{z}^2} 
   +{\color{cyan}\tfrac{1}{\norm{z}^2}}
\\
   &\quad\; =
   -\INNER{\eta^\prime}{z}
   +\Mm(z)
   -\tfrac{\norm{\eta}^2}{8\norm{z}^2} 
   +{\color{cyan}\tfrac{1}{\norm{z}^2}}
\end{split}
\end{equation}
where equality is integration by parts
and the magnetic term $\Mm$ is given by~(\ref{eq:Bb}).
\end{definition}

\begin{lemma}[$L^2$-gradient]\label{le:grad-Aa}
At $(z,\eta)\in T^*\Ll^\times\Zfrak$
the $L^2$-gradient is given by
\begin{equation}\label{eq:grad-Aa}
\begin{split}
   \grad\Aa|_{(z,\eta)}
   &=
   \begin{pmatrix}
      -\eta^\prime
      \\
      z^\prime
   \end{pmatrix}
   +
   \begin{pmatrix}
      \grad\Mm|_z-{\color{cyan}\grad\Uu|_z}
      \\
      0
   \end{pmatrix}
   +
   \tfrac{1}{4\norm{z}^2}
   \begin{pmatrix}
      \tfrac{\norm{\eta}^2}{\norm{z}^2}\, z
      \\
      -\eta
   \end{pmatrix}
\\
   &=
   \begin{pmatrix}
   -\eta^\prime+\grad\Mm|_z-{\color{cyan}\tfrac{2z}{\norm{z}^4}}
   +\tfrac{\norm{\eta}^2}{4\norm{z}^4}\, z
   \\
   z^\prime-\tfrac{1}{4\norm{z}^2} \eta
   \end{pmatrix}
   .
\end{split}
\end{equation}
\end{lemma}

\begin{proof}
Use~(\ref{eq:grad-Hh}) and integrate by parts
$\inner{z}{\eta^\prime}=\underline{0}-\inner{z^\prime}{\eta}$.
Here the boundary terms
$
   \inner{z_1}{\eta_1}_0
   -\inner{z_0}{\eta_0}_0
   =\inner{\pm z_0}{\pm\eta_0}_0
   -\inner{z_0}{\eta_0}_0
   =\left((\pm 1)^2-1\right) \inner{z_0}{\eta_0}_0
   =\underline{0}
$
vanish, since indeed $(z_1,\eta_1)=\pm (z_0,\eta_0)=(\pm z_0,\pm\eta_0)$
by~(\ref{eq:twisted-loop-space-*}).
\end{proof}

\section{Weak Hessian fields and almost extendability}
\label{sec:weak-Hess-fields}

Consider an abstract Hilbert space triple $(H_0,H_1,H_2)$.
Let $U_1\subset H_1$ be open. Set $U_2:=U_1\cap H_2$.
The following definitions are from the article~\cite{Frauenfelder:2025g}.

\boldmath
\subsection{Almost extendability}\label{sec:pDarboux-ae}
\unboldmath

\begin{definition}[Weak Hessian field]\label{def:weak-Hessian-field}
A \textbf{weak Hessian field} on $U_1$ is a continuous map
$A\in C^0(U_1,\Ll(H_1,H_0)) \cap C^0(U_2,\Ll(H_2,H_1))$,
notation $u\mapsto A^u$ and $A^u_2:=A^u|_{H_2}\colon H_2\to H_1$,
satisfying the following two conditions:
\begin{labeling}{\texttt{(Fredholm)}}
\item[\texttt{(Symmetry)}]
  At any point $u\in U_1$ there is $H_0$-symmetry in the sense that
  \begin{equation}\label{eq:u-0-symmetric}
     \forall x,y\in H_1\colon \quad
     \INNER{A^u x}{y}_{H_0}=\INNER{x}{A^u y}_{H_0} .
  \end{equation}
\item[\texttt{(Fredholm)}]
  $\forall u\in U_1\colon$
  $A^u\colon H_1\to H_0$ is Fredholm of index zero.
  
  $\forall u\in U_2\colon$ 
  $A^u_2\colon H_2\to H_1$ is Fredholm of index zero.
\end{labeling}
\end{definition}

\begin{definition}[Extendability]\label{def:extends}
We say that a weak Hessian field $A$ on $U_1$ \textbf{extends}
if $A$ extends to a continuous map
$U_1\to \Ll(H_1,H_0)\cap\underline{\Ll(H_2,H_1)}$,
still denoted by $u\mapsto A^u$, such that the restriction
$A^u_2:=A^u|_{H_2}\colon H_2\to H_1$ is Fredholm of index zero at
every point $u$ of \underline{$U_1$}, and not only of $U_2$.
\end{definition}

In general, the extendability condition is too strong;
see~\cite[\S 6.1]{Frauenfelder:2025g}.
A way out is to decompose the operator family $A$ into two summands.
The idea why and how this should be done is detailed
in~\cite[\S 5.2]{Frauenfelder:2025g}.

\begin{definition}[Almost extendability]\label{def:alm.ext}
(i)~We say that a weak Hessian field $A$ on $U_1$ \textbf{almost extends}
if there exists a decomposition
\begin{equation}\label{eq:A=F+C}
   A=F+C
\end{equation}
with
\begin{equation}\label{eq:hyp-F}
   F\in C^0(U_1,\Ll(H_1,H_0)\cap\Ll(H_2,H_1))
\end{equation}
and
\begin{equation}\label{eq:hyp-C}
   \exists r\in[0,1)\colon
   \quad
   C\in C^0(U_1,\Ll(\underline{H_r},H_0))\cap C^0(U_2,\Ll(\underline{H_1},H_1))
\end{equation}
such that the following two axioms hold.
\begin{labeling}{\texttt{(F)}}

\item[\texttt{(F)}]
  $\forall u \in\underline{\, U_1}\colon$
  $F^u_2:=F^u|_{H_2}\colon H_2\to H_1$
  is Fredholm of index zero.

\item[\texttt{(C)}]
  $\forall u\in U_1$ there exists an $H_1$-open neighborhood $V_u$ of $u$
  and a constant $\kappa$ such that
  for all $v,w\in V_u\cap H_2$ it holds the
  \emph{scale Lipschitz estimate}
  \begin{equation}\label{eq:(C)}
     \norm{C^v-C^w}_{\Ll(H_1)}
     \le\kappa\Bigl(\abs{v-w}_{H_2}
     +\min\{\abs{v}_{H_2},\abs{w}_{H_2}\}\cdot\abs{v-w}_{H_1}\Bigr).
  \end{equation}
\end{labeling}
(ii)~If $A$ almost extends we call the pair $(F,C)$
a \textbf{decomposition of \boldmath$A$}.
\end{definition}

\boldmath
\subsection{A Fredholm result}\label{sec:pDarboux-Fred}
\unboldmath

\begin{definition}[Connecting paths]\label{def:connecting-paths}
Fix two points $u_-,u_+\in U_2:=U_1\cap H_2$.
Fix a \textbf{basic path} $\hat u$ from $u_-$ to $u_+$
(see~\cite{Frauenfelder:2025e}), 
i.e. $\hat u\in C^2(\R,U_2)$ with the property that there exists $T>0$
such that $\hat u(s)=u_-$ whenever $s\le -T$ and
$\hat u(s)=u_+$ whenever $s\ge T$.
A \textbf{connecting path from \boldmath$u_-$ to $u_+$}
is a continuous map $u\colon \R\to U_1$
such that the difference $u-\hat u$ lies in the intersection
Hilbert space $W^{1,2}(\R,H_1)\cap L^2(\R,H_2)$, i.e.
\begin{equation*}
   u\in C^0(\R,U_1)
   ,\qquad
   u-\hat u\in W^{1,2}(\R,H_1)\cap L^2(\R,H_2) .
\end{equation*}
\end{definition}

The following theorem was proved in~\cite[Thm.\,6.11]{Frauenfelder:2025g}.

\begin{theorem}\label{thm:main-pDarboux}
Let $A$ be an almost extendable weak Hessian field on~$U_1$.
Consider two points $u_-,u_+\in U_2$ and a connecting path $u$.
Assume that both asymptotic operators
$A^{u_\mp}$ are isomorphisms as maps $H_1\to H_0$.
Then the operators
\begin{equation*}
\begin{split}
   \D^u=\p_s+A^u\colon
   &W^{1,2}(\R,H_0)\cap L^2(\R,H_1)
   \to L^2(\R,H_0)
\\
   \D^u_2=\p_s+A^u_2\colon 
   &W^{1,2}(\R,H_1)\cap L^2(\R,H_2)
   \to L^2(\R,H_1)
\end{split}
\end{equation*}
are both Fredholm operators of the same Fredholm index.
\end{theorem}

\boldmath
\section{Hamiltonian Hessian field \boldmath$A$}
\label{sec:Hess-Ham}
\unboldmath

\begin{convention}\label{con:loop-spaces}
To simplify the presentation in Section~\ref{sec:Hess-Ham}
we consider only spaces of loops
$\Ll^\times_+\Zfrak=C^\infty(\SS^1,\Zfrak)\setminus\{0\}$ 
in an open subset $0\in\Zfrak\subset\C$, as
opposed to the space $\Ll^\times_-\Zfrak$ of twisted ($-$) loops;
cf. \S\,\ref{sec:twisted-loop-case}.
The use the notation
$$
   \Upsilon=(z,\eta) \in \Ll\Zfrak\times\Ll\C
   =T^*\Ll\Zfrak
   ,\quad
   \Xi=(z,\xi) \in \Ll\Zfrak\times\Ll\C
   =T\Ll\Zfrak
   .
$$
We canonically identify Euclidean $\R^2$
with~$\C\simeq\C^*\simeq(\R^2)^*$.
\end{convention}

\begin{definition}[Analytic setup]\label{def:analytic-setup}
Let $0\in\Zfrak\subset\C$ be an open subset of $\C$.
Consider the Hilbert space triples defined by
\begin{equation*}
\begin{split}
   \left(H_0,H_1,H_2\right)
   :&=\left(L^2(\SS^1,\C^2), W^{1,2}(\SS^1,\C^2), W^{2,2}(\SS^1,\C^2)\right)
\\
   \left(h_0,h_1,h_2\right)
   :&=\left(L^2(\SS^1,\C), W^{1,2}(\SS^1,\C), W^{2,2}(\SS^1,\C)\right)
\end{split}
\end{equation*}
the Hilbert space $h_r:=W^{r,2}(\SS^1,\C)$, $r\in(\frac12,1)$,
the open subsets
$$
   u_1:=\{z\in h_1\mid z\not\equiv 0\wedge
   \forall\tau\in\SS^1\colon z_\tau\in\Zfrak\}
   ,\qquad
   u_2:=u_1\cap W^{2,2}(\SS^1,\C)
   ,
$$
as well as
$
   U_1:=u_1\times h_1\subset H_1
$
and
$
   U_2:=u_2\times h_2\subset H_2
$.
\end{definition}

\boldmath
\subsection[Without magnetic term]
{Kepler case -- $A$ extends}
\label{sec:Hess-Ham-non-mag}
\unboldmath

We first consider the case $\Mm=0$ (no magnetic term) which
corresponds to the Kepler case.

\begin{definition}
The \textbf{\boldmath$\Hh$-perturbed non-local symplectic action}
$$
   \Aa=\Aa_\Hh^\Lambda=i_\Vv \Lambda-\Hh
   \colon 
   T^*\Ll^\times\Zfrak=\Ll^\times\Zfrak\times\Ll\C
   \to\R
$$
takes at a point $\Upsilon=(z,\eta)$
and with $\Hh$ given by~(\ref{eq:Hh}) the value
\begin{equation*}
\begin{split}
   \Aa(z,\eta) :
   &=
   \Lambda_\Upsilon \Upsilon^\prime
   -\Hh(\Upsilon)
\\
   &=
   \INNER{\eta}{z^\prime}
   -\tfrac{\norm{\eta}^2}{8\norm{z}^2} 
   +\tfrac{1}{\norm{z}^2}
\\
   &=
   -\INNER{\eta^\prime}{z}
   -\tfrac{\norm{\eta}^2}{8\norm{z}^2} 
   +\tfrac{1}{\norm{z}^2} .
\end{split}
\end{equation*}
Here the final equality is integration by parts.
The functional $\Aa$ continuously extends to the Sobolev 
$W^{1,2}$-completion $U_1$ by the same formula.
\end{definition}

The \textbf{\boldmath$L^2$-gradient of $\Aa$}
at $\Upsilon=(z,\eta)\in U_1$~is, by~(\ref{eq:grad-Aa}), of the form
\begin{equation}\label{eq:grad-Aa-magn-zero}
   \grad\Aa|_{(z,\eta)}
   =
   \begin{pmatrix}
      -\eta^\prime
      +\frac{\norm{\eta}^2z}{4\norm{z}^4}
      -\frac{2z}{\norm{z}^4}
      \\
      z^\prime-\frac{\eta}{4\norm{z}^2}
   \end{pmatrix}
   =J_0\Upsilon^\prime-\grad\Hh|_\Upsilon
\end{equation}
where
$$
   J_0=\begin{pmatrix}0&-\Id\\\Id&0\end{pmatrix}
   ,\quad
   \Id=\Id_{W^{1,2}(\SS^1,\R)}
   .
$$
Linearizing the gradient at a pair
$\Upsilon=(z,\eta) \in U_1$
in direction of a smooth vector field
$\hat\Upsilon=(\hat z,\hat \eta) \in H_1$ along the loop $\Upsilon$
defines the \textbf{Hessian operator}
\begin{equation}\label{eq:Hess-Aa-magn-zero}
   A^\Upsilon\hat\Upsilon
   :=\left.\tfrac{d}{d\eps}\right|_{\eps=0}\grad\Aa(z_\eps,\eta_\eps)
   =J_0\hat\Upsilon^\prime
   -\left.\tfrac{d}{d\eps}\right|_{\eps=0}\grad\Hh|_{\Upsilon_\eps}
\end{equation}
where
$
   \eps\to z_\eps\not\equiv 0
$
and
$
   \eps\to\eta_\eps
$
are smooth paths in loop space $ U_1$, respectively in
$H_1$, notation $\Upsilon_\eps=(z_\eps,\eta_\eps)$, such that
\begin{equation}\label{eq:par-var}
   \Upsilon_0=(z,\eta)
   ,\qquad
   \left.\tfrac{d}{d\eps}\right|_0 \Upsilon_\eps
   =\hat\Upsilon=(\hat z,\hat\eta).
\end{equation}
The Hessian operator of the non-local Hamiltonian $\Hh$ in~(\ref{eq:Hh}) is
given by
\begin{equation}\label{eq:Hess-Hh}
\begin{split}
  &\left.\tfrac{d}{d\eps}\right|_{0}
  \grad\Hh|_{\Upsilon_\eps}
\\
   &=
   \left.\tfrac{d}{d\eps}\right|_{0}
   \begin{pmatrix}
      \frac{\norm{\eta_\eps}^2z_\eps}{4\norm{z_\eps}^4}
      -\frac{2z_\eps}{\norm{z_\eps}^4}
      \\
      -\frac{\eta_\eps}{4\norm{z_\eps}^2}
   \end{pmatrix}
\\
   &=
   \begin{pmatrix}
      \frac{\inner{\eta}{\hat \eta} z}{2\norm{z}^4}
      +\frac{\norm{\eta}^2\hat z}{4\norm{z}^4}
     -\frac{\norm{\eta}^2\inner{z}{\hat z} z}{\norm{z}^6}
\overbrace{
      -\tfrac{2\hat z}{\norm{z}^4}
      +\tfrac{8\inner{z}{\hat z} z}{\norm{z}^6}
}^{=U^z}
      \\
      -\frac{\hat\eta}{4\norm{z}^2}
      +\frac{\inner{z}{\hat z}\eta}{2\norm{z}^4}
   \end{pmatrix}
\\
   &=
   \tfrac{1}{4\norm{z}^6}
   \begin{pmatrix}
\text{\small$
      2\norm{z}^2z\, \eta^*\hat\eta
      +\norm{z}^2\norm{\eta}^2\hat z
     -4\norm{\eta}^2z\, z^*\hat z
      -8\norm{z}^2\hat z
      +32z\, z^*\hat z
$}
      \\
\text{\small$
      -\norm{z}^4\hat\eta
      +2\norm{z}^2\eta\, z^*\hat z
$}
   \end{pmatrix}
\\
   &=
   \tfrac{1}{4\norm{z}^6}
   \begin{pmatrix}
\text{\small$
   \norm{z}^2\norm{\eta}^2\Id
      -4\norm{\eta}^2z\, z^*
      -8\norm{z}^2\Id
      +32z\, z^*
$}
      &
\text{\small\color{magenta}$
      2\norm{z}^2z\, \eta^*
$}
      \\
\text{\small\color{magenta}$
      2\norm{z}^2\eta\, z^*
$}
      &
\text{\small$
      -\norm{z}^4\Id
$}
   \end{pmatrix}
   \begin{pmatrix}
   \text{\small$\hat z$}\\\text{\small$\hat\eta$}
   \end{pmatrix}
\\
   &=:
   \Hess\,\Hh^{(z,\eta)}
   \begin{pmatrix}
   \hat z\\\hat\eta
   \end{pmatrix}
   .
\end{split}
\end{equation}
Here $z^*\colon \hat z\mapsto \inner{z}{\hat z}=:z^*\hat z$ is
the duality pairing. The Hessian operator takes the form of a
matrix $\Hess\,\Hh^\Upsilon$
called the \textbf{Hessian matrix} of the Hamiltonian~$\Hh$.
The matrix is \textbf{\color{magenta}symmetric}
since $(\eta z^*)^*=z^{**}\eta^*=z\eta^*$.
Hence $\Hess\,\Hh^\Upsilon$ is $H_0$-symmetric.
Along $ U_1\subset W^{1,2}\times W^{1,2}$
the multiplication operator is well defined as a bounded linear map
\begin{equation}\label{eq:Hess-matrix-bd}
   \Hess\,\Hh^\Upsilon\colon H_1\stackrel{\rm bd}{\longrightarrow}H_1
   ,\qquad 
   \forall \Upsilon=(z,\eta)\in U_1
\end{equation}
since the product of two $W^{1,2}$ maps with domain $\SS^1$
lies in $W^{1,2}$, in symbols,
$W^{1,2}\cdot W^{1,2}\subset W^{1,2}$; see~(\ref{eq:Behz-Hol-2}).

\begin{lemma}[compact Hessian matrix]\label{le:Hess-Hh-comp}
As a linear map between the spaces
$$
   \forall \Upsilon\in U_1\colon
   \qquad
   \Hess\,\Hh^\Upsilon\colon
   \quad
   H_1\stackrel{\rm cp}{\longrightarrow} H_0
   ,\quad
   H_2\stackrel{\rm cp}{\longrightarrow} H_1
   ,
$$
the Hessian matrix of $\Hh$ is a compact linear operator.
It is continuous as a map
$$
   [\Upsilon\mapsto \Hess\,\Hh^\Upsilon]
   \in C^0( U_1,\Ll(H_1,H_0))
   \cap C^0(\underline{ U_1},\Ll(H_2,H_1))
   .
$$
\end{lemma}

\begin{proof}
The composition of a bounded and a compact operator is compact.
Consider~(\ref{eq:Hess-matrix-bd}).
On the target side there is the compact embedding
$H_1\INTO H_0$ and on the domain side there is the compact
embedding $H_2\INTO H_1$.

The two continuity assertions are a consequence of the compact
embeddings together with continuity as a map
$$
   [(z,\eta)\mapsto \Hess\,\Hh^{(z,\eta)}]
   \in C^0( U_1,\Ll(H_1))
   .
$$
This continuity follows from the continuity of the
multiplication map $W^{1,2}\times W^{1,2}\to W^{1,2}$;
see~(\ref{eq:Behz-Hol-2}).
This proves Lemma~\ref{le:Hess-Hh-comp}.
\end{proof}

\subsubsection{Weak Hessian field extends}
\label{sec:Ham-Hess-Kepler-extends}

\begin{lemma}[weak Hessian field]\label{le:weak-Hess-nonmag}
The Kepler Hessian operators
\begin{equation}\label{eq:Hess-Ham-nonmag}
   A^\Upsilon
   =J_0\p_\tau-\Hess\,\Hh^\Upsilon
   \colon H_1\to H_0
\end{equation}
one for each $\Upsilon\in U_1$,
determine a weak Hessian field $A$ on $ U_1$.
\end{lemma}

\begin{proof}
There are three steps.

\smallskip
\noindent
\textsc{Step 1 (Spaces).}
It holds that
\begin{equation}\label{eq:Hess-Ham-nonmag-extend}
\begin{split}
   A\in 
   &C^0( U_1,
   \Ll( H_1, H_0)\cap\Ll( H_2, H_1))
\\
   &
{\color{gray}
   \subset 
   C^0( U_1,\Ll( H_1, H_0))
   \cap
   C^0(U_2,\Ll( H_2, H_1)) .
}
\end{split}
\end{equation}

\smallskip
\noindent

The summand $J_0\p_\tau$ is a constant map in
$\Ll( H_1, H_0)\cap \Ll( H_2, H_1)$.
Furthermore, by Lemma~\ref{le:Hess-Hh-comp}
the summand $\Hess\,\Hh$
lies in $C^0( U_1,\Ll( H_1))$.

\smallskip
\noindent
\textsc{Step 2.} The axiom (\texttt{Symmetry}) holds.

\smallskip
\noindent
Both $J_0\p_\tau$ and $\Hess\,\Hh^\Upsilon$
are $ H_0$-symmetric at any point $\Upsilon\in U_1$.

\smallskip
\noindent
\textsc{Step 3.} The axiom (\texttt{Fredholm}) holds.

\smallskip
\noindent
That $J_0\p_\tau$ is Fredholm of index zero as a map
$ H_1\to  H_0$ and as a map $ H_2\to  H_1$ is well
known; see e.g.~\cite[Thm.\,C.1]{Frauenfelder:2025g}.
In either case, by Lemma~\ref{le:Hess-Hh-comp},
the Hessian is a compact perturbation,
but such preserve Fredholm property and index.
This concludes the proof of Lemma~\ref{le:weak-Hess-nonmag}.
\end{proof}

Observe that~(\ref{eq:Hess-Ham-nonmag-extend})
shows that the non-magnetic weak Hessian
field~(\ref{eq:Hess-Ham-nonmag}) is extendable, see
Definition~\ref{def:extends}, which we state as

\begin{proposition}[extends]\label{prop:Theta=0-extends}
The Kepler Hessian field $A$ extends.
\end{proposition}

\boldmath
\subsection{With magnetic term $\Mm$}
\label{sec:Hess-Ham-mag}
\unboldmath

We use $H_k$, $U_k$, $u_k$ as
in Definition~\ref{def:analytic-setup}.
Consider the action functional
$\Aa=i_\Vv\Lambda+\pi^*\Mm-\Hh$
in~(\ref{eq:Aa_Hh}). By~(\ref{eq:grad-Aa}), the $L^2$-gradient at 
$\Upsilon=(z,\eta)\in U_1$ is
\begin{equation*}
\begin{split}
   \grad\Aa|_{(z,\eta)}
   &=
   \begin{pmatrix}
      -\eta^\prime
      +\grad\Mm|_z
      +\frac{\norm{\eta}^2z}{4\norm{z}^4}
      -\frac{2z}{\norm{z}^4}
      \\
      z^\prime-\frac{\eta}{4\norm{z}^2}
   \end{pmatrix}
\\
   &=J_0\Upsilon^\prime
   +\begin{pmatrix}\grad\Mm|_z\\0\end{pmatrix}
   -\grad\Hh|_\Upsilon
\end{split}
\end{equation*}
where $\grad\Mm|_z$ is the sum of the last four terms
in~(\ref{eq:grad-Bb}).
Linearizing the gradient at a pair
$\Upsilon=(z,\eta)\in U_1$ in direction of a vector
field $\hat\Upsilon=(\hat z,\hat \eta)\in H_1$
along the loop $\Upsilon$ yields,
analogously to~(\ref{eq:Hess-Aa-magn-zero}) and~(\ref{eq:Hess-Hh}),
\textbf{Hessian operators}
\begin{equation}\label{eq:Hess-op-mag}
\begin{split}
   A^\Upsilon\hat\Upsilon
   &=A^{(z,\eta)}(\hat z,\hat \eta)
   :
   =\left.\tfrac{d}{d\eps}\right|_{\eps=0}\grad\Aa(z_\eps,\eta_\eps)
\\
   &=J_0\hat\Upsilon^\prime
   +\begin{pmatrix}M^z&0\\0&0\end{pmatrix}
      \begin{pmatrix}\hat z\\\hat\eta\end{pmatrix}
   -\Hess\,\Hh^{(z,\eta)}\begin{pmatrix}\hat z\\\hat\eta\end{pmatrix}
\end{split}
\end{equation}
one operator at every element $\Upsilon=(z,\eta)\in U_1$.
Lemma~\ref{le:Hessians-Lagrange}
provides the Hessian operators $M^z$, one for each $z\in u_1$,
of the magnetic function $\Mm$.
The Hessian matrix of $\Hh$ is given by~(\ref{eq:Hess-Hh}).

\smallskip\noindent
The task at hand is to show that $\Upsilon\mapsto A^\Upsilon$
defines a \emph{weak Hessian field} along~$ U_1$
(Proposition~\ref{prop:Ham-weak-mag})
which \emph{almost extends} (Theorem~\ref{thm:Ham-ae});
cf.~\S\,\ref{sec:weak-Hess-fields}.

\begin{remark}[Magnetic Hessian $M$]\label{rem:T_ij}
Of the three summands of the Hessian~(\ref{eq:Hess-op-mag})
it remains to analyze the second one, namely the map
$$
{\color{gray}
   U_1=u_1\times h_1\supset\;
}
   u_1\ni z\mapsto
   \begin{pmatrix}M^z&0\\0&0\end{pmatrix}
   \colon
   H_1\to H_0
   ,\quad
   H_2\to H_1
   .
$$
Since $H_k=h_k\times h_k$, the magnetic Hessian is actually a linear map
$M^z\colon h_1\to h_0$, $h_2\to h_1$ which is composed of
18 summands
$$
\underbrace{
   0, C_{42}, C_{43}, F_{44}
}_{T_4}
   ,
\underbrace{
   T_{51}, T_{52}, T_{53}, T_{54}
}_{T_5}
   ,
\underbrace{
   T_{61}, T_{62}, T_{63}, T_{64}, T_{65}
}_{T_6}
   ,
\underbrace{
   T_{71}, T_{72}, T_{73}, T_{74}, T_{75}
}_{T_7}
   .
$$
These linear operators are defined and analyzed one-by-one in
Appendix\,\ref{sec:T4}.
\end{remark}

\boldmath
\subsubsection{Weak Hessian field}
\unboldmath

\begin{proposition}[weak Hessian field]\label{prop:Ham-weak-mag}
The Hessian operators $A^\Upsilon$ in~(\ref{eq:Hess-op-mag}),
one operator for every $\Upsilon\in U_1$,
form a weak Hessian field $A$ on $ U_1$
(Definition~\ref{def:weak-Hessian-field}).
\end{proposition}

The proof of the proposition requires two lemmas which we state and
prove thereafter (Lemma~\ref{le:f_44-sFred} and
Lemma~\ref{le:f_65_75-comp}).

\begin{proof}
There are three steps.

\smallskip
\noindent
\textsc{Step 1 (Spaces).}
It holds that
$A\in C^0( U_1,\Ll(H_1,H_0)) \cap C^0(U_2,\Ll(H_2,H_1))$.

\smallskip
\noindent
By the non-magnetic case~(\ref{eq:Hess-Ham-nonmag-extend})
it suffices to show that the magnetic Hessian operators form a map
$z\mapsto M^z\in C^0(u_1,\Ll(h_1,h_0)) \cap C^0(u_2,\Ll(h_2,h_1))$.
This is true by the explicit formulas in~\S\,\ref{sec:T4}.

\smallskip
\noindent
\textsc{Step 2.} The axiom (\texttt{Symmetry}) holds.

\smallskip
\noindent
The Hessian operator $A^\Upsilon$ is the derivative
of the $L^2$-gradient of the functional $\Aa$ in~(\ref{eq:Aa_Hh}).
Hence $A^\Upsilon$ is $H_0$-symmetric by Corollary~\ref{cor:H0-symm}.

\smallskip
\noindent
\textsc{Step 3.} The axiom (\texttt{Fredholm}) holds.

\smallskip
\noindent
On level one and also on level two
there is ${\color{orange}1}$ Fredholm operator of index zero 
$F^{1,\Upsilon}\colon H_1\to H_0$,
$\forall \Upsilon \in U_1$, and 
$F^{1, \Upsilon }_2\colon H_2\to H_1$,
$\forall \Upsilon \in U_2$
(Lemma~\ref{le:f_44-sFred})
plus 17 compact perturbations
$\diag(M^z_{ij},0)$ from the magnetic field
(see \S\,\ref{sec:T4} where to $T_{65}$ and $T_{75}$ we apply, in addition,
Lemma~\ref{le:f_65_75-comp})
plus the compact perturbation $\Hess\, \Hh$
(Lemma~\ref{le:Hess-Hh-comp}).
Since the Fredholm property as well as the index
is stable under compact perturbation the Fredholm axiom follows.
This proves Proposition~\ref{prop:Ham-weak-mag}.
\end{proof}

\subsubsection*{Fredholm perturbation}

The following lemma enters step~3 of the proof of
Proposition~\ref{prop:Ham-weak-mag}.

\begin{lemma}\label{le:f_44-sFred}
Let $(z,\eta)\in U_1$
and $f_{44}^z:=-4 b_{t_z}|_z j_0$, cf.~(\ref{eq:f_44-individual}).
For $r\in[0,1]$ set
\begin{equation*}
   F^{r,z}:=J_0\p_\tau+\diag(r\underbrace{f_{44}^z\p_\tau}_{F_{44}^z},0)
   =
   \underbrace{
   \begin{pmatrix}
      r f_{44}^z & -\Id
      \\
      \Id &0
   \end{pmatrix}
   }_{=: \bbb^{r,z}}
   \p_\tau
   \colon
   \;\;
    H_1\stackrel{F^{r,z}}{\longrightarrow}  H_0
   ,\;\;
    H_2\stackrel{F^{r,z}_2}{\longrightarrow}  H_1
   .
\end{equation*}
Then both operators $F^{r,z}$ and $F^{r,z}_2$  are
Fredholm of index zero for each $r$.
\end{lemma}

\begin{proof}
Pick $z\in u_1$.
For $r=0$ both operators $F^{0,z}=J_0\p_\tau$ and
$F^{0,z}_2=J_0\p_\tau$ do not depend on $z$ and
are well known to be Fredholm of index zero;
see proof of Proposition~\ref{prop:Theta=0-extends}.
Hence, because the semi-Fredholm index is invariant under homotopy,
see e.g.~\cite[\S 18 Cor.\,3]{Muller:2007a},
it suffices to show that $F^{r,z}$ and $F^{r,z}_2$ are semi-Fredholm for every
$r\in[0,1]$. To this end we derive a semi-Fredholm estimate for
the level one operator $F^{r,z}$ and the level two operator $F^{r,z}_2$.

\smallskip
\noindent
\textbf{Step 1 (semi-Fredholm level \boldmath$1$).}
Given $z\in u_1$, there is $c_z>0$ such that
$$
   \norm{\hat\Upsilon}_{ H_1}
   \le c_z
   \left(
   \norm{F^{r,z} \hat\Upsilon}_{ H_0}+\norm{\hat\Upsilon}_{ H_0}
   \right)
$$
for all $\hat\Upsilon=(\hat z,\hat\eta)\in H_1$ and
$r\in[0,1]$.

\begin{proof}
Fix $z\in u_1$.
While $H_0:=L^2(\SS^1,\C^2)$ by Definition~\ref{def:analytic-setup},
let us abbreviate
$$
   L^2:=L^2(\SS^1,\C)
   ,\qquad
   H_0=L^2\times L^2
   .
$$
For $r\in [0,1]$ the operator $\bbb^{r,z}$
is invertible with inverse
\begin{equation*}
\begin{split}
   (\bbb^{r,z})^{-1}
   =
   \begin{pmatrix}
      0& \1
      \\
      -\1 &r f_{44}^z 
   \end{pmatrix}
   \colon L^2\times L^2
   &\to L^2\times L^2 
\\
   (\hat z,\hat\eta)
   &\mapsto\left(\hat \eta,-\hat z+r f_{44}^z\hat\eta\right)
   .
\end{split}
\end{equation*}
To estimate the operator norm of $(\bbb^{r,z})^{-1}$,
let $\hat\Upsilon=(\hat z,\hat \eta)\in L^2\times L^2$ and estimate
\begin{equation*}
\begin{split}
   \norm{(\bbb^{r,z})^{-1}\hat\Upsilon}_{L^2\times L^2}^2
   &=\norm{\hat\eta}_{L^2}^2
   +\norm{-\hat z+r f_{44}^z\hat\eta}_{L^2}^2
\\
   &\le \norm{\hat\eta}_{L^2}^2
   +2 \norm{\hat z}_{L^2}^2
   +2r^2 \norm{f_{44}^z\hat\eta}_{L^2}^2
\\
   &\le \norm{\hat\eta}_{L^2}^2
   +2 \norm{\hat z}_{L^2}^2
   +2 (4 \beta_{\SS^1}^z)^2 \norm{\hat\eta}_{L^2}^2
   \;\;\text{\small , by~(\ref{eq:f_44-individual})}
\\
   &\le
   \max\{2,1+2(4\beta^z_{\SS^1})^2\}
   \norm{\hat\Upsilon}_{L^2\times L^2}^2 .
\end{split}
\end{equation*}
We used that $r\le 1$.
Hence the operator norm of the inverse is bounded 
by
$$
   \norm{(\bbb^{r,z})^{-1}}_{\Ll(L^2\times L^2)}
   \le \mu_z
   ,\qquad
   \mu_z:=\sqrt{\max\{2,1+2(4\beta^z_{\SS^1})^2\}}
   .
$$
To get the semi-Fredholm estimate let
$\hat\Upsilon=(\hat z,\hat \eta)\in W^{1,2}\times W^{1,2}= H_1$, then
\begin{equation*}
\begin{split}
   \norm{\hat \Upsilon}_{W^{1,2}\times W^{1,2}}^2
   &=\norm{\hat \Upsilon}_{L^{2}\times L^{2}}^2
   +\norm{\hat \Upsilon^\prime}_{L^{2}\times L^{2}}^2
\\
   &=\norm{\hat \Upsilon}_{L^{2}\times L^{2}}^2
   +\norm{(\bbb^{r,z})^{-1}\bbb^{r,z}\hat \Upsilon^\prime}_{L^{2}\times L^{2}}^2
\\
   &=\norm{\hat \Upsilon}_{L^{2}\times L^{2}}^2
   +\norm{(\bbb^{r,z})^{-1} F^{r,z}\hat \Upsilon}_{L^{2}\times L^{2}}^2
\\
   &\le\norm{\hat \Upsilon}_{L^{2}\times L^{2}}^2
   +\norm{(\bbb^{r,z})^{-1}}_{\Ll(L^2\times L^2)}^2
   \norm{F^{r,z}\hat \Upsilon}_{L^{2}\times L^{2}}^2
\\
   &\le\max\{1,\mu_z^2\}
   \left(
   \norm{\hat \Upsilon}_{L^{2}\times L^{2}}^2
   +\norm{F^{r,z}\hat \Upsilon}_{L^{2}\times L^{2}}^2
   \right)
\\
   &\le\max\{1,\mu_z^2\}
   \left(
   \norm{\hat \Upsilon}_{L^{2}\times L^{2}}
   +\norm{F^{r,z}\hat \Upsilon}_{L^{2}\times L^{2}}
   \right)^2
   .
\end{split}
\end{equation*}
This proves Step~1 for $c_z=\max\{1,\mu_z\}$.
\end{proof}

\smallskip
\noindent
\textbf{Step 2 (semi-Fredholm level \boldmath$2$).}
Given $z\in u_1$, there is $d_z>0$ such that 
$$
   \norm{\hat\Upsilon}_{ H_2}
   \le d_z
   \left(
   \norm{F^{r,z}_2 \hat\Upsilon}_{ H_1}+\norm{\hat\Upsilon}_{ H_1}
   \right)
$$
for all $\hat\Upsilon=(\hat z,\hat\eta)\in H_2$ and
$r\in[0,1]$.

\begin{proof}
Fix $z\in u_1$. 
While $H_1:=W^{1,2}(\SS^1,\C^2)$ by Definition~\ref{def:analytic-setup},
we abbreviate
$$
   W^{1,2}:=W^{1,2}(\SS^1,\C)
   ,\qquad
   H_1=W^{1,2} \times W^{1,2}
   .
$$
For $r\in [0,1]$ the operator $\bbb^{r,z}$
is invertible with inverse
\begin{equation*}
\begin{split}
   (\bbb^{r,z})^{-1}
   =
   \begin{pmatrix}
      0& \1
      \\
      -\1 &r f_{44}^z 
   \end{pmatrix}
   \colon W^{1,2}\times W^{1,2}
   &\to W^{1,2}\times W^{1,2}
\\
   (\hat z,\hat\eta)
   &\mapsto\left(\hat \eta,-\hat z+r f_{44}^z\hat\eta\right)
   .
\end{split}
\end{equation*}
To estimate the operator norm of $(\bbb^{r,z})^{-1}$, let
$\hat\Upsilon=(\hat z,\hat \eta)\in W^{1,2}\times W^{1,2}$, then
\begin{equation*}
\begin{split}
   \norm{(\bbb^{r,z})^{-1}\hat\Upsilon}_{W^{1,2}\times W^{1,2}}^2
   &=\norm{\hat\eta}_{W^{1,2}}^2
   +\norm{-\hat z+r f_{44}^z\hat\eta}_{W^{1,2}}^2
\\
   &\le \norm{\hat\eta}_{W^{1,2}}^2
   +2 \norm{\hat z}_{W^{1,2}}^2
   +2r^2 \norm{f_{44}^z\hat\eta}_{W^{1,2}}^2
\\
   &\le \norm{\hat\eta}_{W^{1,2}}^2
   +2 \norm{\hat z}_{W^{1,2}}^2
   +8\gamma_{\norm{z}_{1,2}}^{\;2} \norm{\hat\eta}_{W^{1,2}}^2
   \;\text{\small, by~(\ref{eq:f_44})}
\\
   &\le
   \max\{2,1+8\gamma_{\norm{z}_{1,2}}^{\;2}\}
   \norm{\hat\Upsilon}_{W^{1,2}\times W^{1,2}}^2
   .
\end{split}
\end{equation*}
We used that $r\le 1$.
Hence the operator norm of the inverse is bounded by
$$
   \norm{(\bbb^{r,z})^{-1}}_{\Ll(W^{1,2}\times W^{1,2})}
   \le \kappa_z
   ,\qquad
   \kappa_z:=\sqrt{\max\{2,1+8\gamma_{\norm{z}_{1,2}}^{\;2}\}}
   .
$$
To get the semi-Fredholm estimate let
$\hat\Upsilon=(\hat z,\hat \eta)\in W^{1,2}\times W^{1,2}= H_1$,
then
\begin{equation*}
\begin{split}
   \norm{\hat \Upsilon}_{W^{2,2}\times W^{2,2}}^2
   &=\norm{\hat \Upsilon}_{W^{1,2}\times W^{1,2}}^2
   +\norm{\hat \Upsilon^\prime}_{W^{1,2}\times W^{1,2}}^2
\\
   &=\norm{\hat \Upsilon}_{W^{1,2}\times W^{1,2}}^2
   +\norm{(\bbb^{r,z})^{-1}\bbb^{r,z}\hat \Upsilon^\prime}_{W^{1,2}\times W^{1,2}}^2
\\
   &=\norm{\hat \Upsilon}_{W^{1,2}\times W^{1,2}}^2
   +\norm{(\bbb^{r,z})^{-1} F^{r,z}\hat \Upsilon}_{W^{1,2}\times W^{1,2}}^2
\\
   &=\norm{\hat \Upsilon}_{W^{1,2}\times W^{1,2}}^2
   +\norm{(\bbb^{r,z})^{-1}}_{\Ll(W^{1,2}\times W^{1,2})}^2
   \norm{F^{r,z}\hat \Upsilon}_{W^{1,2}\times W^{1,2}}^2
\\
   &\le\max\{1,\kappa_z^2\}
\text{\small$
   \left(
   \norm{\hat \Upsilon}_{W^{1,2}\times W^{1,2}}^2
   +\norm{F^{r,z}\hat \Upsilon}_{W^{1,2}\times W^{1,2}}^2
   \right)
$}
\\
   &\le\max\{1,\kappa_z^2\}
   \left(
   \norm{\hat \Upsilon}_{W^{1,2}\times W^{1,2}}
   +\norm{F^{r,z}\hat \Upsilon}_{W^{1,2}\times W^{1,2}}
   \right)^2
   .
\end{split}
\end{equation*}
This proves Step~2 for $d_z=\max\{1,\kappa_z\}$.
\end{proof}
Since the inclusion maps $ H_1\INTO  H_0$ and
$ H_2\INTO  H_1$ are compact,
Step~1 and Step~2 provide semi-Fredholm estimates,
see e.g.~\cite[Le.\,A.1.1]{mcduff:2004a}.
This completes the proof of Lemma~\ref{le:f_44-sFred}.
\end{proof}

\subsubsection*{Compact perturbations}

The following lemma enters step~3 of the proof of
Proposition~\ref{prop:Ham-weak-mag}.

\begin{lemma}\label{le:f_65_75-comp}
Pick $z\in u_1$.
Both linear operators $T_{65}^z$ and $T_{75}^z$
in~(\ref{eq:T_65}) and~(\ref{eq:T_75}), respectively,
are compact as operators $h_1\to h_0$ and $h_2\to h_1$.
\end{lemma}

\begin{proof}
On level two compactness follows from the fact that
$T_{65}^z,T_{75}^z\in\Ll(h_1)$, see
estimates~(\ref{eq:F_65-est}) and~(\ref{eq:F_75-est}),
and that the embedding $h_2\INTO h_1$ is compact.

\smallskip\noindent
Level one is harder.
For that we show that the operators extend
to bounded linear operators $h_r\to h_0$ for every $r\in(\tfrac12,1)$.
Then the compactness follows from the compact inclusion
$h_1\INTO h_r$ whenever $r\in(\tfrac12,1)$.
Fix $r\in(\tfrac12,1)$.

\smallskip\noindent
\textbf{The operator \boldmath$T_{65}^z$.}
We apply to~(\ref{eq:T_65})
integration by parts to obtain
\begin{equation*}
\begin{split}
   F^z_{65}\xi:
   &=\tfrac{2z_\tau}{\norm{z}_2^2}\int_{\tau}^1
   \inner{\dot{\aaa}_{t_{z}(\sigma)}|_{z_{\sigma}}}{{\color{orange}\xi^\prime_{\sigma}}}_0
   \; d\sigma
\\
   &=\tfrac{2z_\tau}{\norm{z}_2^2}
   \left(
   \inner{\dot{\aaa}_0|_{z_0}}{\xi_0}_0
   -\inner{\dot{\aaa}_{t_{z}(\tau)}|_{z_{\tau}}}{\xi_{\tau}}_0
   -\int_{\tau}^1
   \inner{(\dot{\aaa}_{t_{z}(\sigma)}|_{z_{\sigma}})^\prime}{\xi_\sigma}_0
   \; d\sigma
   \right)
\\
   &=\tfrac{2z_\tau}{\norm{z}_2^2}
   \left(
   \inner{\dot{\aaa}_0|_{z_0}}{\xi_0}_0
   -\inner{\dot{\aaa}_{t_{z}(\tau)}|_{z_{\tau}}}{\xi_{\tau}}_0
   \right)
\\
   &\quad
   -\tfrac{2z_\tau}{\norm{z}_2^2}
   \int_{\tau}^1
   \INNER{\ddot{\aaa}_{t_{z}(\sigma)}|_{z_{\sigma}} 
   \tfrac{\abs{z_\sigma}^2}{\norm{z}_2^2}
   +d \dot{\aaa}_{t_{z}(\sigma)}|_{z_{\sigma}} z^\prime_\sigma}{\xi_\sigma}_0
   \; d\sigma
\end{split}
\end{equation*}
where we used $t_z(1)=1$
and $1$-periodicity of $\dot\aaa$, $z$, and $\xi$.
With the $\alpha$-constants in~(\ref{eq:a-infty})
for $\xi\in W^{r,2}$ we estimate
(cf.~(\ref{eq:C_61}) for the last summand)
\begin{equation*}
\begin{split}
   \norm{F^z_{65}\xi}_{L^2}
   &\le\tfrac{2 \dot\alpha_{\SS^1}^z}{\norm{z}_2}\norm{\xi}_\infty
   +
   \tfrac{2 \dot\alpha_{\SS^1}^z}{\norm{z}_2}\norm{\xi}_\infty
   +
   \tfrac{2}{\norm{z}_2}\tau \ddot\alpha_{\SS^1}^z
   \tfrac{\norm{z}_\infty^2}{\norm{z}_2^2}\norm{\xi}_\infty
   +\tfrac{2}{\norm{z}_2}d\dot\alpha_{\SS^1}^z
   \norm{z^\prime}_2\norm{\xi}_2
\\
   &\le
   const(\norm{z}_{1,2},r)\norm{\xi}_{r,2}
   .
\end{split}
\end{equation*}
This proves that $F^z_{65}\colon h_r\to h_0$ is bounded.

\smallskip\noindent
\textbf{The operator \boldmath$T_{75}^z$.}
We apply to~(\ref{eq:T_75}) integration by parts to get
\begin{equation*}
\begin{split}
   T_{75}^z\xi:
   &=-\tfrac{2 z_\tau}{\norm{z}^4_2} \int_0^1
   {\textstyle
   \int_0^s\Abs{z_{\sigma}}^2 d\sigma
   }
   \cdot\inner{\dot{\aaa}_{t_{z}(s)}|_{z_{s}}}{{\color{orange}\xi^\prime_s}}_0\;
   ds
\\
   &=-\tfrac{2 z_\tau}{\norm{z}^4_2}\left(
   \norm{z}_2^2\inner{\dot{\aaa}_0|_{z_0}}{\xi_0}_0
   -0
   \right)
   +\tfrac{2 z_\tau}{\norm{z}^4_2}
   \int_0^1\Abs{z_s}^2\inner{\dot{\aaa}_{t_{z}(s)}|_{z_{s}}}{\xi_s}\,
   ds
   \\
   &\quad
   +\tfrac{2 z_\tau}{\norm{z}^4_2}
   \int_0^1
   {\textstyle
   \int_0^s\Abs{z_{\sigma}}^2 d\sigma
   }
   \INNER{\ddot{\aaa}_{t_{z}(s)}|_{z_{s}} 
   \tfrac{\abs{z_s}^2}{\norm{z}_2^2}
   +d \dot{\aaa}_{t_{z}(s)}|_{z_{s}} z^\prime_s}{\xi_s}_0
   ds
\end{split}
\end{equation*}
where we used $t_z(1)=1$
and $1$-periodicity of $\dot\aaa$, $z$, and $\xi$.
(We ignore that the first summand of $T_{75}^z$ actually cancels the
first summand of $T_{65}^z$.)
With the $\alpha$-constants in~(\ref{eq:a-infty})
for $\xi\in W^{r,2}$ we estimate
(cf.~(\ref{eq:C_61}) for the last summand)
\begin{equation*}
\begin{split}
   \norm{F^z_{75}\xi}_{L^2}
   &\le\tfrac{2 \dot\alpha_{\SS^1}^z}{\norm{z}_2}\norm{\xi}_\infty
   +\tfrac{2}{\norm{z}_2^3}\norm{z}_\infty^2\dot\alpha_{\SS^1}^z\norm{\xi}_\infty
   +\tfrac{2 \ddot\alpha_{\SS^1}^z}{\norm{z}_2^3}\norm{z}_\infty^2\norm{\xi}_\infty
   +
   \tfrac{2 d\dot\alpha_{\SS^1}^z}{\norm{z}_2}\norm{z^\prime}_2\norm{\xi}_2
\\
   &\le const(\norm{z}_{1,2},r)\norm{\xi}_{r,2}
   .
\end{split}
\end{equation*}
This proves that $F^z_{75}\colon h_r\to h_0$ is bounded.
This proves Lemma~\ref{le:f_65_75-comp}.
\end{proof}

\subsubsection{Almost extendable}
\label{sec:ae}

\begin{theorem}\label{thm:Ham-ae}
The magnetic weak Hessian field $A$ in~(\ref{eq:Hess-op-mag})
is almost extendable.
The pair $(F,C)$ defined for $(z,\eta)$ in $U_1$,
respectively $U_2$, by
\begin{equation*}
\begin{split}
   C^{(z,\eta)}:&=\diag(C_{42}^z,0)+\diag(C_{43}^z,0)
\\
   F^{(z,\eta)}:&=
   A^{(z,\eta)}-C^{(z,\eta)}
{\color{gray}\;
   =F^{1,z}+{\textstyle\sum_1^{15}}\diag(T_{ij},0)
   -\Hess\,\Hh^{(z,\eta)}
}
\end{split}
\end{equation*}
is a decomposition (Definition~\ref{def:alm.ext}).
\end{theorem}

The proof of the theorem requires two propositions which we state and
prove thereafter (Proposition~\ref{prop:C_43} and
Proposition~\ref{prop:C_42}.

\begin{proof}
The proof has four steps~1, 2, (\texttt{C}), and (\texttt{F}).
In Definition~\ref{def:analytic-setup} we defined spaces
$H_k:=W^{k,2}(\SS^1,\C^2)\supset U_k$ and
$h_k:=W^{k,2}(\SS^1,\C)\supset u_k$.

\smallskip\noindent
\textbf{Step~1.}
There is $r\in[0,1)$ such that $(z,\eta)\mapsto C^{(z,\eta)}$ lies in
$C^0(U_1,\Ll(\underline{H_r},H_0))$, where $H_r=W^{r,2}(\SS^1,\C^2)$,
and in $C^0(U_2,\Ll(\underline{H_1},H_1))$; see~(\ref{eq:hyp-C}).

\begin{proof}
The maps of the diagonal form $(z,\eta)\mapsto \diag(C_{42}^z,0)$
have the required continuity property, since the diagonal block
maps lie in the spaces
\begin{equation}\label{eq:C42-C43-maps}
   z\mapsto C_{42}^z
   ,\;\;
   z\mapsto C_{43}^z
   \;\;
   \in
   C^0(u_1,\Ll(\underline{h_r},h_0))\cap
   C^0(u_2,\Ll(\underline{h_1},h_1))
\end{equation}
by~(\ref{eq:C_42}) and~(\ref{eq:C_43}), respectively.
%
\end{proof}

\smallskip\noindent
\textbf{Step~2.}
We need to show that the map $F\colon (z,\eta)\mapsto F^{(z,\eta)}$
is element of the space $C^0(U_1,\Ll(H_1,H_0)\cap\Ll(H_2,H_1))$;
see~(\ref{eq:hyp-F}).

\begin{proof}
The diagonal block $z\mapsto f_{44}^z\p_\tau$ of
$z\mapsto F^{1,z}$, see Lemma~\ref{le:f_44-sFred}.
is a map
$$
   f_{44}^z\p_\tau\colon
   \quad
   h_1\stackrel{\p_\tau}{\longrightarrow} h_0
   \stackrel{\text{(\ref{eq:f_44-individual})}}{\longrightarrow} h_0
   ,\quad
   h_2\stackrel{\p_\tau}{\longrightarrow} h_1
   \stackrel{\text{(\ref{eq:f_44})}}{\longrightarrow} h_1
   ,
$$
which depends continuously on $z \in u_1$,
see line two in~(\ref{eq:f_44}).
Since the other three blocks of $F^{1,z}$ are constant,
this shows that $z\mapsto F^{1,z}$ has the required continuity property.
The Hessian operator $\Hess\,\Hh$ has the required continuity property
by Lemma~\ref{le:Hess-Hh-comp}.
The fifteen operators of the diagonal form
\begin{align*}
   z\mapsto \diag(T_{ij},0)&\colon
   &H_1&\to H_0
   &H_2&\to H_1&
\\
   z\mapsto T_{ij}^z&\colon
   &h_1&\to h_0
   &h_2&\to h_1&
\end{align*}
have the required continuity property
since so do the diagonal blocks
$T_{ij}$ as shown for each of them in \S\,\ref{sec:T4}.
This proves Step~2.
\end{proof}

\medskip\noindent
\textbf{Step~(\texttt{C}).} The map $(z,\eta) \mapsto C^{(z,\eta)}=\diag(C_{42}^z,0)+\diag(C_{43}^z,0)$ satisfies the scale
Lipschitz estimate~(\ref{eq:(C)}).

\begin{proof}
Due to the specific diagonal form it suffices to consider the
upper diagonal operators~(\ref{eq:C42-C43-maps}).
Proposition~\ref{prop:C_43} and Proposition~\ref{prop:C_42}
prove Step~(\texttt{C}).
\end{proof}

\smallskip\noindent
\textbf{Step~(\texttt{F}).} $\forall \Upsilon=(z,\eta)\in U_1\colon$
$F^{1,z}_2+{\textstyle\sum_1^{15}}\diag(T_{ij},0)-\Hess\,\Hh^\Upsilon
\colon H_2\to H_1$ is Fredholm of index zero.

\begin{proof}
Lemma~\ref{le:f_44-sFred} asserts that
$F^{1,z}_2\colon H_2\to H_1$ is Fredholm of index zero.
Lemma~\ref{le:Hess-Hh-comp} asserts that
$\Hess\,\Hh^\Upsilon\colon H_2\to H_1$ is compact.
The fifteen operators of diagonal form
$\diag(T_{ij},0)\colon H_2\to H_1$
are compact since so are the diagonal blocks
$T_{ij}\colon h_2\to h_1$ as shown in \S\,\ref{sec:T4}.
Since Fredholm property and index are stable under compact
perturbation this proves Step~(\texttt{F}).
\end{proof}

This concludes the proof of Theorem~\ref{thm:Ham-ae}.
\end{proof}

\subsubsection{Scale Lipschitz estimates}
\label{sec:sc-Lipschitz-Ham}

The following proposition is part of Step~(\texttt{C}) of the proof of
Theorem~\ref{thm:Ham-ae}.

\begin{proposition}[$C_{43}$ is scale Lipschitz]\label{prop:C_43}
For $z\mapsto C^z_{43}$ in~(\ref{eq:C42-C43-maps}) the scale Lipschitz
estimate~\texttt{(C)} holds true, namely~(\ref{eq:(C)}) with lower
case spaces $h_k$ and~$u_1$.
\end{proposition}

\begin{proof}
Let $z\in u_1$ and $\xi\in h_1$. We consider the map
in~(\ref{eq:C_43}), namely
$$
   C^z\xi
   :=C^z_{43}\xi
   :=(db_{t_z(\tau)}|_{z_\tau}\xi_\tau)\, j_0 z_\tau^\prime
   .
$$
Since $z\colon\SS^1\to\Zfrak$ is continuous its image is
compact. So the image of $z$ admits an open neighborhood $\Vfrak_z$
in $\Zfrak\subset\R^2$ of compact closure~$\bar \Vfrak_z$.
By compactness $\bar \Vfrak_z$ is contained in a ball of some radius
$\rho=\rho(z)$, in symbols
\begin{equation}\label{eq:rho_z}
   \bar \Vfrak_z
   \subset \Bfrak_\rho
   :=\{\nu\in \C\colon \abs{\nu}<\rho\}
   .
\end{equation}
We define by
\begin{equation}\label{eq:V_z}
   \mathrm{v}_z:=\{v\in W^{1,2}(\SS^1,\Vfrak_z)\colon
   \text{$\norm{v}_2>\tfrac12\norm{z}_2>0$
   and
   $\norm{v}_{1,2}<2\norm{z}_{1,2}$}
   \}
\end{equation}
an $h_1$-open neighborhood of $z$ in~$u_1$
uniformly bounded away from~$0$.
Hence
$$
   v\in \mathrm{v}_z\quad\Rightarrow\quad
   \norm{v}_2\le\norm{v}_\infty\le \rho.
$$

\smallskip
Pick $v,w\in \mathrm{v}_z\CAP h_2$.
By definition of the $W^{1,2}$ norm, calculating the derivative $\p_\tau$,
the triangle inequality and since
$t_v^\prime(\tau)={\color{cyan}\Abs{v_\tau}^2/\norm{v}_2^2}$
by~(\ref{eq:class-time-tau-deriv-C}), we get
\begin{equation*}
\begin{split}
   &\norm{(C^v-C^w)\xi}_{1,2}
\\
   &\le\norm{
   (db_{t_v(\tau)}|_{v_\tau}\xi_\tau)\, j_0 v_\tau^\prime
   -
   (db_{t_w(\tau)}|_{w_\tau}\xi_\tau)\, j_0 w_\tau^\prime
   }_{2}
   \\
   &\quad
   +\norm{
   \p_\tau\left(
   (db_{t_v(\tau)}|_{v_\tau}\xi_\tau)\, j_0 v_\tau^\prime
   \right)
   -
   \p_\tau\left(
   (db_{t_w(\tau)}|_{w_\tau}\xi_\tau)\, j_0 w_\tau^\prime
   \right)
   }_{2}
\\
   &\le 
   \norm{
   (db_{t_v(\tau)}|_{v_\tau}\xi_\tau)\, j_0 v_\tau^\prime
   -
   (db_{t_w(\tau)}|_{w_\tau}\xi_\tau)\, j_0 w_\tau^\prime
   }_{2}
   \\
   &\quad
   +\norm{
   (d\dot b_{t_v(\tau)}|_{v_\tau}\xi_\tau)
   {\color{cyan}\tfrac{\abs{v_\tau}^2}{\norm{v}_2^2}}\, j_0 v_\tau^\prime
   -
   (d\dot b_{t_w(\tau)}|_{w_\tau}\xi_\tau)
   {\color{cyan}\tfrac{\abs{w_\tau}^2}{\norm{w}_2^2}}\, j_0 w_\tau^\prime
   }_{2}
   \\
   &\quad
   +\norm{
   (d^2b_{t_v(\tau)}|_{v_\tau}\xi_\tau v_\tau^\prime)\, j_0 v_\tau^\prime
   -
   (d^2b_{t_w(\tau)}|_{w_\tau}\xi_\tau w_\tau^\prime)\, j_0 w_\tau^\prime
   }_{2}
   \\
   &\quad
   +\norm{
   (db_{t_v(\tau)}|_{v_\tau}\xi_\tau^\prime)\, j_0 v_\tau^\prime
   -
   (db_{t_w(\tau)}|_{w_\tau}\xi_\tau^\prime)\, j_0 w_\tau^\prime
   }_{2}
   \\
   &\quad
   +\norm{
   (db_{t_v(\tau)}|_{v_\tau}\xi_\tau)\, j_0 v_\tau^{\prime\prime}
   -
   (db_{t_w(\tau)}|_{w_\tau}\xi_\tau)\, j_0 w_\tau^{\prime\prime}
   }_{2}
   .
\end{split}
\end{equation*}
We need to estimate the $L^2$-norms of five differences, 
notation \textbf{D1-D5}.

\medskip
\noindent
\textbf{Difference D1.}
Add twice zero to write D1 as a sum $D_{11}\xi+D_{12}\xi+D_{13}\xi$
\begin{equation*}
\begin{split}
   &(db_{t_v(\tau)}|_{v_\tau}\xi_\tau)\, j_0 v_\tau^\prime
   -
   (db_{t_w(\tau)}|_{w_\tau}\xi_\tau)\, j_0 w_\tau^\prime
\\
   &=
   (db_{t_v(\tau)}|_{v_\tau}\xi_\tau)\, j_0 v_\tau^\prime
   -
   (db_{t_w(\tau)}|_{v_\tau}\xi_\tau)\, j_0 v_\tau^\prime
   \\
   &\quad
   +
   (db_{t_w(\tau)}|_{v_\tau}\xi_\tau)\, j_0 v_\tau^\prime
   -
   (db_{t_w(\tau)}|_{w_\tau}\xi_\tau)\, j_0 v_\tau^\prime
   \\
   &\quad
   +
   (db_{t_w(\tau)}|_{w_\tau}\xi_\tau)\, j_0 v_\tau^\prime
   -
   (db_{t_w(\tau)}|_{w_\tau}\xi_\tau)\, j_0 w_\tau^\prime
\end{split}
\end{equation*}
pointwise at $\tau\in\SS^1$.

\smallskip
\noindent
\textbf{\boldmath$D_{11}$.}
By compactness of $\bar \Vfrak_z$ the maximum of the
operator norm is finite
\begin{equation}\label{eq:d-dot-beta}
   d\dot\beta_{\bar\Vfrak_z}
   :=\max_{\SS^1\times\bar\Vfrak_z}\norm{d\dot b}_{\Ll(\R^2,\R)}
   <\infty .
\end{equation}
By Taylor's theorem, then using~(\ref{eq:t_z-diff}) on
$\Abs{t_v(\tau)-t_w(\tau)}$, we 
estimate
\begin{equation}\label{eq:db-t_v-t_w}
\begin{split}
   \norm{db_{t_v(\tau)}|_{v_\tau}-db_{t_w(\tau)}|_{v_\tau}}_{\Ll(\R^2,\R)}
   &\le d\dot\beta_{\bar\Vfrak_z}
   \Abs{t_v(\tau)-t_w(\tau)}
\\
   &\le 2d\dot\beta_{\bar\Vfrak_z}
   \tfrac{\norm{v}_2+\norm{w}_2}{\norm{v}_2^2}   
   \norm{v-w}_2
\\
   &\le 2 d\dot\beta_{\bar\Vfrak_z} \tfrac{2^2 2\rho}{\norm{z}_2^2}
   \norm{v-w}_2
\\
   &= \tfrac{16\rho (d\dot\beta_{\bar\Vfrak_z})}{\norm{z}_2^2}
   \norm{v-w}_2
   .
\end{split}
\end{equation}
The last step is by~(\ref {eq:rho_z}) and~(\ref{eq:V_z}).
Use this operator norm estimate to obtain
\begin{equation*}
\begin{split}
   \norm{D_{11}\xi}_2
   &=
   \norm{
   \left((db_{t_v(\tau)}|_{v_\tau}-db_{t_w(\tau)}|_{v_\tau})\xi_\tau\right)
   j_0 v_\tau^\prime
   }_2
\\
   &\le
   \tfrac{16\rho (d\dot\beta_{\bar\Vfrak_z})}{\norm{z}_2^2}
   \norm{v-w}_2\norm{\xi}_\infty\norm{v^\prime}_2
\\
   &\le \tfrac{16\rho (d\dot\beta_{\bar\Vfrak_z})}{\norm{z}_2^2}
{\color{blue}\;
   \abs{v}_{h_1}\cdot\abs{v-w}_{h_0}
\;}
   \norm{\xi}_{1,2}
   .
\end{split}
\end{equation*}

\smallskip
\noindent
\textbf{\boldmath$D_{12}$.}
By compactness of $\bar \Vfrak_z$ the maximum of the
operator norm is finite
\begin{equation}\label{eq:d^2beta-frak}
   d^2\beta_{\bar\Vfrak_z}
   :=\max_{\SS^1\times\bar\Vfrak_z}\norm{d^2b}_{\Ll(\R^2\times\R^2,\R)}
   <\infty .
\end{equation}
By Taylor's theorem we estimate
\begin{equation*}
\begin{split}
   \norm{D_{12}\xi}_2
   &=
   \norm{
   \left((db_{t_w(\tau)}|_{v_\tau}-db_{t_w(\tau)}|_{w_\tau})\xi_\tau\right)
   j_0 v_\tau^\prime
   }_2
\\
   &\le
   d^2\beta_{\bar\Vfrak_z}\norm{v-w}_\infty\norm{\xi}_\infty\norm{v^\prime}_2
\\
   &\le
   d^2\beta_{\bar\Vfrak_z}
{\color{red}\;
   \abs{v}_{h_1}\cdot\abs{v-w}_{h_1}
\;}
   \norm{\xi}_{1,2}
   .
\end{split}
\end{equation*}

\smallskip
\noindent
\textbf{\boldmath$D_{13}$.}
With the constant $d\beta_{\SS^1}^w$ defined
by~(\ref{eq:b-infty}) we estimate
\begin{equation*}
\begin{split}
   \norm{D_{13}\xi}_2
   &=\norm{
   (db_{t_w(\tau)}|_{w_\tau}\xi_\tau)\, j_0 \left(v_\tau^\prime-w_\tau^\prime\right)
   }_2
\\
   &\le d\beta_{\SS^1}^w\norm{\xi}_\infty\norm{v^\prime-w^\prime}_2
\\
   &\le d\beta_{\SS^1}^w
{\color{blue}\;
   \abs{v-w}_{h_1  }
\;}
   \norm{\xi}_{1,2}
\end{split}
\end{equation*}
where the constant $d\beta_{\SS^1}^w$ is defined by~(\ref{eq:b-infty}).

\medskip
\noindent
\textbf{Difference D2}
Add four zeroes to write D2 as a sum
$D_{21}\xi+\dots+D_{25}\xi$
\begin{equation*}
\begin{split}
   &(d\dot b_{t_v(\tau)}|_{v_\tau}\xi_\tau)
   \tfrac{\abs{v_\tau}^2}{\norm{v}_2^2}\, j_0 v_\tau^\prime
   -
   (d\dot b_{t_w(\tau)}|_{w_\tau}\xi_\tau)
   \tfrac{\abs{w_\tau}^2}{\norm{w}_2^2}\, j_0 w_\tau^\prime
\\
   &=
   \left(
   d\dot b_{t_v(\tau)}|_{v_\tau}-d\dot b_{t_w(\tau)}|_{v_\tau}
   \right)
   \xi_\tau \tfrac{\abs{v_\tau}^2}{\norm{v}_2^2}\, j_0 v_\tau^\prime
\\
   &\quad+(d\dot b_{t_w(\tau)}|_{v_\tau}\xi_\tau)
   \tfrac{\abs{v_\tau}^2}{\norm{v}_2^2}\, j_0 v_\tau^\prime
   -
   (d\dot b_{t_w(\tau)}|_{w_\tau}\xi_\tau)
   \tfrac{\abs{v_\tau}^2}{\norm{v}_2^2}\, j_0 v_\tau^\prime
\\
   &\quad+
   (d\dot b_{t_w(\tau)}|_{w_\tau}\xi_\tau)
   \tfrac{\abs{v_\tau}^2}{\norm{v}_2^2}\, j_0 v_\tau^\prime
   -
   (d\dot b_{t_w(\tau)}|_{w_\tau}\xi_\tau)
   \tfrac{\abs{w_\tau}^2}{\norm{v}_2^2}\, j_0 v_\tau^\prime
\\
   &\quad+
   (d\dot b_{t_w(\tau)}|_{w_\tau}\xi_\tau)
   \tfrac{\abs{w_\tau}^2}{\norm{v}_2^2}\, j_0 v_\tau^\prime
   -
   (d\dot b_{t_w(\tau)}|_{w_\tau}\xi_\tau)
   \tfrac{\abs{w_\tau}^2}{\norm{w}_2^2}\, j_0 v_\tau^\prime
\\
   &\quad+
   (d\dot b_{t_w(\tau)}|_{w_\tau}\xi_\tau)
   \tfrac{\abs{w_\tau}^2}{\norm{w}_2^2}\, j_0 v_\tau^\prime
   -(d\dot b_{t_w(\tau)}|_{w_\tau}\xi_\tau)
   \tfrac{\abs{w_\tau}^2}{\norm{w}_2^2}\, j_0 w_\tau^\prime
   .
\end{split}
\end{equation*}

\noindent
\smallskip
\textbf{\boldmath$D_{21}$.}
By compactness of $\bar \Vfrak_z$ the maximum of the
operator norm is finite
\begin{equation}\label{eq:d-ddot-beta-43}
   d\ddot\beta_{\bar\Vfrak_z}
   :=\max_{\SS^1\times\bar\Vfrak_z}\norm{d\ddot b}_{\Ll(\R^2,\R)}
   <\infty .
\end{equation}
As in~(\ref{eq:db-t_v-t_w}) we obtain an estimate, uniform in
$\tau\in\SS^1$, for the operator norm
\begin{equation}\label{eq:d-dot-b-diff-43}
\begin{split}
   \norm{d\dot b_{t_v(\tau)}|_{v_\tau}-d\dot b_{t_w(\tau)}|_{v_\tau}}_{\Ll(\R^2,\R)}
   &\le \tfrac{16\rho (d\ddot\beta_{\bar\Vfrak_z})}{\norm{z}_2^2}
   \norm{v-w}_2
   .
\end{split}
\end{equation}
Use this operator norm estimate, as well
as~(\ref {eq:rho_z}) and~(\ref{eq:V_z}), to estimate
\begin{equation*}
\begin{split}
   \norm{D_{21}\xi}_2
   &=\norm{
   (d\dot b_{t_v}|_{v}-d\dot b_{t_w}|_{v})
   \xi \tfrac{\abs{v}^2}{\norm{v}_2^2}\, j_0 v^\prime
   }_2
\\
   &\le \tfrac{16\rho (d\ddot\beta_{\bar\Vfrak_z})}{\norm{z}_2^2}
   \norm{v-w}_2 \norm{\xi}_\infty
   \tfrac{\norm{v}_\infty^2}{\norm{v}_2^2}{\color{cyan}\norm{v}_{1,2}}
\\
   &\le \tfrac{16\rho (d\ddot\beta_{\bar\Vfrak_z})}{\norm{z}_2^2}
   \tfrac{2^2\rho^2}{\norm{z}_2^2} {\color{cyan} 2\norm{z}_{1,2}}
   \norm{v-w}_2
   \norm{\xi}_{1,2}
\\
   &=c_{21}^z
{\color{blue}\;
   \abs{v-w}_{h_0}
\;}
   \norm{\xi}_{1,2}
   \quad
{\color{gray}
   ,\;
   c_{21}^z:=\tfrac{128\rho^3\norm{z}_{1,2}}{\norm{z}_2^4}d\ddot\beta_{\bar\Vfrak_z}
   .
}
\end{split}
\end{equation*}

\noindent
\smallskip
\textbf{\boldmath$D_{22}$.}
By compactness of $\bar \Vfrak_z$ the maximum of the
operator norm is finite
\begin{equation}\label{eq:d^2-dot-beta-43}
   d^2\dot\beta_{\bar\Vfrak_z}
   :=\max_{\SS^1\times\bar\Vfrak_z}\norm{d^2\dot b}_{\Ll(\R^2\times\R^2,\R)}
   <\infty .
\end{equation}
Similarly as for $D_{21}$ we estimate
\begin{equation*}
\begin{split}
   \norm{D_{22}\xi}_2
   &=\norm{
   (d\dot b_{t_w(\tau)}|_{v_\tau}-d\dot b_{t_w(\tau)}|_{w_\tau})
   \xi_\tau
   \tfrac{\abs{v_\tau}^2}{\norm{v}_2^2}\, j_0 v_\tau^\prime
   }_2
\\
   &\le
   d^2\dot\beta_{\bar\Vfrak_z}\norm{v-w}_\infty
   \norm{\xi}_\infty \tfrac{\norm{v}_\infty^2}{\norm{v}_2^2}
   {\color{cyan}\norm{v}_{1,2}}
\\
   &\le    d^2\dot\beta_{\bar\Vfrak_z}\norm{v-w}_{1,2}
   \norm{\xi}_{1,2} \tfrac{2^2\rho^2}{\norm{z}_2^2}
   {\color{cyan} 2\norm{z}_{1,2}}
\\
   &\le c_{22}^z
{\color{blue}\;
   \abs{v-w}_{h_1}
\;}
   \norm{\xi}_{1,2}
   \quad
{\color{gray}
   ,\;
   c_{22}^z:=\tfrac{8\rho^2\norm{z}_{1,2}}{\norm{z}_2^2}d^2\dot\beta_{\bar\Vfrak_z}
   .
}
\end{split}
\end{equation*}

\smallskip
\noindent
\textbf{\boldmath$D_{23}$.}
With $d\dot\beta_{\SS^1}^w$ from~(\ref{eq:b-infty}) we estimate
\begin{equation*}
\begin{split}
   \norm{D_{23}\xi}_2
   &=
   \norm{
   (d\dot b_{t_w(\tau)}|_{w_\tau}\xi_\tau)
   \tfrac{\abs{v_\tau}^2}{\norm{v}_2^2}\, j_0 v_\tau^\prime
   -
   (d\dot b_{t_w(\tau)}|_{w_\tau}\xi_\tau)
   \tfrac{\abs{w_\tau}^2}{\norm{v}_2^2}\, j_0 v_\tau^\prime
   }_2
\\
   &=
   \norm{
   (d\dot b_{t_w(\tau)}|_{w_\tau}\xi_\tau)
   \tfrac{\abs{v_\tau}^2
{\color{gray}\;
   -\langle v_\tau,w_\tau\rangle_0+\langle v_\tau,w_\tau\rangle_0
}
   -\abs{w_\tau}^2}{\norm{v}_2^2}\, j_0 v_\tau^\prime
   }_2
\\
   &=
   \norm{
   (d\dot b_{t_w(\tau)}|_{w_\tau}\xi_\tau)
   \tfrac{\inner{v_\tau}{v_\tau-w_\tau}
   +\inner{v_\tau-w_\tau}{w_\tau}}{\norm{v}^2_2}
   \, j_0 v_\tau^\prime
   }_2
\\
   &\le d\dot\beta_{\SS^1}^w\norm{\xi}_\infty\norm{v^\prime}_\infty
   \tfrac{\norm{v}_\infty\norm{v-w}_2+\norm{w}_\infty\norm{v-w}_2}
   {\norm{v}^2_2}
\\
   &\le
   \tfrac{2^2 2 \norm{z}_{1,2}}{\norm{z}_2^2} d\dot\beta_{\SS^1}^w
{\color{red}\;
   \abs{v}_{h_2}\cdot{\color{blue}\,\abs{v-w}_{h_0}}
\;}
   \norm{\xi}_{1,2}
   .
\end{split}
\end{equation*}
In the last step we used~(\ref{eq:V_z}) in
combination with $\norm{\cdot}_\infty\le \norm{\cdot}_{1,2}$.

\smallskip
\noindent
\textbf{\boldmath$D_{24}$.}
With $d\dot\beta_{\SS^1}^w$ from~(\ref{eq:b-infty}) we estimate
\begin{equation*}
\begin{split}
   \norm{D_{24}\xi}_2
   &=
   \norm{
   (d\dot b_{t_w(\tau)}|_{w_\tau}\xi_\tau)
   \tfrac{\abs{w_\tau}^2}{\norm{v}_2^2}\, j_0 v_\tau^\prime
   -
   (d\dot b_{t_w(\tau)}|_{w_\tau}\xi_\tau)
   \tfrac{\abs{w_\tau}^2}{\norm{w}_2^2}\, j_0 v_\tau^\prime
   }_2
\\
   &=
   \norm{
   (d\dot b_{t_w(\tau)}|_{w_\tau}\xi_\tau) \abs{w_\tau}^2
   \tfrac{\norm{w}_2^2
{\color{gray}\;
   -\langle v,w\rangle+\langle v,w\rangle
}
   -\norm{v}_2^2}{\norm{v}_2^2\norm{w}_2^2}
   j_0 v_\tau^\prime
   }_2
\\
   &=
   \norm{
   (d\dot b_{t_w(\tau)}|_{w_\tau}\xi_\tau) {\color{cyan}\,\abs{w_\tau}^2}
   \tfrac{\langle w-v,w\rangle+\langle v,w-v\rangle}{\norm{v}_2^2\norm{w}_2^2}
   j_0 v_\tau^\prime
   }_2
\\
   &\le d\dot\beta_{\SS^1}^w\norm{\xi}_\infty
   \tfrac{{\color{cyan}\rho^2} 2^2 2^2}{\norm{z}_2^2\norm{z}_2^2}
   \norm{v-w}_2\left(\norm{w}_2+\norm{v}_2\right)
   \norm{v^\prime}_2
\\
   &\le c_{24}^z
{\color{blue}\;
   \abs{v-w}_{h_0}
\;}
   \norm{\xi}_{1,2}
   \quad
{\color{gray}
   ,\;
   c_{24}^z:=\tfrac{64\rho^3 \norm{z}_{1,2}}{\norm{z}_2^4}d\dot\beta_{\SS^1}^w
   .
}
\end{split}
\end{equation*}
The last step is by~(\ref {eq:rho_z}) and~(\ref{eq:V_z}).

\smallskip
\noindent
\textbf{\boldmath$D_{25}$.}
With $d\dot\beta_{\SS^1}^w$ from~(\ref{eq:b-infty}), as well
as~(\ref {eq:rho_z}) and~(\ref{eq:V_z}), we estimate
\begin{equation*}
\begin{split}
   \norm{D_{25}\xi}_2
   &=
   \norm{
   (d\dot b_{t_w(\tau)}|_{w_\tau}\xi_\tau)
   \tfrac{\abs{w_\tau}^2}{\norm{w}^2}\, j_0 v_\tau^\prime
   -(d\dot b_{t_w(\tau)}|_{w_\tau}\xi_\tau)
   \tfrac{\abs{w_\tau}^2}{\norm{w}^2}\, j_0 w_\tau^\prime
   }_2
\\
   &=   \norm{
   (d\dot b_{t_w(\tau)}|_{w_\tau}\xi_\tau)
   \tfrac{\abs{w_\tau}^2}{\norm{w}^2}\, j_0 
   \left(v_\tau^\prime-w_\tau^\prime\right)
   }_2
\\
   &\le d\dot\beta_{\SS^1}^w\norm{\xi}_\infty
   \tfrac{2^2\rho^2}{\norm{z}_2^2}
   \norm{v-w}_{1,2}
\\
   &\le c_{25}^z
{\color{blue}\;
   \abs{v-w}_{h_1}
\;}
   \norm{\xi}_{1,2}
   \quad
{\color{gray}
   ,\;
   c_{25}^z:=\tfrac{4\rho^2}{\norm{z}_2^2}d\dot\beta_{\bar\Vfrak_z}
   .
}
\end{split}
\end{equation*}

\medskip
\noindent
\textbf{Difference D3}
Add three zeroes to write D3 as a sum
$D_{31}\xi+\dots+D_{34}\xi$
\begin{equation*}
\begin{split}
   &(d^2b_{t_v(\tau)}|_{v_\tau}\xi_\tau v_\tau^\prime)\, j_0 v_\tau^\prime
   -
   (d^2b_{t_w(\tau)}|_{w_\tau}\xi_\tau w_\tau^\prime)\, j_0w_\tau^\prime
\\
   &=
   (d^2b_{t_v(\tau)}|_{v_\tau}\xi_\tau v_\tau^\prime)\, j_0 v_\tau^\prime
   -
   (d^2b_{t_w(\tau)}|_{v_\tau}\xi_\tau v_\tau^\prime)\, j_0 v_\tau^\prime
   \\
   &\quad
  +(d^2b_{t_w(\tau)}|_{v_\tau}\xi_\tau v_\tau^\prime)\, j_0 v_\tau^\prime
   -
   (d^2b_{t_w(\tau)}|_{w_\tau}\xi_\tau v_\tau^\prime)\, j_0 v_\tau^\prime
   \\
   &\quad
  +(d^2b_{t_w(\tau)}|_{w_\tau}\xi_\tau v_\tau^\prime)\, j_0 v_\tau^\prime
   -
   (d^2b_{t_w(\tau)}|_{w_\tau}\xi_\tau w_\tau^\prime)\, j_0 v_\tau^\prime
   \\
   &\quad
  +(d^2b_{t_w(\tau)}|_{w_\tau}\xi_\tau w_\tau^\prime)\, j_0 v_\tau^\prime
   -
   (d^2b_{t_w(\tau)}|_{w_\tau}\xi_\tau w_\tau^\prime)\, j_0w_\tau^\prime
   .
\end{split}
\end{equation*}

\smallskip
\noindent
\textbf{\boldmath\color{red}$D_{31}$.}
By compactness of $\bar \Vfrak_z$ the maximum of the
operator norm is finite
$$
   d^2\ddot\beta_{\bar\Vfrak_z}
   :=\max_{\SS^1\times\bar\Vfrak_z}\norm{d^2\ddot b}_{\Ll(\R^2,\R)}
   <\infty .
$$
As in~(\ref{eq:db-t_v-t_w}) we obtain an estimate, uniform in
$\tau\in\SS^1$, for the operator norm
\begin{equation*}
\begin{split}
   \norm{d^2\dot b_{t_v(\tau)}|_{v_\tau}-d^2\dot b_{t_w(\tau)}|_{v_\tau}}_{\Ll(\R^2,\R)}
   &\le \tfrac{16\rho (d^2\ddot\beta_{\bar\Vfrak_z})}{\norm{z}_2^2}
   \norm{v-w}_2
   .
\end{split}
\end{equation*}
Use this operator norm estimate to obtain
\begin{equation*}
\begin{split}
   \norm{D_{31}\xi}_2
   &=\norm{
   \left((d^2b_{t_v(\tau)}|_{v_\tau}-d^2b_{t_w(\tau)}|_{v_\tau})
   (\xi_\tau v_\tau^\prime)\right) j_0 v_\tau^\prime
   }_2
\\
   &\le \tfrac{16\rho (d^2\ddot\beta_{\bar\Vfrak_z})}{\norm{z}_2^2}
   \norm{v-w}_2
   \norm{\xi}_\infty\norm{v^\prime}_2{\color{red}\norm{v^\prime}_\infty}
\\
   &\le c_{31}^z
{\color{blue}\;
   {\color{red}\abs{v}_{h_2}}\cdot\abs{v-w}_{h_0}
\;}
   \norm{\xi}_{1,2}
   \quad
{\color{gray}
   ,\;
   c_{31}^z:=
   \tfrac{32\rho\norm{z}_{1,2}}{\norm{z}_2^2}d^2\ddot\beta_{\bar\Vfrak_z}
   .
}
\end{split}
\end{equation*}

\smallskip
\noindent
\textbf{\boldmath$\color{red}D_{32}$.}
By compactness of $\bar \Vfrak_z$ the maximum of the
operator norm is finite
\begin{equation*}
   d^3\beta_{\bar\Vfrak_z}
   :=\max_{\SS^1\times\bar\Vfrak_z}\norm{d^3b}_{\Ll(\R^2\times\R^2\times\R^2,\R)}
   <\infty .
\end{equation*}
By Taylor's theorem we estimate
\begin{equation*}
\begin{split}
   \norm{D_{32}\xi}_2
   &=\norm{
   \left((d^2b_{t_w(\tau)}|_{v_\tau}-d^2b_{t_w(\tau)}|_{w_\tau})
   (\xi_\tau v_\tau^\prime)\right) j_0 v_\tau^\prime
   }_2
\\
   &\le d^3\beta_{\bar\Vfrak_z}
   \norm{v-w}_\infty\norm{\xi}_\infty
   \norm{v^\prime}_\infty\norm{v^\prime}_2
\\
   &\le 2\norm{z}_{1,2}d^2\beta_{\bar\Vfrak_z}
{\color{red}\;
   \abs{v}_{h_2}\cdot\abs{v-w}_{h_1}
\;}
   \norm{\xi}_{1,2}
   .
\end{split}
\end{equation*}

\smallskip
\noindent
\textbf{\boldmath\color{red}$D_{33}$.}
With the constant $d^2\beta_{\SS^1}^w$ from~(\ref{eq:b-infty}) we estimate
\begin{equation*}
\begin{split}
   \norm{D_{33}\xi}_2
   &=\norm{
   (d^2b_{t_w(\tau)}|_{w_\tau}\xi_\tau
   \left(v_\tau^\prime-w_\tau^\prime)\right)  j_0 v_\tau^\prime
   }_2
\\
   &\le d^2\beta_{\SS^1}^w
   \norm{\xi}_\infty\norm{v^\prime-w^\prime}_\infty\norm{v^\prime}_2
\\
  &\le 2\norm{z}_{1,2}d^2\beta_{\SS^1}^w
{\color{red}\;
   \abs{v-w}_{h_2}
\;}
   \norm{\xi}_{1,2}
   .
\end{split}
\end{equation*}

\smallskip
\noindent
\textbf{\boldmath\color{red}$D_{34}$.}
With the constant $d^2\beta_{\SS^1}^w$ from~(\ref{eq:b-infty}) we estimate
\begin{equation*}
\begin{split}
   \norm{D_{34}\xi}_2
   &=\norm{
   (d^2b_{t_w(\tau)}|_{w_\tau}\xi_\tau w_\tau^\prime)\, j_0 
   \left(v_\tau^\prime-w_\tau^\prime\right)
   }_2
\\
   &\le d^2\beta_{\SS^1}^w
   \norm{\xi}_\infty\norm{w^\prime}_2\norm{v^\prime-w^\prime}_\infty
\\
  &\le 2\norm{z}_{1,2}d^2\beta_{\SS^1}^w
{\color{red}\;
   \abs{v-w}_{h_2}
\;}
   \norm{\xi}_{1,2}
   .
\end{split}
\end{equation*}

\medskip
\noindent
\textbf{Difference D4}
Add twice zero to write D4 as a sum
$D_{41}\xi+D_{42}\xi+ D_{43}\xi$
\begin{equation*}
\begin{split}
   &(db_{t_v(\tau)}|_{v_\tau}\xi_\tau^\prime)\, j_0 v_\tau^\prime
   -
   (db_{t_w(\tau)}|_{w_\tau}\xi_\tau^\prime)\, j_0 w_\tau^\prime
\\
   &=
   (db_{t_v(\tau)}|_{v_\tau}\xi_\tau^\prime)\, j_0 v_\tau^{\prime}
   -
   (db_{t_w(\tau)}|_{v_\tau}\xi_\tau^\prime)\, j_0 v_\tau^{\prime}
   \\
   &\quad
  +(db_{t_w(\tau)}|_{v_\tau}\xi_\tau^\prime)\, j_0 v_\tau^{\prime}
   -
   (db_{t_w(\tau)}|_{w_\tau}\xi_\tau^\prime)\, j_0 v_\tau^{\prime}
   \\
   &\quad
  +(db_{t_w(\tau)}|_{w_\tau}\xi_\tau^\prime)\, j_0 v_\tau^{\prime}
   -
   (db_{t_w(\tau)}|_{w_\tau}\xi_\tau^\prime)\, j_0 w_\tau^{\prime}
\end{split}
\end{equation*}

\smallskip
\noindent
\textbf{\boldmath$D_{41}$.}
We use~(\ref{eq:db-t_v-t_w}) to estimate
\begin{equation*}
\begin{split}
   \norm{D_{41}\xi}_2
   &=\norm{
   \left(
   (db_{t_v(\tau)}|_{v_\tau}-db_{t_w(\tau)}|_{v_\tau})\xi_\tau^{\prime}
   \right)
   j_0 v_\tau^{\prime}
   }_2
\\
   &\le \tfrac{16\rho (d\dot\beta_{\bar\Vfrak_z})}{\norm{z}_2^2}
   \norm{v-w}_2
   \norm{\xi}_\infty\norm{v^\prime}_2
\\
   &\le c_{41}^z
{\color{blue}\;
   \abs{v-w}_{h_0}
\;}
   \norm{\xi}_{1,2}
   \quad
{\color{gray}
   ,\;
   c_{41}^z:=
   \tfrac{32\rho\norm{z}_{1,2}}{\norm{z}_2^2}d\dot\beta_{\bar\Vfrak_z}
   .
}
\end{split}
\end{equation*}

\smallskip
\noindent
\textbf{\boldmath\color{red}$D_{42}$.}
By Taylor's theorem with constant $d^2\beta_{\bar\Vfrak_z}$
in~(\ref{eq:d^2beta-frak}) we estimate
\begin{equation*}
\begin{split}
   \norm{D_{42}\xi}_2
   &=\norm{
   \left(
   (db_{t_w(\tau)}|_{v_\tau}-db_{t_w(\tau)}|_{w_\tau})\xi_\tau^{\prime}
   \right)
   j_0 v_\tau^{\prime}
   }_2
\\
   &\le d^2\beta_{\bar\Vfrak_z}
   \norm{v-w}_\infty\norm{\xi^\prime}_2\norm{v^\prime}_\infty
\\
   &\le d^2\beta_{\bar\Vfrak_z}
{\color{red}\;
   \abs{v}_{h_2}\cdot\abs{v-w}_{h_1}
\;}
   \norm{\xi}_{1,2}
   .
\end{split}
\end{equation*}

\smallskip
\noindent
\textbf{\boldmath\color{red}$D_{43}$.}
With the constant $d\beta_{\SS^1}^w$ from~(\ref{eq:b-infty}) we estimate
\begin{equation*}
\begin{split}
   \norm{D_{43}\xi}_2
   &=\norm{
   (db_{t_w(\tau)}|_{w_\tau}\xi_\tau^{\prime})\, j_0 
   \left(v_\tau^{\prime}-w_\tau^{\prime}\right)
   }_2
\\
   &\le d\beta_{\SS^1}^w\norm{\xi^\prime}_2\norm{v^{\prime}-w^{\prime}}_\infty
\\
   &\le d\beta_{\SS^1}^w
{\color{red}\;
   \abs{v-w}_{h_2}
\;}
   \norm{\xi}_{1,2}
   .
\end{split}
\end{equation*}

\medskip
\noindent
\textbf{Difference D5}
Add twice zero to write D5 as a sum
$D_{51}\xi+D_{52}\xi+ D_{53}\xi$
\begin{equation*}
\begin{split}
   &(db_{t_v(\tau)}|_{v_\tau}\xi_\tau)\, j_0 v_\tau^{\prime\prime}
   -
   (db_{t_w(\tau)}|_{w_\tau}\xi_\tau)\, j_0 w_\tau^{\prime\prime}
\\
   &=
   (db_{t_v(\tau)}|_{v_\tau}\xi_\tau)\, j_0 v_\tau^{\prime\prime}
   -
   (db_{t_w(\tau)}|_{v_\tau}\xi_\tau)\, j_0 v_\tau^{\prime\prime}
   \\
   &\quad
  +
   (db_{t_w(\tau)}|_{v_\tau}\xi_\tau)\, j_0 v_\tau^{\prime\prime}
   -
   (db_{t_w(\tau)}|_{w_\tau}\xi_\tau)\, j_0 v_\tau^{\prime\prime}
   \\
   &\quad
  +
   (db_{t_w(\tau)}|_{w_\tau}\xi_\tau)\, j_0 v_\tau^{\prime\prime}
   -
   (db_{t_w(\tau)}|_{w_\tau}\xi_\tau)\, j_0 w_\tau^{\prime\prime}
\end{split}
\end{equation*}

\smallskip
\noindent
\textbf{\boldmath$D_{51}$.}
We use~(\ref{eq:db-t_v-t_w}) to estimate
\begin{equation*}
\begin{split}
   \norm{D_{51}\xi}_2
   &=\norm{
   \left(
   (db_{t_v(\tau)}|_{v_\tau}-db_{t_w(\tau)}|_{v_\tau})\xi_\tau
   \right)
   j_0 v_\tau^{\prime\prime}
   }_2
\\
   &\le \tfrac{16\rho (d\dot\beta_{\bar\Vfrak_z})}{\norm{z}_2^2}
   \norm{v-w}_2
   \norm{\xi}_\infty\norm{v^{\prime\prime}}_2
\\
   &\le c_{51}^z
{\color{blue}\;
   {\color{red}\abs{v}_{h_2}}\cdot
   \abs{v-w}_{h_0}
\;}
   \norm{\xi}_{1,2}
   \quad
{\color{gray}
   ,\;
   c_{51}^z:=
   \tfrac{16\rho}{\norm{z}_2^2}d\dot\beta_{\bar\Vfrak_z}
   .
}
\end{split}
\end{equation*}

\smallskip
\noindent
\textbf{\boldmath\color{red}$D_{52}$.}
By Taylor's theorem with constant $d^2\beta_{\bar\Vfrak_z}$
in~(\ref{eq:d^2beta-frak}) we estimate
\begin{equation*}
\begin{split}
   \norm{D_{52}\xi}_2
   &=\norm{
   \left(
   (db_{t_w(\tau)}|_{v_\tau}-db_{t_w(\tau)}|_{w_\tau})\xi_\tau
   \right)
   j_0 v_\tau^{\prime\prime}
   }_2
\\
   &\le d^2\beta_{\bar\Vfrak_z}
   \norm{v-w}_\infty\norm{\xi}_\infty\norm{v^{\prime\prime}}_2
\\
   &\le d^2\beta_{\bar\Vfrak_z}
{\color{red}\;
   \abs{v}_{h_2}\cdot\abs{v-w}_{h_1}
\;}
   \norm{\xi}_{1,2}
   .
\end{split}
\end{equation*}

\smallskip
\noindent
\textbf{\boldmath\color{red}$D_{53}$.}
With the constant $d\beta_{\SS^1}^w$ from~(\ref{eq:b-infty}) we estimate
\begin{equation*}
\begin{split}
   \norm{D_{53}\xi}_2
   &=\norm{
   (db_{t_w(\tau)}|_{w_\tau}\xi_\tau)\, j_0 
   \left(v_\tau^{\prime\prime}-w_\tau^{\prime\prime}\right)
   }_2
\\
   &\le d\beta_{\SS^1}^w\norm{\xi}_\infty\norm{v^{\prime\prime}-w^{\prime\prime}}_2
\\
   &\le d\beta_{\SS^1}^w
{\color{red}\;
   \abs{v-w}_{h_2}
\;}
   \norm{\xi}_{1,2}
   .
\end{split}
\end{equation*}

\medskip
Summarizing we have shown that there exists a constant $\kappa=\kappa(z)$
such that
\begin{equation*}
     \norm{C^v-C^w}_{\Ll(h_1)}
     \le\kappa\Bigl(\abs{v-w}_{h_2}
     +\abs{v}_{h_2}\cdot\abs{v-w}_{h_1}\Bigr).
\end{equation*}
whenever $v,w\in \mathrm{v}_z$.
Interchanging the roles of $v$ and $w$ we get
\begin{equation*}
     \norm{C^v-C^w}_{\Ll(h_1)}
     \le\kappa\Bigl(\abs{v-w}_{h_2}
     +\abs{w}_{h_2}\cdot\abs{v-w}_{h_1}\Bigr).
\end{equation*}
whenever $v,w\in \mathrm{v}_z$.
These two inequalities imply that
\begin{equation*}
     \norm{C^v-C^w}_{\Ll(h_1)}
     \le\kappa\Bigl(\abs{v-w}_{h_2}
     +\min\{\abs{v}_{h_2}, \abs{w}_{h_2}\}\cdot\abs{v-w}_{h_1}\Bigr).
\end{equation*}
whenever $v,w\in \mathrm{v}_z$.
This proves Proposition~\ref{prop:C_43}.
\end{proof}

The following proposition is part of Step~(\texttt{C}) of the proof of
Theorem~\ref{thm:Ham-ae}.

\begin{proposition}[$C_{42}$ is scale Lipschitz]\label{prop:C_42}
For $z\mapsto C^z_{42}$ in~(\ref{eq:C42-C43-maps}) the scale Lipschitz
estimate~\texttt{(C)} holds true, namely~(\ref{eq:(C)}) with lower
case spaces $h_k$ and~$u_1$.
\end{proposition}

\begin{proof}
Let $z\in u_1$ and $\xi\in h_1$. We consider the map
in~(\ref{eq:C_42}), namely
$$
   C^z\xi
   :=C^z_{42}\xi
   :=\dot b_{t_z(\tau)}|_{z_\tau}(dt_z\xi)_\tau\, j_0 z_\tau^\prime
   .
$$
As in Proposition~\ref{prop:C_43} let the ball
$\bar \Vfrak_z$ be defined by~(\ref{eq:rho_z}) and
$\mathrm{v}_z$ by~(\ref{eq:V_z}).

\smallskip
Pick $v,w\in \mathrm{v}_z\cap h_2$.
By definition of the $W^{1,2}$ norm, calculating the derivative $\p_\tau$,
the triangle inequality and since
$t_v^\prime(\tau)={\color{cyan}\Abs{v_\tau}^2/\norm{v}_2^2}$
by~(\ref{eq:class-time-tau-deriv-C}), we get
\begin{equation*}
\begin{split}
   &\norm{(C^v-C^w)\xi}_{1,2}
\\
   &\le\norm{
   \dot b_{t_v(\tau)}|_{v_\tau}(dt_v\xi)_\tau\, j_0 v_\tau^\prime
   -
   \dot b_{t_w(\tau)}|_{w_\tau}(dt_w\xi)_\tau\, j_0 w_\tau^\prime
   }_{2}
   \\
   &\quad
   +\norm{
   \p_\tau\left(
   \dot b_{t_v(\tau)}|_{v_\tau}(dt_v\xi)_\tau\, j_0 v_\tau^\prime
   \right)
   -
   \p_\tau\left(
   \dot b_{t_w(\tau)}|_{w_\tau}(dt_w\xi)_\tau\, j_0 w_\tau^\prime
   \right)
   }_{2}
\\
   &\le 
   \norm{
   \dot b_{t_v(\tau)}|_{v_\tau}(dt_v\xi)_\tau\, j_0 v_\tau^\prime
   -
   \dot b_{t_w(\tau)}|_{w_\tau}(dt_w\xi)_\tau\, j_0 w_\tau^\prime
   }_{2}
   \\
   &\quad
   +\norm{
   \ddot b_{t_v(\tau)}|_{v_\tau}
   {\color{cyan}\tfrac{\abs{v_\tau}^2}{\norm{v}_2^2}}
   (dt_v\xi)_\tau\, j_0 v_\tau^\prime
   -
   \ddot b_{t_w(\tau)}|_{w_\tau}
   {\color{cyan}\tfrac{\abs{w_\tau}^2}{\norm{w}_2^2}}
   (dt_w\xi)_\tau\, j_0 w_\tau^\prime
   }_{2}
   \\
   &\quad
   +\norm{
   \left(d\dot b_{t_v(\tau)}|_{v_\tau} v_\tau^\prime\right)
   (dt_v\xi)_\tau\, j_0 v_\tau^\prime
   -
   \left(d\dot b_{t_w(\tau)}|_{w_\tau} w_\tau^\prime\right)
   (dt_w\xi)_\tau\, j_0 w_\tau^\prime
   }_{2}
   \\
   &\quad
   +\norm{
   \dot b_{t_v(\tau)}|_{v_\tau}(dt_v\xi)^\prime_\tau\, j_0 v_\tau^\prime
   -
   \dot b_{t_w(\tau)}|_{w_\tau}(dt_w\xi)^\prime_\tau\, j_0 w_\tau^\prime
   }_{2}
   \\
   &\quad
   +\norm{
   \dot b_{t_v(\tau)}|_{v_\tau}(dt_v\xi)_\tau\, j_0 v_\tau^{\prime\prime}
   -
   \dot b_{t_w(\tau)}|_{w_\tau}(dt_w\xi)_\tau\, j_0 w_\tau^{\prime\prime}
   }_{2}
   .
\end{split}
\end{equation*}
We need to estimate the $L^2$-norms of five differences, 
notation \textbf{D1-D5}.

\medskip
\noindent
\textbf{Difference D1.}
Add three times zero to write D1 as a sum $D_{11}\xi+\dots+D_{14}\xi$
\begin{equation*}
\begin{split}
   &
   \dot b_{t_v(\tau)}|_{v_\tau}(dt_v\xi)_\tau\, j_0 v_\tau^\prime
   -
   \dot b_{t_w(\tau)}|_{w_\tau}(dt_w\xi)_\tau\, j_0 w_\tau^\prime
\\
   &=
   \left(
   \dot b_{t_v(\tau)}|_{v_\tau}-\dot b_{t_w(\tau)}|_{v_\tau}
   \right)
   (dt_v\xi)_\tau\, j_0 v_\tau^\prime
   \\
   &\quad
   +
   \left(
   \dot b_{t_w(\tau)}|_{v_\tau} -\dot b_{t_w(\tau)}|_{w_\tau}
   \right)
   (dt_v\xi)_\tau\, j_0 v_\tau^\prime
   \\
   &\quad
   +
   \dot b_{t_w(\tau)}|_{w_\tau}
   \left(
   (dt_v\xi)_\tau-(dt_w\xi)_\tau
   \right)
   j_0 v_\tau^\prime
   \\
   &\quad
   +
   \dot b_{t_w(\tau)}|_{w_\tau}(dt_w\xi)_\tau\, j_0
   \left(v_\tau^\prime-w_\tau^\prime\right)
\end{split}
\end{equation*}
pointwise at $\tau\in\SS^1$.

\smallskip
\noindent
\textbf{\boldmath$D_{11}$.}
By compactness of $\bar \Vfrak_z$ the maximum of the
absolute value is finite
$$
   \ddot\beta_{\bar\Vfrak_z}
   :=\max_{\SS^1\times\bar\Vfrak_z}\abs{\ddot b}
   <\infty .
$$
As in~(\ref{eq:db-t_v-t_w}) we obtain an estimate, uniform in
$\tau\in\SS^1$, for the difference
\begin{equation}\label{eq:dot-b-diff}
\begin{split}
   \Abs{\dot b_{t_v(\tau)}|_{v_\tau}-\dot b_{t_w(\tau)}|_{v_\tau}}
   &\le \tfrac{16\rho \ddot\beta_{\bar\Vfrak_z}}{\norm{z}_2^2}
   \norm{v-w}_2
   .
\end{split}
\end{equation}
Use this difference estimate, then~(\ref{eq:dt|_z-2-infty})
followed by~(\ref{eq:V_z}), to obtain
\begin{equation}\label{eq:43-D11}
\begin{split}
   \norm{D_{11}\xi}_2
   &=\norm{
   \left(
   \dot b_{t_v(\tau)}|_{v_\tau}-\dot b_{t_w(\tau)}|_{v_\tau}
   \right)
   (dt_v\xi)_\tau\, j_0 v_\tau^\prime
   }_2
\\
   &\le \tfrac{16\rho \ddot\beta_{\bar\Vfrak_z}}{\norm{z}_2^2}
   \norm{v-w}_2
   \norm{dt_v\xi}_\infty\norm{v^\prime}_2
\\
   &\le \tfrac{16\rho \ddot\beta_{\bar\Vfrak_z}}{\norm{z}_2^2}
   \norm{v-w}_2
   \tfrac{4}{\norm{v}_2}\norm{\xi}_2\norm{v}_{1,2}
\\
   &\le \tfrac{128\rho \ddot\beta_{\bar\Vfrak_z}}{\norm{z}_2^3}
{\color{blue}\;
   \abs{v}_{h_1}
   \cdot\abs{v-w}_{h_0}
\;}
   \norm{\xi}_{1,2}
   .
\end{split}
\end{equation}

\smallskip
\noindent
\textbf{\boldmath$D_{12}$.}
By Taylor's theorem and the constant $d\dot\beta_{\bar\Vfrak_z}$
in~(\ref{eq:d-dot-beta}) we estimate
\begin{equation}\label{eq:43-D12}
\begin{split}
   \norm{D_{12}\xi}_2
   &=\norm{
   \left(
   \dot b_{t_w(\tau)}|_{v_\tau}-\dot b_{t_w(\tau)}|_{w_\tau}
   \right)
   (dt_v\xi)_\tau\, j_0 v_\tau^\prime
   }_2
\\
   &\le d\dot\beta_{\bar\Vfrak_z}\norm{v-w}_\infty
   \norm{dt_v\xi}_\infty\norm{v^\prime}_2
\\
   &\le \tfrac{8d\dot\beta_{\bar\Vfrak_z}}{\norm{z}_2}
{\color{blue}\;
   \abs{v}_{h_1}
   \cdot\abs{v-w}_{h_1}
\;}
   \norm{\xi}_{1,2}
   .
\end{split}
\end{equation}

\smallskip
\noindent
\textbf{\boldmath$D_{13}$.}
With constants $\dot\beta_{\SS^1}^w$ in~(\ref{eq:b-infty})
and $C_\rho^{\norm{z}_2}$ in~(\ref{eq:dt_z-diff})
we estimate
\begin{equation}\label{eq:43-D13}
\begin{split}
   \norm{D_{13}\xi}_2
   &=\norm{
   \dot b_{t_w(\tau)}|_{w_\tau}
   \left(
   (dt_v\xi)_\tau-(dt_w\xi)_\tau
   \right)
   j_0 v_\tau^\prime
   }_2
\\
   &\le \dot\beta_{\SS^1}^w
   \norm{dt_v\xi-dt_w\xi}_\infty
   \norm{v^\prime}_2
\\
   &\le \dot\beta_{\SS^1}^w 2 C_\rho^{\norm{z}_2} 
{\color{blue}\;
   \abs{v}_{h_1}\cdot\abs{v-w}_{h_1}
\;}
   \norm{\xi}_{1,2}
   .
\end{split}
\end{equation}

\smallskip
\noindent
\textbf{\boldmath$D_{14}$.}
With $\dot\beta_{\SS^1}^w$ in~(\ref{eq:b-infty})
and~(\ref{eq:dt|_z-2-infty}), followed by~(\ref{eq:V_z}),
we estimate
\begin{equation}\label{eq:43-D14}
\begin{split}
   \norm{D_{14}\xi}_2
   &=\norm{
   \dot b_{t_w(\tau)}|_{w_\tau}(dt_w\xi)_\tau\, j_0
   \left(v_\tau^{\prime}-w_\tau^{\prime}\right)
   }_2
\\
   &\le \dot\beta_{\SS^1}^w\tfrac{4}{\norm{v}_2}\norm{\xi}_2
   \norm{v^{\prime}-w^{\prime}}_2
\\
   &\le \tfrac{8d\dot\beta_{\bar\Vfrak_z}}{\norm{z}_2}
{\color{blue}\;
   \abs{v-w}_{h_1}
\;}
   \norm{\xi}_{1,2}
   .
\end{split}
\end{equation}

\medskip
\noindent
\textbf{Difference D2}
Add five zeroes to write D2 as a sum
$D_{21}\xi+\dots+D_{26}\xi$
\begin{equation*}
\begin{split}
   &
   \ddot b_{t_v(\tau)}|_{v_\tau}
   {\color{cyan}\tfrac{\abs{v_\tau}^2}{\norm{v}_2^2}}
   (dt_v\xi)_\tau\, j_0 v_\tau^\prime
   -
   \ddot b_{t_w(\tau)}|_{w_\tau}
   {\color{cyan}\tfrac{\abs{w_\tau}^2}{\norm{w}_2^2}}
   (dt_w\xi)_\tau\, j_0 w_\tau^\prime
\\
   &=
   \left(
   \ddot b_{t_v(\tau)}|_{v_\tau}-\ddot b_{t_w(\tau)}|_{v_\tau}
   \right)
   {\color{cyan}\tfrac{\abs{v_\tau}^2}{\norm{v}_2^2}}
   (dt_v\xi)_\tau\, j_0 v_\tau^\prime
\\
   &\quad+
   \left(
   \ddot b_{t_w(\tau)}|_{v_\tau}-\ddot b_{t_w(\tau)}|_{w_\tau}
   \right)
   {\color{cyan}\tfrac{\abs{v_\tau}^2}{\norm{v}_2^2}}
   (dt_v\xi)_\tau\, j_0 v_\tau^\prime
\\
   &\quad+
   \ddot b_{t_w(\tau)}|_{w_\tau}
   {\color{cyan}\tfrac{\abs{v_\tau}^2-\abs{w_\tau}^2}{\norm{v}_2^2}}
   (dt_v\xi)_\tau\, j_0 v_\tau^\prime
\\
   &\quad+
   \ddot b_{t_w(\tau)}|_{w_\tau} \abs{w_\tau}^2
   {\color{cyan}\tfrac{\norm{w}_2^2-\norm{v}_2^2}{\norm{v}_2^2 \norm{w}_2^2}}
   (dt_v\xi)_\tau\, j_0 v_\tau^\prime
\\
   &\quad+
   \ddot b_{t_w(\tau)}|_{w_\tau}
   {\color{cyan}\tfrac{\abs{w_\tau}^2}{\norm{w}_2^2}}
   \left((dt_v\xi)_\tau-(dt_w\xi)_\tau\right)
   j_0 v_\tau^\prime
\\
   &\quad+
   \ddot b_{t_w(\tau)}|_{w_\tau}
   {\color{cyan}\tfrac{\abs{w_\tau}^2}{\norm{w}_2^2}}
   (dt_w\xi)_\tau\, j_0
   \left(v_\tau^\prime-w_\tau^\prime\right)
   .
\end{split}
\end{equation*}

\smallskip
\noindent
\textbf{\boldmath$D_{21}$.}
By compactness of $\bar \Vfrak_z$ the maximum of the
absolute value is finite
$$
   \dddot\beta_{\bar\Vfrak_z}
   :=\max_{\SS^1\times\bar\Vfrak_z}\abs{\dddot b}
   <\infty .
$$
As in~(\ref{eq:db-t_v-t_w}) we obtain an estimate, uniform in
$\tau\in\SS^1$, for the operator norm
\begin{equation*}
\begin{split}
   \abs{\ddot b_{t_v(\tau)}|_{v_\tau}-\ddot b_{t_w(\tau)}|_{v_\tau}}
   &\le \tfrac{16\rho \dddot\beta_{\bar\Vfrak_z}}{\norm{z}_2^2}
   \norm{v-w}_2
   .
\end{split}
\end{equation*}
Use this operator norm estimate, as well
as~(\ref {eq:rho_z}) and~(\ref{eq:V_z}), to estimate
\begin{equation*}
\begin{split}
   \norm{D_{21}\xi}_2
   &=\norm{
   \left(
   \ddot b_{t_v(\tau)}|_{v_\tau}-\ddot b_{t_w(\tau)}|_{v_\tau}
   \right)
   {\color{cyan}\tfrac{\abs{v_\tau}^2}{\norm{v}_2^2}}
   (dt_v\xi)_\tau\, j_0 v_\tau^\prime
   }_2
\\
   &\le \tfrac{16\rho \dddot\beta_{\bar\Vfrak_z}}{\norm{z}_2^2}
   \norm{v-w}_2
   \tfrac{2^2\rho^2}{\norm{z}_2^2}
   \tfrac{4}{\norm{z}_2}\norm{\xi}_2
   \norm{v^\prime}_2
\\
   &\le \tfrac{16\cdot 32\rho^3\dddot\beta_{\bar\Vfrak_z}\norm{z}_{1,2}}
      {\norm{z}_2^5}
{\color{blue}\;
   \abs{v-w}_{h_0}
\;}
   \norm{\xi}_{1,2}
   .
\end{split}
\end{equation*}

\smallskip
\noindent
\textbf{\boldmath$D_{22}$.}
By Taylor's theorem and $d\ddot\beta_{\bar\Vfrak_z}$
in~(\ref{eq:d-ddot-beta-43}) we estimate similarly as above
\begin{equation*}
\begin{split}
   \norm{D_{22}\xi}_2
   &=\norm{
   \left(
   \ddot b_{t_w(\tau)}|_{v_\tau}-\ddot b_{t_w(\tau)}|_{w_\tau}
   \right)
   {\color{cyan}\tfrac{\abs{v_\tau}^2}{\norm{v}_2^2}}
   (dt_v\xi)_\tau\, j_0 v_\tau^\prime
   }_2
\\
   &\le d\ddot\beta_{\bar\Vfrak_z} \norm{v-w}_\infty
   \tfrac{2^2\rho^2}{\norm{z}_2^2}
   \tfrac{4}{\norm{z}_2}\norm{\xi}_2
   \norm{v^\prime}_2
\\
   &\le \tfrac{32\rho^2(d\ddot\beta_{\bar\Vfrak_z})\norm{z}_{1,2}}
      {\norm{z}_2^3}
{\color{blue}\;
   \abs{v-w}_{h_1}
\;}
   \norm{\xi}_{1,2}
   .
\end{split}
\end{equation*}

\smallskip
\noindent
\textbf{\boldmath$D_{23}$.}
With $\ddot\beta_{\SS^1}^w$ in~(\ref{eq:b-infty})
and~(\ref{eq:dt|_z-2-infty}), followed by~(\ref{eq:V_z}),
we estimate
\begin{equation*}
\begin{split}
   \norm{D_{23}\xi}_2
   &=\norm{
   \ddot b_{t_w(\tau)}|_{w_\tau}
   {\color{cyan}\tfrac{\abs{v_\tau}^2
   {\color{gray}\;-\langle v_\tau,w_\tau\rangle_0+\langle v_\tau,w_\tau\rangle_0 \;}
   -\abs{w_\tau}^2}{\norm{v}_2^2}}
   (dt_v\xi)_\tau\, j_0 v_\tau^\prime
   }_2
\\
   &\le \ddot\beta_{\SS^1}^w
   \tfrac{\norm{v-w}_\infty\left(\norm{v}_\infty+\norm{w}_\infty\right)}{\norm{v}_2^2}
   \norm{dt_v\xi}_\infty \norm{v^\prime}_2
\\
   &\le 
   \tfrac{128\rho \ddot\beta_{\SS^1}^w\norm{z}_{1,2}}{\norm{z}_2^3}
{\color{blue}\;
   \abs{v-w}_{h_1}
\;}
   \norm{\xi}_{1,2}
   .
\end{split}
\end{equation*}

\smallskip
\noindent
\textbf{\boldmath$D_{24}$.}
With $\ddot\beta_{\SS^1}^w$ in~(\ref{eq:b-infty})
and~(\ref{eq:dt|_z-2-infty}), followed by~(\ref{eq:V_z}),
we estimate
\begin{equation*}
\begin{split}
   \norm{D_{24}\xi}_2
   &=\norm{
   \ddot b_{t_w(\tau)}|_{w_\tau} \abs{w_\tau}^2
   {\color{cyan}\tfrac{\norm{w}_2^2
   {\color{gray}\;-\langle v,w\rangle+\langle v,w\rangle \;}
   -\norm{v}_2^2}{\norm{v}_2^2 \norm{w}_2^2}}
   (dt_v\xi)_\tau\, j_0 v_\tau^\prime
   }_2
\\
   &\le \ddot\beta_{\SS^1}^w\norm{w}_\infty^2
   \tfrac{\norm{v-w}_2\left(\norm{v}_\infty+\norm{w}_\infty\right)}
      {\norm{v}_2^2 \norm{w}_2^2}
   \norm{dt_v\xi}_\infty\norm{v^\prime}_2
\\
   &\le 
   \tfrac{512\rho^3 \ddot\beta_{\SS^1}^w\norm{z}_{1,2}}{\norm{z}_2^5}
{\color{blue}\;
   \abs{v-w}_{h_1}
\;}
   \norm{\xi}_{1,2}
   .
\end{split}
\end{equation*}

\smallskip
\noindent
\textbf{\boldmath$D_{25}$.}
With constants $\ddot\beta_{\SS^1}^w$ in~(\ref{eq:b-infty})
and $C_\rho^{\norm{z}_2}$ in~(\ref{eq:dt_z-diff})
we estimate
\begin{equation*}
\begin{split}
   \norm{D_{25}\xi}_2
   &=\norm{
   \ddot b_{t_w(\tau)}|_{w_\tau}
   {\color{cyan}\tfrac{\abs{w_\tau}^2}{\norm{w}_2^2}}
   \left((dt_v\xi)_\tau-(dt_w\xi)_\tau\right)
   j_0 v_\tau^\prime
   }_2
\\
   &\le \ddot\beta_{\SS^1}^w {\color{cyan}\tfrac{4\rho^2}{\norm{z}_2^2}}
   \norm{dt_v\xi-dt_w\xi}_\infty\norm{v^\prime}_2
\\
   &\le \ddot\beta_{\SS^1}^w
   {\color{cyan}\tfrac{4\rho^2}{\norm{z}_2^2}}
   2 C_\rho^{\norm{z}_2}
{\color{blue}\;
   \abs{v}_{h_1}\cdot\abs{v-w}_{h_1}
\;}
   \norm{\xi}_{1,2}
   .
\end{split}
\end{equation*}

\smallskip
\noindent
\textbf{\boldmath$D_{26}$.}
With the constant $\ddot\beta_{\SS^1}^w$ in~(\ref{eq:b-infty}) 
and~(\ref {eq:rho_z}),~(\ref{eq:V_z}),~(\ref{eq:dt|_z-2-infty}) we estimate
\begin{equation*}
\begin{split}
   \norm{D_{26}\xi}_2
   &=\norm{
   \ddot b_{t_w(\tau)}|_{w_\tau}
   {\color{cyan}\tfrac{\abs{w_\tau}^2}{\norm{w}_2^2}}
   (dt_w\xi)_\tau\, j_0
   \left(v_\tau^\prime-w_\tau^\prime\right)
   }_2
\\
   &\le \ddot\beta_{\SS^1}^w\tfrac{2^2\rho^2}{\norm{z}_2^2}
   \tfrac{4\cdot 2}{\norm{z}_2}\norm{\xi}_2\norm{v^\prime-w^\prime}_2
\\
   &\le \tfrac{32\rho^2\ddot\beta_{\bar\Vfrak_z}}
      {\norm{z}_2^3}
{\color{blue}\;
   \abs{v-w}_{h_1}
\;}
   \norm{\xi}_{1,2}
   .
\end{split}
\end{equation*}

\medskip
\noindent
\textbf{\color{red}Difference D3}
Add four zeroes to write D3 as a sum
$D_{31}\xi+\dots+D_{35}\xi$
\begin{equation*}
\begin{split}
   &
   \left(d\dot b_{t_v(\tau)}|_{v_\tau} v_\tau^\prime\right)
   (dt_v\xi)_\tau\, j_0 v_\tau^\prime
   -
   \left(d\dot b_{t_w(\tau)}|_{w_\tau} w_\tau^\prime\right)
   (dt_w\xi)_\tau\, j_0 w_\tau^\prime
\\
   &=
   \left((d\dot b_{t_v(\tau)}|_{v_\tau}-d\dot b_{t_w(\tau)}|_{v_\tau})v_\tau^\prime\right)
   (dt_v\xi)_\tau\, j_0 v_\tau^\prime
   \\
   &\quad+
   \left((d\dot b_{t_w(\tau)}|_{v_\tau}-d\dot b_{t_w(\tau)}|_{w_\tau})v_\tau^\prime\right)
   (dt_v\xi)_\tau\, j_0 v_\tau^\prime
    \\
   &\quad+
   \left(d\dot b_{t_w(\tau)}|_{w_\tau} (v_\tau^\prime-w_\tau^\prime)\right)
   (dt_v\xi)_\tau\, j_0 v_\tau^\prime
   \\
   &\quad+
   \left(d\dot b_{t_w(\tau)}|_{w_\tau} w_\tau^\prime\right)
   \left((dt_v\xi)_\tau-(dt_w\xi)_\tau\right)
   j_0 v_\tau^\prime
   \\
   &\quad+
   \left(d\dot b_{t_w(\tau)}|_{w_\tau} w_\tau^\prime\right)
   (dt_w\xi)_\tau\, j_0 \left(v_\tau^\prime-w_\tau^\prime\right)
   .
\end{split}
\end{equation*}

\smallskip
\noindent
\textbf{\boldmath\color{red}$D_{31}$.}
By estimate~(\ref{eq:d-dot-b-diff-43}) with constant $d\ddot\beta_{\bar\Vfrak_z}$
and~(\ref{eq:dt|_z-2-infty}) we obtain
\begin{equation*}
\begin{split}
   \norm{D_{31}\xi}_2
   &=\norm{
   \left((d\dot b_{t_v(\tau)}|_{v_\tau}-d\dot b_{t_w(\tau)}|_{v_\tau})v_\tau^\prime\right)
   (dt_v\xi)_\tau\, j_0 v_\tau^\prime
   }_2
\\
   &\le \tfrac{16\rho (d\ddot\beta_{\bar\Vfrak_z})}{\norm{z}_2^2}
   \norm{v-w}_2
{\color{red}\;
   \norm{v^\prime}_\infty
}
   \norm{dt_v\xi}_\infty{\color{cyan}\norm{v^\prime}_2}
\\
   &\le \tfrac{16\rho (d\ddot\beta_{\bar\Vfrak_z})}{\norm{z}_2^2}
   \tfrac{8\cdot {\color{cyan}2\norm{z}_{1,2}}}{\norm{z}_2}
{\color{blue}\;
{\color{red}\;
   \abs{v}_{h_2}\cdot
}
   \abs{v-w}_{h_0}
\;}
   \norm{\xi}_{1,2}
   .
\end{split}
\end{equation*}

\smallskip
\noindent
\textbf{\boldmath\color{red}$D_{32}$.}
By Taylor's theorem and $d^2\dot\beta_{\bar\Vfrak_z}$
in~(\ref{eq:d^2-dot-beta-43}) we estimate as above
\begin{equation*}
\begin{split}
   \norm{D_{32}\xi}_2
   &=\norm{
   \left((d\dot b_{t_w(\tau)}|_{v_\tau}-d\dot b_{t_w(\tau)}|_{w_\tau})v_\tau^\prime\right)
   (dt_v\xi)_\tau\, j_0 v_\tau^\prime
   }_2
\\
   &\le d^2\dot\beta_{\bar\Vfrak_z}\norm{v-w}_\infty
{\color{red}\;
   \norm{v^\prime}_\infty
}
   \norm{dt_v\xi}_\infty{\color{cyan}\norm{v^\prime}_2}
\\
   &\le d^2\dot\beta_{\bar\Vfrak_z}
   \tfrac{8\cdot {\color{cyan}2\norm{z}_{1,2}}}{\norm{z}_2}
{\color{red}\;
   \abs{v}_{h_2}\cdot
   \abs{v-w}_{h_1}
\;}
   \norm{\xi}_{1,2}
   .
\end{split}
\end{equation*}

\smallskip
\noindent
\textbf{\boldmath\color{red}$D_{33}$.}
With $d\dot\beta_{\SS^1}^w$ from~(\ref{eq:b-infty})
we estimate as above
\begin{equation*}
\begin{split}
   \norm{D_{33}\xi}_2
   &=\norm{
   \left(d\dot b_{t_w(\tau)}|_{w_\tau} (v_\tau^\prime-w_\tau^\prime)\right)
   (dt_v\xi)_\tau\, j_0 v_\tau^\prime
   }_2
\\
   &\le d\dot\beta_{\SS^1}^w\norm{v^\prime-w^\prime}_\infty
   \norm{dt_v\xi}_\infty{\color{cyan}\norm{v^\prime}_2}
\\
   &\le d^2\dot\beta_{\bar\Vfrak_z}
   \tfrac{8\cdot {\color{cyan}2\norm{z}_{1,2}}}{\norm{z}_2}
{\color{red}\;
   \abs{v-w}_{h_2}
\;}
   \norm{\xi}_{1,2}
   .
\end{split}
\end{equation*}

\smallskip
\noindent
\textbf{\boldmath\color{red}$D_{34}$.}
With constants $d\dot\beta_{\SS^1}^w$ in~(\ref{eq:b-infty})
and $C_\rho^{\norm{z}_2}$ in~(\ref{eq:dt_z-diff})
we estimate
\begin{equation*}
\begin{split}
   \norm{D_{34}\xi}_2
   &=\norm{
   \left(d\dot b_{t_w(\tau)}|_{w_\tau} w_\tau^\prime\right)
   \left((dt_v\xi)_\tau-(dt_w\xi)_\tau\right)
   j_0 v_\tau^\prime
   }_2
\\
   &\le d\dot\beta_{\SS^1}^w
   {\color{cyan}\norm{w^\prime}_2}
   \norm{dt_v\xi-dt_w\xi}_\infty\norm{v^\prime}_\infty
\\
   &\le d\dot\beta_{\SS^1}^w {\color{cyan}2\norm{z}_{1,2}}
   2 C_\rho^{\norm{z}_2}
{\color{red}\;
   \abs{v}_{h_2}\cdot\abs{v-w}_{h_1}
\;}
   \norm{\xi}_{1,2}
   .
\end{split}
\end{equation*}

\smallskip
\noindent
\textbf{\boldmath\color{red}$D_{35}$.}
With $d\dot\beta_{\SS^1}^w$ from~(\ref{eq:b-infty}) and~(\ref{eq:V_z})
for ${\color{magenta}\norm{w}_{1,2}}$ we estimate as above
\begin{equation*}
\begin{split}
   \norm{D_{35}\xi}_2
   &=\norm{
   \left(d\dot b_{t_w(\tau)}|_{w_\tau} w_\tau^\prime\right)
   (dt_w\xi)_\tau\, j_0 \left(v_\tau^\prime-w_\tau^\prime\right)
   }_2
\\
   &\le d\dot\beta_{\SS^1}^w{\color{magenta}\norm{w^\prime}_2}
   \norm{dt_w\xi}_\infty\norm{v^\prime-w^\prime}_\infty
\\
   &\le d\dot\beta_{\SS^1}^w{\color{magenta}2\norm{z}_{1,2}}
   \tfrac{8}{\norm{z}_2}\norm{\xi}_2\norm{v-w}_{2,2}
\\
   &\le \tfrac{16(d^2\dot\beta_{\bar\Vfrak_z})\norm{z}_{1,2}}{\norm{z}_2}
{\color{red}\;
   \abs{v-w}_{h_2}
\;}
   \norm{\xi}_{1,2}
   .
\end{split}
\end{equation*}

\medskip
\noindent
\textbf{Difference D4}
Add three zeroes to write D4 as a sum
$D_{41}\xi+\dots\xi+ D_{44}\xi$
\begin{equation*}
\begin{split}
   &
{\color{red}
   \dot b_{t_v(\tau)}|_{v_\tau}(dt_v\xi)^\prime_\tau\, j_0 v_\tau^\prime
   -
   \dot b_{t_w(\tau)}|_{w_\tau}(dt_w\xi)^\prime_\tau\, j_0 w_\tau^\prime
}
\\
   &=
   \dot b_{t_v(\tau)}|_{v_\tau}(dt_v\xi)^\prime_\tau\, j_0 v_\tau^\prime
   -
   \dot b_{t_w(\tau)}|_{v_\tau}(dt_v\xi)^\prime_\tau\, j_0 v_\tau^\prime
   \\
   &\quad
   \dot b_{t_w(\tau)}|_{v_\tau}(dt_v\xi)^\prime_\tau\, j_0 v_\tau^\prime
   -
   \dot b_{t_w(\tau)}|_{w_\tau}(dt_v\xi)^\prime_\tau\, j_0 v_\tau^\prime
   \\
   &\quad
   \dot b_{t_w(\tau)}|_{w_\tau}(dt_v\xi)^\prime_\tau\, j_0 v_\tau^\prime
   -
   \dot b_{t_w(\tau)}|_{w_\tau}(dt_w\xi)^\prime_\tau\, j_0 v_\tau^\prime
   \\
   &\quad
   \dot b_{t_w(\tau)}|_{w_\tau}(dt_w\xi)^\prime_\tau\, j_0 v_\tau^\prime
   -
   \dot b_{t_w(\tau)}|_{w_\tau}(dt_w\xi)^\prime_\tau\, j_0 w_\tau^\prime
   .
\end{split}
\end{equation*}

\smallskip
\noindent
\textbf{\boldmath$D_{41}$.}
We use~(\ref{eq:V_z}),~(\ref{eq:dot-b-diff})
and~(\ref{eq:dt_z-xi-infty}) to estimate
\begin{equation*}
\begin{split}
   \norm{D_{41}\xi}_2
   &=\norm{
   \left(
   \dot b_{t_v(\tau)}|_{v_\tau}-\dot b_{t_w(\tau)}|_{v_\tau}
   \right)
   (dt_v\xi)^\prime_\tau\, j_0 v_\tau^\prime
   }_2
\\
   &\le \tfrac{16\rho \ddot\beta_{\bar\Vfrak_z}}{\norm{z}_2^2}
   \norm{v-w}_2\norm{(dt_v\xi)^\prime}_\infty{\color{cyan}\norm{v^\prime}_2}
\\
   &\le \tfrac{16\rho \ddot\beta_{\bar\Vfrak_z}{\color{cyan}2\norm{z}_{1,2}}}{\norm{z}_2^2}
   \left(
   \tfrac{2 \norm{z}_{1,2}}{\norm{z}_2^2}
   +\tfrac{2 \norm{z}_{1,2}^2}{\norm{z}_2^3}
   \right)
{\color{blue}\;
   \abs{v-w}_{h_0}
\;}
   \norm{\xi}_{1,2}
   .
\end{split}
\end{equation*}

\smallskip
\noindent
\textbf{\boldmath$D_{42}$.}
By Taylor's theorem and $d\dot\beta_{\bar\Vfrak_z}$
in~(\ref{eq:d-dot-beta}) 
and by~(\ref{eq:dt_z-xi-infty}) we estimate
\begin{equation*}
\begin{split}
   \norm{D_{42}\xi}_2
   &=\norm{
   \left(
   \dot b_{t_w(\tau)}|_{v_\tau}-\dot b_{t_w(\tau)}|_{w_\tau}
   \right)
   (dt_v\xi)^\prime_\tau\, j_0 v_\tau^\prime
   }_2
\\
   &\le d\dot\beta_{\bar\Vfrak_z}\norm{v-w}_\infty
   \norm{(dt_v\xi)^\prime}_\infty{\color{cyan}\norm{v^\prime}_2}
\\
   &\le d\dot\beta_{\bar\Vfrak_z} {\color{cyan}2\norm{z}_{1,2}}
   \left(
   \tfrac{2 \norm{z}_{1,2}}{\norm{z}_2^2}
   +\tfrac{2 \norm{z}_{1,2}^2}{\norm{z}_2^3}
   \right)
{\color{blue}\;
   \abs{v-w}_{h_1}
\;}
   \norm{\xi}_{1,2}
   .
\end{split}
\end{equation*}

\smallskip
\noindent
\textbf{\boldmath$D_{43}$.}
With constants $\dot\beta_{\SS^1}^w$ in~(\ref{eq:b-infty}) 
and $C_\rho^{\norm{z}_2}$ in~(\ref{eq:dt_z-deriv-difference-infty}) we
estimate
\begin{equation}\label{eq:D43-42}
\begin{split}
   \norm{D_{43}\xi}_2
   &=\norm{
   \dot b_{t_w(\tau)}|_{w_\tau}
   \left((dt_v\xi)^\prime_\tau-(dt_w\xi)^\prime_\tau\right)
   j_0 v_\tau^\prime
   }_2
\\
   &\le \dot\beta_{\SS^1}^w
   \norm{(dt_v\xi)^\prime-(dt_w\xi)^\prime}_\infty{\color{cyan}\norm{v^\prime}_2}
\\
   &\le \dot\beta_{\SS^1}^w
   {\color{cyan}2\norm{z}_{1,2}} 2C_\rho^{\norm{z}_2}
{\color{blue}\;
   \abs{v-w}_{h_1}
\;}
   \norm{\xi}_{1,2}
   .
\end{split}
\end{equation}

\smallskip
\noindent
\textbf{\boldmath$D_{44}$.}
With $\dot\beta_{\SS^1}^w$ in~(\ref{eq:b-infty})
and by~(\ref{eq:dt_z-xi-infty}) we estimate
\begin{equation*}
\begin{split}
   \norm{D_{44}\xi}_2
   &=\norm{
   \dot b_{t_w(\tau)}|_{w_\tau}(dt_w\xi)^\prime_\tau\, j_0 
   \left(v_\tau^\prime-w_\tau^\prime\right)
   }_2
\\
   &\le \dot\beta_{\SS^1}^w
   \norm{(dt_v\xi)^\prime}_\infty\norm{v^\prime-w^\prime}_2
\\
   &\le \dot\beta_{\SS^1}^w
   \left(
   \tfrac{2 \norm{z}_{1,2}}{\norm{z}_2^2}
   +\tfrac{2 \norm{z}_{1,2}^2}{\norm{z}_2^3}
   \right)
{\color{blue}\;
   \abs{v-w}_{h_1}
\;}
   \norm{\xi}_{1,2}
   .
\end{split}
\end{equation*}

\medskip
\noindent
\textbf{\color{red}Difference D5}
Add three zeroes to write D5 as a sum
$D_{51}\xi+\dots+ D_{54}\xi$
\begin{equation*}
\begin{split}
   &
   \dot b_{t_v(\tau)}|_{v_\tau}(dt_v\xi)_\tau\, j_0 v_\tau^{\prime\prime}
   -
   \dot b_{t_w(\tau)}|_{w_\tau}(dt_w\xi)_\tau\, j_0 w_\tau^{\prime\prime}
\\
   &=
   \dot b_{t_v(\tau)}|_{v_\tau}(dt_v\xi)_\tau\, j_0 v_\tau^{\prime\prime}
   -
   \dot b_{t_w(\tau)}|_{v_\tau}(dt_v\xi)_\tau\, j_0 v_\tau^{\prime\prime}
   \\
   &\quad
  +
   \dot b_{t_w(\tau)}|_{v_\tau}(dt_v\xi)_\tau\, j_0 v_\tau^{\prime\prime}
   -
   \dot b_{t_w(\tau)}|_{w_\tau}(dt_v\xi)_\tau\, j_0 v_\tau^{\prime\prime}
   \\
   &\quad
  +
   \dot b_{t_w(\tau)}|_{w_\tau}(dt_v\xi)_\tau\, j_0 v_\tau^{\prime\prime}
   -
   \dot b_{t_w(\tau)}|_{w_\tau}(dt_w\xi)_\tau\, j_0 v_\tau^{\prime\prime}
   \\
   &\quad
  +
   \dot b_{t_w(\tau)}|_{w_\tau}(dt_w\xi)_\tau\, j_0 v_\tau^{\prime\prime}
   -
   \dot b_{t_w(\tau)}|_{w_\tau}(dt_w\xi)_\tau\, j_0 w_\tau^{\prime\prime}
   .
\end{split}
\end{equation*}

\smallskip
\noindent
\textbf{\boldmath\color{red}$D_{51}$.}
This is estimate~(\ref{eq:43-D11})
with $v^\prime$ replaced by $v^{\prime\prime}$, so
\begin{equation*}
\begin{split}
   \norm{D_{51}\xi}_2
   &=\norm{
   \left(
   \dot b_{t_v(\tau)}|_{v_\tau}-\dot b_{t_w(\tau)}|_{v_\tau}
   \right)
   (dt_v\xi)_\tau\, j_0 v_\tau^{\prime\prime}
   }_2
\\
   &\le \tfrac{16\rho \ddot\beta_{\bar\Vfrak_z}}{\norm{z}_2^2}
   \norm{v-w}_2\norm{dt_v\xi}_\infty\norm{v^{\prime\prime}}_2
\\
   &\le 
   \tfrac{128\rho \ddot\beta_{\bar\Vfrak_z}}{\norm{z}_2^3}
{\color{blue}\;
{\color{red}\;
   \abs{v}_{h_2}
}
   \cdot\abs{v-w}_{h_0}
\;}
   \norm{\xi}_{1,2}
   .
\end{split}
\end{equation*}

\smallskip
\noindent
\textbf{\boldmath\color{red}$D_{52}$.}
This is estimate~(\ref{eq:43-D12})
with $v^\prime$ replaced by $v^{\prime\prime}$, so
\begin{equation*}
\begin{split}
   \norm{D_{52}\xi}_2
   &=\norm{
   \left(
   \dot b_{t_w(\tau)}|_{v_\tau}-\dot b_{t_w(\tau)}|_{w_\tau}
   \right)
   (dt_v\xi)_\tau\, j_0 v_\tau^{\prime\prime}
   }_2
\\
   &\le d\dot\beta_{\bar\Vfrak_z}\norm{v-w}_\infty
   \norm{dt_v\xi}_\infty\norm{v^{\prime\prime}}_2
\\
   &\le \tfrac{8d\dot\beta_{\bar\Vfrak_z}}{\norm{z}_2}
{\color{red}\;
   \abs{v}_{h_2}
   \cdot\abs{v-w}_{h_1}
\;}
   \norm{\xi}_{1,2}
   .
\end{split}
\end{equation*}

\smallskip
\noindent
\textbf{\boldmath\color{red}$D_{53}$.}
This is estimate~(\ref{eq:43-D13})
with $v^\prime$ replaced by $v^{\prime\prime}$, so
\begin{equation}\label{eq:42-D53}
\begin{split}
   \norm{D_{53}\xi}_2
   &=\norm{
   \dot b_{t_w(\tau)}|_{w_\tau}
   \left((dt_v\xi)_\tau-(dt_w\xi)_\tau\right)
   j_0 v_\tau^{\prime\prime}
   }_2
\\
   &\le \dot\beta_{\SS^1}^w
   \norm{dt_v\xi-dt_w\xi}_\infty
   \norm{v^{\prime\prime}}_2
\\
   &\le \dot\beta_{\SS^1}^w 2 C_\rho^{\norm{z}_2} 
{\color{red}\;
   \abs{v}_{h_2}\cdot\abs{v-w}_{h_1}
\;}
   \norm{\xi}_{1,2}
   .
\end{split}
\end{equation}

\smallskip
\noindent
\textbf{\boldmath\color{red}$D_{54}$.}
This is estimate~(\ref{eq:43-D14}) with $v^\prime-w^\prime$ replaced
by $v^{\prime\prime}-w^{\prime\prime}$, so
\begin{equation*}
\begin{split}
   \norm{D_{54}\xi}_2
   &=\norm{
   \dot b_{t_w(\tau)}|_{w_\tau}(dt_w\xi)_\tau\, j_0
   \left(v_\tau^{\prime\prime}-w_\tau^{\prime\prime}\right)
   }_2
\\
   &\le \dot\beta_{\SS^1}^w\tfrac{4}{\norm{v}_2}\norm{\xi}_2
   \norm{v^{\prime\prime}-w^{\prime\prime}}_2
\\
   &\le \tfrac{8d\dot\beta_{\bar\Vfrak_z}}{\norm{z}_2}
{\color{red}\;
   \abs{v-w}_{h_2}
\;}
   \norm{\xi}_{1,2}
   .
\end{split}
\end{equation*}

\medskip
Summarizing we have shown that there exists a constant $\kappa=\kappa(z)$
such that
\begin{equation*}
     \norm{C^v-C^w}_{\Ll(h_1)}
     \le\kappa\Bigl(\abs{v-w}_{h_2}
     +\abs{v}_{h_2}\cdot\abs{v-w}_{h_1}\Bigr).
\end{equation*}
whenever $v,w\in \mathrm{v}_z$.
Interchanging the roles of $v$ and $w$ we get
\begin{equation*}
     \norm{C^v-C^w}_{\Ll(h_1)}
     \le\kappa\Bigl(\abs{v-w}_{h_2}
     +\abs{w}_{h_2}\cdot\abs{v-w}_{h_1}\Bigr).
\end{equation*}
whenever $v,w\in \mathrm{v}_z$.
These two inequalities imply that
\begin{equation*}
     \norm{C^v-C^w}_{\Ll(h_1)}
     \le\kappa\Bigl(\abs{v-w}_{h_2}
     +\min\{\abs{v}_{h_2}, \abs{w}_{h_2}\}\cdot\abs{v-w}_{h_1}\Bigr).
\end{equation*}
whenever $v,w\in \mathrm{v}_z$.
This proves Proposition~\ref{prop:C_42}.
\end{proof}

\subsection{Twisted loops}
\label{sec:twisted-loop-case}

In the case of twisted loops $z(\tau+1)=-z(\tau)$ $\forall\tau\in\SS^1$ 
the Hilbert space triple has to be adjusted as follows
\begin{equation*}
\begin{split}
   H_0^-
   :&=L^2([0,1],\C^2)
\\
   H_1^-
   :&=\{\Upsilon\in W^{1,2}([0,1],\C^2)\mid \Upsilon(1)=-\Upsilon(0)\}
\\
   H_2^-
   :&=\{\Upsilon\in W^{2,2}([0,1],\C^2)\mid \Upsilon(1)=-\Upsilon(0)
   \wedge \Upsilon^\prime(1)=-\Upsilon^\prime(0)\}
\end{split}
\end{equation*}
and $h_0^-$, $h_1^-$, $h_2^-$
are defined equally, up to replacing $\C^2$ by $\C$.
Apart from this, the case of twisted loops proceeds completely
analogous to the periodic case.

\subsection{Proof of Theorem~\ref{thm:A} and
Theorem~\ref{thm:B}}
\label{sec:proof-AB}

\begin{itemize}\setlength\itemsep{0ex}
\item
  Theorem~\ref{thm:B} is Theorem~\ref{thm:Ham-ae}.
\item
  Theorem~\ref{thm:A} follows from Theorem~\ref{thm:B}
  and Theorem~\ref{thm:main-pDarboux}.
\end{itemize}

\appendix

\boldmath
\section{Symmetry of the \boldmath$L^2$-Hessian}
\label{sec:L2-Hess-symmetry}
\unboldmath

\begin{theorem}\label{thm:2nd-derivative}
Let $(H_0,H_1)$ be a Hilbert space pair\footnote{
  $H_0$ and $H_1$ are both infinite dimensional Hilbert spaces,
  $H_1\subset H_0$ is a dense subset, and inclusion
  $\iota\colon H_1\to H_0$ is a compact linear map.
  $H_0$ and $H_1$ are separable by~\cite[Cor.\,A.5]{Frauenfelder:2024c}.
}
and $U_1\subset H_1$ open.
Suppose $f\colon U_1\to\R$ is (i) a $C^1$ function
such that (ii) there exists a $C^1$ map
$\grad f\colon U_1\to H_0$
and such that (iii) there is the identity
$$
   df|_x\xi
   =\INNER{\grad f|_x}{\xi}_0
$$
whenever $x\in U_1$ and $\xi\in H_1$.
Then the following is true.
The second derivative of $f$ exists at every point $x\in U_1$,
it is given by the formula
$$
   d^2f|_x(\xi,\eta)
   =\INNER{d(\grad f)_x\xi}{\eta}_0
   =:B_x(\xi,\eta)
   ,
$$
and it varies continuously in $x$, that is $f\in C^2(U_1,\R)$.
\end{theorem}

\begin{proof}
The proof has three steps.
Without loss of generality we suppose that
$\abs{\cdot}_0\le\abs{\cdot}_1$, otherwise choose equivalent norms.

\medskip
\noindent
\textbf{Step~1.}
$\forall x\in U_1$ the map
$B_x\colon H_1\times H_1\to\R$ is a bounded bilinear form.

\begin{proof}
By definition of $B$ followed by Cauchy-Schwarz and then pulling out
the operator norm we estimate
\begin{equation*}
\begin{split}
   \Abs{B_x(\xi,\eta)}
   =\Abs{\INNER{d(\grad f)_x\xi}{\eta}_0}
   &\le
   \abs{d(\grad f)_x\xi}_0\abs{\eta}_0
\\
   &\le
   \norm{d(\grad f)_x}_{\Ll(H_1,H_0)}\abs{\xi}_1\abs{\eta}_1.
\end{split}
\end{equation*}
By assumption (ii) the right hand side is bounded.
\end{proof}

\medskip
\noindent
\textbf{Step~2.}
The map\footnote{
  By $\Ll(H_1,H_1;\R)$ we denote the Banach space of bounded bilinear
  maps $H_1\times H_1\to\R$.
  }
$
   B\colon U_1\to \Ll(H_1,H_1;\R)
$,
$
   x\mapsto B_x
$,
is continuous. 

\begin{proof}
For $x,y\in U_1$ and $\xi,\eta\in H_1$ we estimate
\begin{equation*}
\begin{split}
   \Abs{(B_x-B_y)(\xi,\eta)}
   &=\Abs{\INNER{d(\grad f)_x\xi-d(\grad f)_y\xi}{\eta}_0}
\\
   &\le
   \abs{\left(d(\grad f)_x-d(\grad f)_y\right)\xi}_0\abs{\eta}_0
\\
   &\le
   \norm{d(\grad f)_x-d(\grad f)_y}_{\Ll(H_1,H_0)}\abs{\xi}_1\abs{\eta}_1.
\end{split}
\end{equation*}
By assumption (ii) the map $x\mapsto d(\grad f)_x$ is continuous.
\end{proof}

\medskip
\noindent
\textbf{Step~3.}
At any $x\in U_1$ the map $B_x$ is the second derivative $d^2f|_x$ of
$f$ at $x$.

\begin{proof}
The assertion of Step~3 is equivalent to
$$
   \sup_{\abs{\eta}_1=1}
   \tfrac{1}{\abs{\xi}_1}
   \left|
   df|_x\eta-df|_{x+\xi}\eta-B_x(\xi,\eta)
   \right|
   \longrightarrow 0
   \quad
   \text{, as $\abs{\xi}_1\to 0$.}
$$
We use the definitions of $\grad f|_x$ and $B_x$ to write
\begin{equation*}
\begin{split}
   &\sup_{\abs{\eta}_1=1}
   \tfrac{1}{\abs{\xi}_1}
   \left|
   \INNER{\grad  f|_x}{\eta}_0
   -\INNER{\grad  f|_{x+\xi}}{\eta}_0
   -\INNER{d(\grad f)_x\xi}{\eta}_0
   \right|
\\
   &=\sup_{\abs{\eta}_1=1}
   \left|
   \INNER{\tfrac{\grad  f|_x-\grad f|_{x+\xi}-d(\grad f)_x\xi}{\abs{\xi}_1}}
   {\eta}_0
   \right|
\\
   &\le \sup_{\abs{\eta}_1=1}
   \tfrac{\Abs{\grad f|_x-\grad f|_{x+\xi}-d(\grad f)_x\xi}_0}
   {\abs{\xi}_1}
   \Abs{\eta}_0
\\
   &\le
   \tfrac{\Abs{\grad f|_x-\grad f|_{x+\xi}-d(\grad f)_x\xi}_0}
   {\abs{\xi}_1}
\end{split}
\end{equation*}
whenever $x+\xi\in U_1$.
The first inequality is by Cauchy Schwarz.
The second inequality uses $\abs{\cdot}_0\le\abs{\cdot}_1$.
By definition of the derivative $d(\grad f)_x$
the right hand side converges to zero, as $\abs{\xi}_1\to 0$.
Hence the left hand side converges to zero, as $\abs{\xi}_1\to 0$.
This proves Step~3.
\end{proof}

By Step~3 the second derivative of $f$ exists at every $x\in U_1$
and by Step~2 it varies continuously in $x$.
This proves Theorem~\ref{thm:2nd-derivative}.
\end{proof}

The following corollary is used to prove
Proposition~\ref{prop:Ham-weak-mag}
(Step~2).

\begin{corollary}[$H_0$-symmetry]\label{cor:H0-symm}
The derivative of the $H_0$-gradient 
is $H_0$-symmetric, namely
$
   \INNER{d(\grad f)_x\xi}{\eta}_0
   =\INNER{d(\grad f)_x\eta}{\xi}_0
$
for all $\xi,\eta\in H_1$.
\end{corollary}

\begin{proof}
By Theorem~\ref{thm:2nd-derivative} we have $f\in C^2$.
The second derivative of a $C^2$ function is symmetric by the infinite
dimensional version of the classical Theorem of Schwarz;
see e.g.~\cite[Thm.\,3.4]{ambrosetti:1993a}.
\end{proof}

\boldmath
\section{Lagrangian Hessian field \boldmath$B$}
\label{sec:Hess-Lag}
\unboldmath

In~\ref{sec:calc-Lag-Hess} we calculate the
Hessian of the regularized Lagrangian functional
$\Bb=\Kk-{\,\color{cyan}\Uu}+\Mm$ defined by~(\ref{eq:Bb}).
The contributions from the magnetic functional $\Mm$
play a crucial role in the main
Section~\ref{sec:Hess-Ham} of this article.
In Section~\ref{sec:Lag-Kepler} we show that in the Kepler case
($\Mm=0$) the Lagrangian Hessian field almost extends.
%

\boldmath
\subsection{Calculation of Hessian operators}
\label{sec:calc-Lag-Hess}
\unboldmath

Abbreviate $X_\eps:=\grad\Bb(z_\eps)$.
To get rid of the fractions in formula~(\ref{eq:grad-Bb})
for $\grad\Bb$ we proceed as follows.
Since $\norm{z_\eps}^4\not=0$ and since by the product rule
$$
   \norm{z}^4 \underbrace{\left.\tfrac{d}{d\eps}\right|_0 X_\eps}_{B^z\xi}
   =\left.\tfrac{d}{d\eps}\right|_0\left(\norm{z_\eps}^4 X_\eps\right)
   -\underbrace{\left(\left.\tfrac{d}{d\eps}\right|_0\norm{z_\eps}^4\right)X_0}
   _{4\norm{z}^2\inner{z}{\xi}\grad\Bb(z)}
   ,
$$
the Hessian operator $B^z$ is equal to the difference
\begin{equation}\label{eq:B^z-alternativ}
   B^z\xi
   =
\underline{
   \tfrac{1}{\norm{z}^4}
   \left.\tfrac{d}{d\eps}\right|_0\left(\norm{z_\eps}^4 X_\eps\right)
}
   -\tfrac{4 \inner{z}{\xi}}{\norm{z}^2}
   \,\grad\Bb(z)
   ,\quad
   {\color{gray}X_\eps=\grad\Bb(z_\eps),}
\end{equation}
where summand two is a sum of
\textbf{seven terms \boldmath$m1+\dots+m7$}, namely
\begin{equation}\label{eq:factor-grad-Bb}
\begin{split}
   &-\tfrac{4\inner{z}{\xi}}{\norm{z}^2}(\grad\Bb|_z)_\tau
\\
   &\stackrel{\text{(\ref{eq:grad-Bb})}}{=}
\underbrace{
   -16\inner{z}{\xi}\tfrac{\norm{z^\prime}^2}{\norm{z}^2} z_\tau
}_{=: m_1\text{ ($\tilde T_{11}+m_1=0$)}}
   +
\underbrace{
   16 \inner{z}{\xi} z^{\prime\prime}_\tau
}_{=: m_2\text{ ($+\tilde T_{21}$)}}
   +
\underbrace{
{\color{cyan}
   8 \tfrac{\inner{z}{\xi}}{\norm{z}^6} z_\tau
}
}_{=: {\color{cyan}m_3}\text{ (add to $\tilde T_{31}$)}}
   \\
   &\qquad
   +
\underbrace{
   \tfrac{8\inner{z}{\xi}z_\tau}{\norm{z}^6}
   \int_0^1
{\textstyle
   \int_0^s\Abs{z_\sigma}^2 d\sigma
}
   \cdot
   \inner{\dot{\aaa}_{t_z(s)}|_{z_s}}{z^\prime_s}_0\; ds
}_{=:m_7\text{ (add to $\tilde T_{71}$)}}
   +
\underbrace{
   \tfrac{4\inner{z}{\xi}\Abs{z_\tau}^2}{\norm{z}^4}\dot{\aaa}_{t_z(\tau)}|_{z_\tau}
}_{=:m_5\text{ (add to $\tilde T_{56}$)}}
   \\
   &\qquad
   -
\underbrace{
   \tfrac{8\inner{z}{\xi} z_\tau}{\norm{z}^4}
{\textstyle
   \int_{\tau}^1
}
   \inner{\dot{\aaa}_{t_z(\sigma)}|_{z_\sigma}}{z^\prime_\sigma}_0\;
   d\sigma
}_{=:m_6\text{ (add to $\tilde T_{62}$)}}
   +
\underbrace{
   \tfrac{4\inner{z}{\xi}}{\norm{z}^2}
   \; b_{t_z(\tau)}|_{z_\tau}\; j_0 z_\tau^\prime
}_{=:m_4\text{ (add to $\tilde T_{41}$)}}
   .
\end{split}
\end{equation}
So the task at hand is to calculate summand one
in~(\ref{eq:B^z-alternativ}), namely
\begin{equation*}
\small
\begin{split}
   \left.\tfrac{d}{d\eps}\right|_0
   \tfrac{\norm{z_\eps}^4 X_\eps}{\norm{z}^4}
   &=\tfrac{1}{\norm{z}^4}\left.\tfrac{d}{d\eps}\right|_0
   \left(
   \norm{z_\eps}^4\grad\Kk|_{z_\eps}
   -
{\color{cyan}
   \norm{z_\eps}^4\grad\Uu|_{z_\eps}
}
   +\norm{z_\eps}^4\grad\Mm|_{z_\eps}
   \right) .
\end{split}
\end{equation*}
To this end we differentiate and get the following sum of
\textbf{seven terms T1--T7}
\begin{equation*}
\begin{split}
   &
\underline{
   \tfrac{1}{\norm{z}^4}
   \left.\tfrac{d}{d\eps}\right|_0\left(\norm{z_\eps}^4 X_\eps\right)
}
\\
   &\stackrel{(\ref{eq:grad-Bb})}{=}
   \tfrac{1}{\norm{z}^4}\left.\tfrac{d}{d\eps}\right|_0
   \biggl(
   4\norm{z_\eps}^4\norm{z_\eps^\prime}^2z_\eps
   -4\norm{z_\eps}^6z_\eps^{\prime\prime}
   -
{\color{cyan}
   2z_\eps
}
   -\norm{z_\eps}^4\bigl(\rot\,\aaa_{t_{z_\eps}}|_{z_\eps}\bigr)\; j_0 z_\eps^\prime
   \\
   &\qquad
   -\Abs{z_\eps}^2\norm{z_\eps}^2\dot{\aaa}_{t_{z_\eps}}|_{z_\eps}
   +2 z_\eps\norm{z_\eps}^2\int_{\sigma=\tau}^1
   \inner{\dot{\aaa}_{t_{z_\eps}(\sigma)}|_{z_{\eps,\sigma}}}{z^\prime_{\eps,\sigma}}_0
   \; d\sigma
   \\
   &\qquad
  -2z_\eps\int_0^1
{\textstyle
   \int_0^s\Abs{z_{\eps,\sigma}}^2 d\sigma
}
   \cdot\inner{\dot{\aaa}_{t_{z_\eps}(s)}|_{z_{\eps,s}}}{z^\prime_{\eps,s}}_0\; ds
   \biggr) .
\end{split}
\end{equation*}

\boldmath
\subsubsection{Terms T1--T7}
\unboldmath

\begin{remark}[Notation $T$, $C$, $F$]
In the following we encounter many summands, general notation $T_{ij}$,
occasionally $\tilde T_{rs}$, also $C_{42}$, $C_{43}$, and $F_{44}$.
The tilde indicates that some correction term $m_r$ will eventually be
added. So we can denote the sum $\tilde T_{rs}+m_r$ by $T_{rs}$.
The naming of the $F_{44}$ and the two $C$ summands refers to the
$A=F+C$
decomposition~(\ref{eq:A=F+C}), while all the many $T_{ij}$'s will be
compact perturbations of a Fredholm operator $F$.
For the $C$ summands one must establish the scale Lipschitz
estimate~(\ref{eq:(C)}). So one wishes to minimize the number of $C$
summands.
A summand enforces to be a $C$ summand whenever the base point
$z$ is required to be in level two (i.e. $z\in u_2$ in the Lagrangian or
$z\in U_2$ in the Hamiltonian case \S\,\ref{sec:Hess-Ham}, respectively).
\end{remark}

\noindent
\textbf{Term T1.} (Kinetic energy)
Differentiate $\norm{z_\eps}^4$ as $(\norm{z_\eps}^2)^2$ to get
\begin{equation*}
\begin{split}
   \left.\tfrac{d}{d\eps}\right|_0
   \tfrac{
   4\norm{z_\eps}^4\norm{z_\eps^\prime}^2z_\eps
   }{\norm{z}^4}
   &=\tfrac{1}{\norm{z}^4}
   \left(
   16\norm{z}^2\norm{z^\prime}^2\inner{z}{\xi} z
   +8\norm{z}^4\inner{z^\prime}{\xi^\prime} z
   +4\norm{z}^4\norm{z^\prime}^2\xi
   \right)
\\
   &=
\underbrace{
   16 \tfrac{\norm{z^\prime}^2}{\norm{z}^2}\inner{z}{\xi} z
}_{\tilde T_{11}^z\xi\; (+m_1=T_{11}^z\xi)}
  +
\underbrace{
   8\inner{z^\prime}{\xi^\prime} z
}_{T_{12}^z\xi}
  +
\underbrace{
   \norm{z^\prime}^2\xi
}_{T_{13}^z\xi}
   %
   .
\end{split}
\end{equation*}
\newline
\textbf{Term T2.} (Kinetic energy)
Differentiate $\norm{z_\eps}^6$ as $(\norm{z_\eps}^2)^3$ to get
\begin{equation*}
\begin{split}
   -\left.\tfrac{d}{d\eps}\right|_0
   \tfrac{
   4\norm{z_\eps}^6z_\eps^{\prime\prime}
   }{\norm{z}^4}
   &=-
   \tfrac{
   24\norm{z}^4\inner{z}{\xi} z^{\prime\prime}
   +4\norm{z}^6\xi^{\prime\prime}
   }{\norm{z}^4}
   =
\underbrace{
  -24\inner{z}{\xi} z^{\prime\prime}
}_{\tilde C_{21}^z\xi\; (+m_2=C_{21}^z\xi)}
\underbrace{
   -4\norm{z}^2 {\color{orange}\,\xi^{\prime\prime}}
}_{{\color{orange}\, F_{22}^z\,}\xi}
   .
\end{split}
\end{equation*}

\smallskip
\noindent
\textbf{Term T3.} ({\color{cyan}Potential})
$-\tfrac{1}{\norm{z}^4}\left.\tfrac{d}{d\eps}\right|_0 2z_\eps
=
{\color{cyan}
   -\tfrac{2\xi}{\norm{z}^4}
}
=\tilde T_{31}^z \;(+m_3=T_{31}^z)
.
$

\begin{remark}[Magnetic T4-T7]\label{rem:M_ij}
\smallskip
Consider the function and its differential
\begin{equation}\label{eq:b}
\begin{split}
   b_{t_z}|_z:
   &=\rot\,\aaa_{t_{z}}|_{z}
   =(\p_1a^2-\p_2a^1)_{t_z}|_z
\\
   d b_{t_z}|_z\xi
   &=(\p_{11}a^2\xi^1+\p_{21}a^2\xi^2
   -\p_{12}a^1\xi^1-\p_{22}a^1\xi^1)_{t_z}|_z
\end{split}
\end{equation}
in order to economize the computation of term T4.
In the following we enumerate each summand appearing
in terms T4-T7 in an obvious way and baptize it
$C_{ij}$ or $F_{ij}$ depending on to which side in the
decomposition $F^z+C^z$ in~(\ref{eq:A=F+C}) we put it.
The underbracing of summands indicates to which
of the four summands in~(\ref{eq:factor-grad-Bb}) it will be added
to obtain the magnetic Hessian operator $M^z$.

\smallskip
In the following formulas the term $(dt_z\xi)$
is given by~(\ref{eq:dt|_z}).
\newline
\textbf{Term T4.}
$-\tfrac{1}{\norm{z}^4}\left.\tfrac{d}{d\eps}\right|_0
\norm{z_\eps}^4\, b_{t_{z_\eps}}|_{z_\eps}\, j_0 z_\eps^\prime=$
\begin{equation*}
\begin{split}
   &
\underbrace{
   -\tfrac{4 \inner{z}{\xi}\, b_{t_z}|_z}{\norm{z}^2} \, j_0 z^\prime
}_{\tilde T_{41}^z\xi\; (+m_4=T_{41}^z\xi)}
\underbrace{
   -\dot b_{t_{z}}|_{z} \,(dt_z\xi)\, j_0 z^\prime
}_{C_{42}^z \xi}
\underbrace{
   -(db_{t_{z}}|_{z}\xi)\, j_0 z^\prime
}_{C_{43}^z \xi}
\underbrace{
   -4b_{t_z}|_z\, j_0 {\color{orange}\,\p_\tau\xi}
}_{{\color{orange} F_{44}^z\,}\xi}
   .
\end{split}
\end{equation*}
\textbf{Term T5.} $-\tfrac{1}{\norm{z}^4}\left.\tfrac{d}{d\eps}\right|_0
\Abs{z_\eps}^2\norm{z_\eps}^2\dot{\aaa}_{t_{z_\eps}}|_{z_\eps}=$
\begin{equation*}
\begin{split}
   &
\underbrace{
   -\tfrac{2\inner{z_\tau}{\xi_\tau}_0}{\norm{z}^2}\,\dot\aaa_{t_z}|_z
}_{T_{51}^z\xi}
\underbrace{
   -\tfrac{2\abs{z_\tau}^2\inner{z}{\xi}}{\norm{z}^4} \,\dot\aaa_{t_z}|_z
}_{\tilde T_{52}^z\xi\; (+m_5=T_{52}^z\xi)}
\underbrace{
   -\tfrac{\abs{z_\tau}^2}{\norm{z}^2}\,
   \ddot\aaa_{t_z}|_z {\color{red}\,(dt_z\xi)}
}_{T_{53}^z\xi}
\underbrace{
   -\tfrac{\abs{z_\tau}^2}{\norm{z}^2}
   \begin{pmatrix}
      d\dot a^1_{t_z}|_z \xi
      \\
      d\dot a^2_{t_z}|_z \xi
   \end{pmatrix}
}_{T_{54}^z\xi}
   .
\end{split}
\end{equation*}
\textbf{Term T6.} $\tfrac{1}{\norm{z}^4}\left.\tfrac{d}{d\eps}\right|_0
   2 z_\eps\norm{z_\eps}^2
   \int_{\tau}^1
   \inner{\dot{\aaa}_{t_{z_\eps}(\sigma)}|_{z_{\eps,\sigma}}}{z^\prime_{\eps,\sigma}}_0
   \; d\sigma=$
\begin{equation*}
\begin{aligned}
   &\tfrac{2 \xi}{\norm{z}^2}
{\textstyle
   \int_{\tau}^1
}
   \inner{\dot{\aaa}_{t_{z}(\sigma)}|_{z_{\sigma}}}{z^\prime_{\sigma}}_0
   \; d\sigma
   &&\colon T_{61}^z\xi
\\
   &
   +\tfrac{4 z\inner{z}{\xi}}{\norm{z}^4}
{\textstyle
   \int_{\tau}^1
}
   \inner{\dot{\aaa}_{t_{z}(\sigma)}|_{z_{\sigma}}}{z^\prime_{\sigma}}_0
   \; d\sigma
   &&\colon \tilde T_{62}^z\xi \;(+m_6=T_{62}^z)
\\
   &
   +\tfrac{2z}{\norm{z}^2}
{\textstyle
   \int_{\tau}^1
}
   \inner{\ddot{\aaa}_{t_{z}(\sigma)}|_{z_{\sigma}}{\color{red}\,(dt_z\xi)}_\sigma}{z^\prime_{\sigma}}_0
   \; d\sigma
   &&\colon  T_{63}^z\xi
\\
   &
   +\tfrac{2z}{\norm{z}^2}
   \int_{\tau}^1
   \INNER{
   \begin{pmatrix}
      d\dot a^1_{t_z(\sigma)}|_{z_\sigma} \xi_\sigma
      \\
      d\dot a^2_{t_z(\sigma)}|_{z_\sigma} \xi_\sigma
   \end{pmatrix}
   }
   {z^\prime_{\sigma}}_0
   \; d\sigma
   &&\colon  T_{64}^z\xi
\\
   &
   +\tfrac{2z}{\norm{z}^2}
{\textstyle
   \int_{\tau}^1
}
   \inner{\dot{\aaa}_{t_{z}(\sigma)}|_{z_{\sigma}}}{{\color{orange}\,\xi^\prime_{\sigma}}}_0
   \; d\sigma
   &&\colon {\color{orange} T_{65}} \xi
   .
\end{aligned}
\end{equation*}
%
%
%
\textbf{Term T7.} $-\tfrac{1}{\norm{z}^4}\left.\tfrac{d}{d\eps}\right|_0
   2z_\eps\int_0^1
   \int_0^s\Abs{z_{\eps,\sigma}}^2 d\sigma
   \cdot\inner{\dot{\aaa}_{t_{z_\eps}(s)}|_{z_{\eps,s}}}{z^\prime_{\eps,s}}_0\; ds=$
\begin{equation*}
\begin{aligned}
   &-\xi\tfrac{2}{\norm{z}^4} \int_0^1
{\textstyle
   \int_0^s\Abs{z_{\sigma}}^2 d\sigma
}
   \cdot\inner{\dot{\aaa}_{t_{z}(s)}|_{z_{s}}}{z^\prime_{s}}_0\; ds
   &&\colon \tilde T_{71}^z\xi    \;(+m_7=T_{71}^z)
\\
   &-z\tfrac{4}{\norm{z}^4} \int_0^1
{\textstyle
   \int_0^s\inner{z_\sigma}{\xi_\sigma}\, d\sigma
}
   \cdot\inner{\dot{\aaa}_{t_{z}(s)}|_{z_{s}}}{z^\prime_{s}}_0\; ds
   &&\colon  T_{72}^z\xi
\\
   &
   -z\tfrac{2}{\norm{z}^4}\int_0^1
{\textstyle
   \int_0^s\Abs{z_{\sigma}}^2 d\sigma
}
   \cdot\inner{\ddot{\aaa}_{t_{z}(s)}|_{z_{s}}{\color{red}\,(dt_z\xi)}_s}{z^\prime_{s}}_0\;
   ds
   &&\colon  T_{73}^z\xi
\\
   &
   -z\tfrac{2}{\norm{z}^4}\int_0^1
{\textstyle
   \int_0^s\Abs{z_{\sigma}}^2 d\sigma
}
   \cdot\INNER{
   \begin{pmatrix}
      d\dot a^1_{t_z(s)}|_{z_s} \xi_s
      \\
      d\dot a^2_{t_z(s)}|_{z_s} \xi_s
   \end{pmatrix}
   }{z^\prime_{s}}_0\; ds
   &&\colon  T_{74}^z\xi
\\
   &
   -z\tfrac{2}{\norm{z}^4} \int_0^1
{\textstyle
   \int_0^s\Abs{z_{\sigma}}^2 d\sigma
}
   \cdot\inner{\dot{\aaa}_{t_{z}(s)}|_{z_{s}}}{{\color{orange}\,\xi^\prime_{s}}}_0\; ds
   &&\colon {\color{orange} T_{75}^z\xi}
   .
\end{aligned}
\end{equation*}
This finishes Remark~\ref{rem:M_ij}.
\end{remark}

\begin{lemma}\label{le:Hessians-Lagrange}
The Hessian operators of $\Bb=\Kk+\Mm-\Uu$ are given by
\begin{equation*}
\begin{split}
   B^z\xi
   &=\left(T1+T2+T3+T4+T5+T6+T7\right)
   -\tfrac{4\inner{z}{\xi}}{\norm{z}^2}\grad\Bb|_z
\\
   K^z\xi
   &
{\color{gray}\;
   =\left(T1+T2\right)
   -\tfrac{4\inner{z}{\xi}}{\norm{z}^2}\grad\Kk|_z
}
   \\
   &
{\color{gray}\;
   =\left(T1+T2\right)
\underbrace{
   -16 \tfrac{\norm{z^\prime}^2}{\norm{z}^2} \inner{z}{\xi} z
}_{=m_1=-\tilde T_{11}}
   +
\underbrace{
   16 \inner{z}{\xi} z^{\prime\prime}
}_{m_2}
}
   \\
   &=
\underbrace{
   0
}_{T_{11}}
   +
\underbrace{
   8\inner{z^\prime}{\xi^\prime}z
}_{T_{12}}
   +
\underbrace{
   4\norm{z^\prime}^2\xi
}_{T_{13}}
\underbrace{
{\color{red}
   -8\inner{z}{\xi} z^{\prime\prime}
}
}_{T_{21}=\tilde T_{21}+m_2}
\underbrace{ 
{\color{blue}
   -4\norm{z}^2\xi^{\prime\prime}
}
}_{=T_{22}}
\\
   M^z\xi
   &
   =T4+T5+T6+T7+m_4+m_5+m_6+m_7
\\
{\color{cyan}
   -U^z\xi
}
   &
{\color{cyan}\;
   =T3+m_3=\tilde T_{31}^z\xi+m_3
}
   \\
   &=
{\color{cyan}\;
   -\tfrac{2\xi}{\norm{z}^4}
   +
   \tfrac{8\inner{z}{\xi}z}{\norm{z}^6}
   =T_{31}^z\xi
}
\end{split}
\end{equation*}
at $z\in\Ll^\times\Zfrak$ and
where $\grad\Bb|_z$ is given by~(\ref{eq:grad-Bb}).
\end{lemma}

\begin{proof}
Hessian~(\ref{eq:B^z-alternativ}), calculations thereafter,
and gradient formula~(\ref{eq:grad-Bb}).
\end{proof}

\boldmath
\subsection[Without magnetic term]
{Kepler case}
\label{sec:Lag-Kepler}
\unboldmath

\boldmath
\subsubsection{Analytic setup}
\unboldmath

The Kepler case refers to $\Mm=0$.
We restrict ourselves to the case of periodic loops.
The case of twisted periodic loops works analogously.

\begin{definition}[analytic setup in Lagrangian Section~\ref{sec:Hess-Lag}]
\label{def:analytic-setup-Lag}
In the Lagrangian setup target space is the configuration space
instead of the phase space, therefore the target space $\C^2$ has to
be replaced by $\C$.
On the other hand, the Lagrangian case requires an additional
derivative.
Hence the Hilbert space triple in the Lagrangian case is defined by
\begin{equation}\label{eq:bar_+_meeting}
   (\hfrak_0,\hfrak_1,\hfrak_2)=(L^2(\SS^1,\C),W^{2,2}(\SS^1,\C),W^{4,2}(\SS^1,\C))
\end{equation}
and open subsets are defined by
$$
   \ufrak_1:=\{z\in \hfrak_1\mid z\not\equiv 0\wedge
   \forall\tau\in\SS^1\colon z(\tau)\in\Zfrak\}
   ,\qquad
   \ufrak_2:=\ufrak_1\cap \hfrak_2
   .
$$
We denote spaces by lower case letters for distinction from
Section~\ref{sec:Hess-Ham} (uppercase).
\end{definition}

\begin{remark}[twisted loops]
In the case of twisted loops $z(\tau+1)=-z(\tau)$ $\forall
\tau\in\SS^1$ the Hilbert space triple has to be adjusted as follows
\begin{equation*}
\begin{split}
   \hfrak_0^-
   :&=L^2([0,1],\C)
\\
   \hfrak_1^-
   :&=\{z\in W^{1,2}([0,1],\C)\mid z(1)=-z(0)\}
\\
   \hfrak_2^-
   :&=\{z\in W^{2,2}([0,1],\C)\mid z(1)=-z(0)
   \wedge z^\prime(1)=-z^\prime(0)\}
\end{split}
\end{equation*}
Apart from this, the case of twisted loops proceeds completely
analogous as the case of periodic loops.
\end{remark}

\begin{remark}[open problem]
Since in the Lagrangian case the construction of the Hilbert space
triple involves more derivatives,
almost extendability is more involved.
This is already seen in the Kepler case
where in the Hamiltonian setup the Hessian is extendable
whereas in the Lagrangian setup the Hessian is only almost extendable
as we show in this section.

Almost extendability in the general Lagrangian setup
(with magnetic term) is probably still true, but
we did not check it due to the huge number of summands appearing.
\end{remark}

On periodic loop space $\Ll_+^\times\Zfrak$,
defined in~(\ref{eq:twisted-loop-space}), the regularized
Lagrangian action functional $\Bb$
is given by formula~(\ref{eq:Bb}).
The functional $\Bb$ continuously extends to the Sobolev 
$W^{2,2}$-completion $\ufrak_1$ by the same formula.
\\
Let $z\in \ufrak_1$.
As $d\Bb|_z\cdot=\inner{\grad\Bb|_z}{\cdot}\colon \hfrak_1\to\R$
characterizes the $L^2$-gradient,
the Hessian bi-linear form is characterized by
$d^2\Bb|_z(\cdot,\cdot)=\inner{d\grad\Bb|_z\cdot}{\cdot}$.
In words, the derivative of the $L^2$-gradient is the linear operator
$B^z:=d\grad\Bb|_z\colon \hfrak_1\to \hfrak_0$ representing the Hessian
bi-linear form with respect to the $L^2$-inner product.
This is formalized in Appendix~\ref{sec:L2-Hess-symmetry}.
See Theorem~\ref{thm:2nd-derivative} for the next definition.

\begin{definition}\label{def:B^z}
The \textbf{Hessian operator of $\Bb$ at \boldmath$z\in \ufrak_1$}
is the linearization\footnote{
  Given $z\in \ufrak_1$ and $\xi\in \hfrak_1$,
  pick a smooth path of loops $\eps\mapsto z_\eps \in \ufrak_1$
  with $z_0=z$ and $\left.\tfrac{d}{d\eps}\right|_0 z_\eps=\xi$,
  then the linearization is of the form
  $d\grad\Bb|_z\xi=\left.\tfrac{d}{d\eps}\right|_{\eps=0}\grad\Bb(z_\eps)$.
  }
\begin{equation}\label{eq:B^z}
   B^z:=d\grad\Bb|_z\colon \hfrak_1\to \hfrak_0
\end{equation}
of the $L^2$-gradient~(\ref{eq:grad-Bb}) of the function
$\Bb=\Kk-\Uu+\Mm$ in~(\ref{eq:Bb}).
By linearity
$$
   B^z=K^z-U^z+M^z
$$
where the sum consists of the Hessian operators of $\Kk$, $\Uu$, and
$\Mm$; cf.~(\ref{eq:Bb}).
\end{definition}

From now on we consider the Kepler case ($\Mm=0$), in particular
$M=0$.

\boldmath
\subsubsection{Weak Hessian field almost extends}
\label{sec:Lag-Hess-Kepler-ae}
\unboldmath

\begin{theorem}[Kepler case almost extends]
\label{thm:Lag-almex-M=0}
Let $B_0=K{\color{cyan}\,-\,U}\colon \ufrak_1\ni z\mapsto B^z$ be defined
by~(\ref{eq:B^z}) with $\Mm=0$. Then the following holds.
\begin{itemize}\setlength\itemsep{0ex}
\item[\rm (i)]
  $B_0$ defines a weak Hessian field along the set $\ufrak_1$.
\item[\rm (ii)]
  The weak Hessian field $B_0$ almost extends in the sense of
  Definition~\ref{def:alm.ext}.
\end{itemize}
\end{theorem}

We prove the theorem in the form of two separate results.
Part~(i) is Lemma~\ref{le:Lag-weak-nonmag} below,
part~(ii)~is Theorem~\ref{thm:Lag-Hess-nonmag}.

\begin{lemma}\label{le:Lag-weak-nonmag}
The non-magnetic Hessian operators $B^z_0\colon \hfrak_1\to \hfrak_0$
given~by
\begin{equation}\label{le:B^z_0}
   B^z_0\xi
   =
{\color{blue}
   -4\norm{z}^2\xi^{\prime\prime} 
}
   -4\inner{z^{\prime\prime}}{z}\xi
   -8\inner{z^{\prime\prime}}{\xi}z
{\,\color{red}\,
   -8\inner{z}{\xi} z^{\prime\prime}
}
{\color{cyan}
   -\tfrac{2\xi}{\norm{z}^4}
   +\tfrac{8\inner{z}{\xi}z}{\norm{z}^6}
   ,
}
\end{equation}
one for each $z\in \ufrak_1$, determine a weak Hessian field $B_0$ on $\ufrak_1$.
\end{lemma}

\begin{proof}
The formula for $B_0^z$ is due to Lemma~\ref{le:Hessians-Lagrange}
where, in addition, we applied to $T_{12}$ and $T_{13}$ integration by
parts (boundary terms vanishes by periodicity).
By 
Theorem~\ref{thm:Behz-Hol} in the form $W^{2,2}\times L^2\to L^2$
and $W^{4,2}\times W^{2,2}\to W^{2,2}$ 
the map $z\mapsto B^z_0$ is
element of the spaces $C^0(\ufrak_1,\Ll(\hfrak_1,\hfrak_0))$ and
$C^0(\ufrak_2,\Ll(\hfrak_2,\hfrak_1))$.
\\
We verify the (\texttt{Symmetry}) axiom~(\ref{eq:u-0-symmetric}):
Pick $\xi,\eta\in \hfrak_0$, then
\begin{equation*}
\begin{split}
   \inner{B^z_0\xi}{\eta}
   &=\inner{
   -4\norm{z}^2\xi^{\prime\prime}
   -4\inner{z^{\prime\prime}}{z}\xi
   -8\inner{z^{\prime\prime}}{\xi}z
   -8\inner{z}{\xi} z^{\prime\prime}
   -\tfrac{2\xi}{\norm{z}^4}
   +\tfrac{8\inner{z}{\xi} z}{\norm{z}^6}
   }
   {\eta}
\\
   &=
   -4\norm{z}^2\inner{\xi}{\eta^{\prime\prime}}
   -4\inner{z^{\prime\prime}}{z}\inner{\xi}{\eta}
   -8\inner{z^{\prime\prime}}{\xi}\inner{z}{\eta}
   -8\inner{z}{\xi} \inner{z^{\prime\prime}}{\eta}
   \\
   &\quad\,
   -\tfrac{2\inner{\xi}{\eta}}{\norm{z}^4}
   +\tfrac{8\inner{z}{\xi} \inner{z}{\eta}}{\norm{z}^6}
\\
   &=\inner{\xi}{B^z_0\eta}
   .
\end{split}
\end{equation*}
Concerning equality two, to summand one we applied twice integration
by parts. Inequality three holds by symmetry of  inner products
and by interchanging the two summands with factor 8.
\\
It remains to check the (\texttt{Fredholm}) axiom.
The following operator as a map
\begin{equation*}
   -\p_\tau\p_\tau\colon\quad W^{2,2}\to L^2
   ,\quad W^{4,2}\to W^{2,2}
   ,\qquad \forall z\in \ufrak_1
\end{equation*}
is positive semi-definite, symmetric, and $\ker\p_\tau\p_\tau=\C$
are the constant maps only
(by periodicity affine maps are excluded).
By symmetry, kernel and cokernel coincide,
so $-\p_\tau\p_\tau$ is Fredholm of index zero.
The factor $4\norm{z}^2$ is non-zero since $z\in \ufrak_1$
is not the zero map, hence the \textbf{leading term}
$-4\norm{z}^2\p_\tau\p_\tau$ in the formula for $B_0^z$ is
Fredholm of index zero as well.
\\
In the formula for $B_0^z$ remain five summands.
We separate $-8\inner{z}{\xi} z^{\prime\prime}_\tau$ for later recycling
in Theorem~\ref{thm:Lag-Hess-nonmag} Step~(\texttt{F}).
The sum of the other four, denoted by
\begin{equation}\label{eq:T^z}
   T^z\colon \xi\mapsto
   \left[
   \tau\mapsto
   -4\inner{z^{\prime\prime}}{z}\xi_\tau
   -8\inner{z^{\prime\prime}}{\xi}z_\tau
   %
   -\tfrac{2\xi_\tau}{\norm{z}^4}\tfrac{\norm{z}^2}{\norm{z}^2}
   +\tfrac{8\inner{z}{\xi}z_\tau}{\norm{z}^6}
   \right]
   ,
\end{equation}
is bounded as a map $\hfrak_0\to \hfrak_0$ and even as a map $\hfrak_1\to \hfrak_1$
if $z\in \ufrak_1$. Indeed
\begin{equation*}
\begin{split}
   \norm{T^z\xi}_2
   &\le 8\norm{z}_{2,2}^2
   \left(
   \tfrac12+1+\tfrac{1}{4\norm{z}_2^6} +\tfrac{1}{\norm{z}_2^6}
   \right)
   \norm{\xi}_2
\\
   \norm{(T^z\xi)^{\prime\prime}}_2
   &=\Norm{
   -4\inner{z^{\prime\prime}}{z}\xi^{\prime\prime}
   -8\inner{z^{\prime\prime}}{\xi}z^{\prime\prime}
   %
   -\tfrac{2\xi^{\prime\prime}}{\norm{z}^4}\tfrac{\norm{z}^2}{\norm{z}^2}
   +\tfrac{8\inner{z}{\xi}z^{\prime\prime}}{\norm{z}^6}
   }_2
\\
   &\le 8\norm{z}_{2,2}^2
   \left(
   \tfrac12+1+\tfrac{1}{4\norm{z}_2^6} +\tfrac{1}{\norm{z}_2^6}
   \right)
   \norm{\xi}_{2,2}
   .
\end{split}
\end{equation*}
So, by compactness of the embeddings
$\hfrak_1\INTO \hfrak_0$ and $\hfrak_2\INTO \hfrak_1$, both operators $T^z$, as $\hfrak_1\to \hfrak_0$
and as $\hfrak_2\to \hfrak_1$, are compact.
But Fredholm property and index are stable under compact perturbation.
So both operators
\begin{equation}\label{eq:F^z}
   F^z\colon \xi\mapsto -4\norm{z}^2\xi^{\prime\prime}+T^z\xi
   \colon\quad \hfrak_1\to \hfrak_0
   ,\quad \hfrak_2\to \hfrak_1
   ,\qquad \forall z\in \ufrak_1
\end{equation}
are Fredholm of index zero where $T^z$ is defined by~(\ref{eq:T^z}).
It remains to exhibit the yet missing summand
as a compact perturbation of $F^z$.
Indeed we estimate
\begin{equation*}
\begin{split}
   \norm{8\inner{z}{\xi} z^{\prime\prime}}_2
   &\le
   8\norm{z}_2\norm{\xi}_2\norm{z^{\prime\prime}}_2\,\;
   \le 8\norm{z}_{2,2}^2\norm{\xi}_2
\\
   \norm{(8\inner{z}{\xi} z^{\prime\prime})^{\prime\prime}}_2
   &\le
   8\norm{z}_2\norm{\xi}_2\norm{z^{\prime\prime \prime\prime}}_2
   \le 8\norm{z}_{4,2}^2\norm{\xi}_2
   .
\end{split}
\end{equation*}
This shows that $\xi\mapsto 8\inner{z}{\xi} z^{\prime\prime}$
is bounded as a map $\hfrak_0\to \hfrak_0$ if $z\in \ufrak_1$
and as a map $\hfrak_1\INTO \hfrak_0\to \hfrak_1$ if $z\in \ufrak_2$.
Now the compactness and perturbation argument previous
to~(\ref{eq:F^z}) proves the (\texttt{Fredholm}) axiom
and Lemma~\ref{le:Lag-weak-nonmag}.
\end{proof}

\begin{theorem}\label{thm:Lag-Hess-nonmag}
The non-magnetic weak Hessian field $B$ is almost extendable.
Moreover, the pair $(F,C)$ defined for $z\in \ufrak_1$ by
\begin{equation}\label{eq:Lag-Hess-nonmag}
\begin{split}
   F^z\xi:&=
{\color{blue}
   -4\norm{z}^2\xi^{\prime\prime} 
}
   -4\inner{z^{\prime\prime}}{z}\xi
   -8\inner{z^{\prime\prime}}{\xi}z
{\color{cyan}\,
   -\tfrac{2\xi}{\norm{z}^4}
   +\tfrac{8\inner{z}{\xi}z}{\norm{z}^6}
}
\\
   C^z\xi:&=
{\color{red}
   -8\inner{z}{\xi} z^{\prime\prime}
}
\end{split}
\end{equation}
is a decomposition~(\ref{eq:A=F+C}).
\end{theorem}

\begin{proof}[Proof of Theorem~\ref{thm:Lag-Hess-nonmag}]
The proof has four steps~1, 2, (\texttt{C}), and (\texttt{F}).

\smallskip\noindent
\textbf{Step~1.}
By~(\ref{eq:Behz-Hol-2})
the map $u\mapsto C^u$ is element of the spaces
$C^0(\ufrak_1,\Ll(\underline{\hfrak_r},\hfrak_0))$ and $C^0(\ufrak_2,\Ll(\underline{\hfrak_1},\hfrak_1))$
where $\hfrak_r=W^{2r,2}(\SS^1,\R^{2})$ for some $r\in(\frac12,1)$.

\begin{proof} Given $z\in \ufrak_1$, the estimate
$$
   \norm{C^z\xi}_2
   =\norm{8\inner{z}{\xi} z^{\prime\prime}}_2
   \le 8\norm{z}_2\norm{\xi}_2\norm{z^{\prime\prime}}_2
   \le 8\norm{z}_{2,2}^2\norm{\xi}_2
$$
shows that $C$ is continuous as a map from $\ufrak_1$ even to
$\Ll(\hfrak_0,\hfrak_0)$, so to $\Ll(\hfrak_r,\hfrak_0)$ whenever $r\in[0,1)$ due to
the embedding $\hfrak_r\INTO \hfrak_0$.
Given $z\in \ufrak_2$, then
$$
   \norm{(C^z\xi)^{\prime\prime}}_2
   =\norm{8\inner{z}{\xi} z^{\prime\prime\prime\prime}}_2
   \le 8\norm{z}_2\norm{\xi}_2\norm{z^{\prime\prime\prime\prime}}_2
   \le 8\norm{z}_{4,2}^2\norm{\xi}_2
$$
shows that $C$ is continuous as a map from $\ufrak_2$ even to
$\Ll(\hfrak_0,\hfrak_1)$, hence to $\Ll(\hfrak_1,\hfrak_1)$ due to
the embedding $\hfrak_1\INTO \hfrak_0$.
\end{proof}

\smallskip\noindent
\textbf{Step~2.}
We need to show that the map $F\colon z\mapsto F^z$
is element of the space $C^0(\ufrak_1,\Ll(\hfrak_1,\hfrak_0)\cap\Ll(\hfrak_2,\hfrak_1))$;
see~(\ref{eq:hyp-F}).

\begin{proof}
We focus on $C^0(\ufrak_1,\Ll(\hfrak_2,\hfrak_1))$. Given $z\in \ufrak_1$, then the estimate
\begin{equation*}
\begin{split}
   \norm{(F^z\xi)^{\prime\prime}}_2
   &=\Norm{
   -4\norm{z}_2^2 \xi^{\prime\prime\prime\prime}
   -4\inner{z^{\prime\prime}}{z}\xi^{\prime\prime}
   -8\inner{z^{\prime\prime}}{\xi}z^{\prime\prime}
   -\tfrac{2\xi^{\prime\prime}}{\norm{z}^4}\tfrac{\norm{z}^2}{\norm{z}^2}
   +\tfrac{8\inner{z}{\xi}z^{\prime\prime}}{\norm{z}^6}
   }_2
\\
   &\le 8\norm{z}_{2,2}^2
   \left(
   \tfrac12+\tfrac12+1+\tfrac{1}{4\norm{z}_2^6} +\tfrac{1}{\norm{z}_2^6}
   \right)
   \norm{\xi}_{4,2}
\end{split}
\end{equation*}
and an analogous estimate for the simpler case $\norm{F^z\xi}_2$
show that $F$ is continuous as a map from $\ufrak_1$ to $\Ll(\hfrak_2,\hfrak_1)$
and also to $\Ll(\hfrak_1,\hfrak_0)$.
\end{proof}

\medskip\noindent
\textbf{Step~(\texttt{C}).} The map $z\mapsto C^z$ satisfies the scale
Lipschitz estimate~(\ref{eq:(C)}).

\begin{proof}
Fix a loop $z\in \ufrak_1=W^{2,2}(\SS^1,\Zfrak)\subset W^{2,2}(\SS^1,\C)=\hfrak_1$.
By continuity of $z$ its image is compact.
So there exists a bounded open neighborhood $\Ufrak$
of $\im z$ in $\Zfrak$. Let $\kappa_0$ be a bound of $\Ufrak$.
An open neighborhood of $z$ in $\ufrak_1$ is defined by
$$
   \mathrm{v}_z:=W^{2,2}(\SS^1,\Ufrak)
   \supset W^{2,2}(\SS^1,\Zfrak)=\ufrak_1 \ni z
   .
$$
The $L^\infty$ norm of the elements of $\mathrm{v}_z$ is
bounded by $\kappa_0$.
Let $\xi\in \hfrak_1$ and $v,w\in\mathrm{v}_z$.
In the following we use the norm
$\norm{f}_2^2+\norm{f^{\prime\prime}}_2^2$
which is equivalent to the $W^{2,2}$ norm
$\norm{f}_2^2+\norm{f^{\prime}}_2^2+\norm{f^{\prime\prime}}_2^2$.
Adding twice \underline{zero} we estimate
\begin{equation*}
\begin{split}
   \tfrac{1}{8^2}\norm{(C^v-C^w)\xi}_{2,2}^2
   &=\norm{
   \inner{v}{\xi} v^{\prime\prime}-\inner{w}{\xi} w^{\prime\prime}
   }_{2}^2
   +\norm{
   \inner{v}{\xi} v^{\prime\prime \prime\prime}
   -\inner{w}{\xi} w^{\prime\prime\prime\prime}
   }_{2}^2
\\
   &\le
   2\norm{
   \inner{v\underline{\,-\,w}}{\xi} v^{\prime\prime}
   }_{2}^2
   +2\norm{
   \inner{w}{\xi}\left(\underline{v^{\prime\prime}}-w^{\prime\prime} \right)
   }_{2}^2
   \\
   &\quad
   +2\norm{
   \inner{v\underline{\,-w\,}}{\xi} v^{\prime\prime \prime\prime}
   }_{2}^2
   +2\norm{
   \inner{w}{\xi}
   \left(\underline{v^{\prime\prime \prime\prime}}-w^{\prime\prime \prime\prime}\right)
   }_{2}^2
\\
   &\le
   2 \inner{v-w}{\xi}^2\norm{v^{\prime\prime}}_{2}^2
   +2 \inner{w}{\xi}^2\norm{v^{\prime\prime} -w^{\prime\prime}}_{2}^2
   \\
   &\quad
   +2 \inner{v-w}{\xi}^2\norm{v^{\prime\prime \prime\prime}}_{2}^2
   +2 \inner{w}{\xi}^2
   \norm{v^{\prime\prime \prime\prime}-w^{\prime\prime \prime\prime}}_{2}^2
\\
   &\le
   2\norm{v-w}_2^2\norm{\xi}_2^2\norm{v}_{4,2}^2
   +2\norm{w}_2^2\norm{\xi}_2^2\norm{v-w}_{4,2}^2
\\
   &\le
   2\abs{v-w}_{\hfrak_1}^2\abs{\xi}_{\hfrak_1}^2\abs{v}_{\hfrak_2}^2
   +2\norm{w}_\infty^2\abs{\xi}_{\hfrak_1}^2\abs{v-w}_{\hfrak_2}^2
\\
   &\le
   2\abs{v-w}_{\hfrak_1}^2\abs{\xi}_{\hfrak_1}^2\abs{v}_{\hfrak_2}^2
   +2\kappa_0^2\abs{\xi}_{\hfrak_1}^2\abs{v-w}_{\hfrak_2}^2
   .
\end{split}
\end{equation*}
To get inequality three we apply Cauchy-Schwarz and
in addition we incorporate the first two summands of four
into the estimate of the last two summands.
\newline
With
$
   \kappa
   :=8\sqrt{2}\max\{1,\kappa_0\}
$
this implies that the operator norm is bounded by
$$
   \norm{C^v-C^w}_{\Ll(\hfrak_1)}
   \le\kappa \Bigl(\abs{v-w}_{\hfrak_2}+\abs{v}_{\hfrak_2}\abs{v-w}_{\hfrak_1}\Bigr).
$$
Interchanging the roles of $v$ and $w$ we get
$$
   \norm{C^v-C^w}_{\Ll(\hfrak_1)}
   \le\kappa \Bigl(\abs{v-w}_{\hfrak_2}+\abs{w}_{\hfrak_2}\abs{v-w}_{\hfrak_1}\Bigr).
$$
The above two estimates imply the scale Lipschitz estimate
\begin{equation}\label{eq:C-paraDarboux}
   \norm{C^v-C^w}_{\Ll(\hfrak_1)}
   \le\kappa \Bigl(\abs{v-w}_{\hfrak_2}
   +\min\{\abs{v}_{\hfrak_2},\abs{w}_{\hfrak_2}\}\cdot\abs{v-w}_{\hfrak_1}
   \Bigr)
\end{equation}
which is precisely~(\ref{eq:(C)}) in axiom~(\texttt{C}).
This proves Step~(\texttt{C}).
\end{proof}

\smallskip\noindent
\textbf{Step~(\texttt{F}).} $\forall z\in \ufrak_1\colon$
$F^z_2:=F^z|_{\hfrak_2}\colon \hfrak_2\to \hfrak_1$ is Fredholm of index zero.

\smallskip\noindent
The proof of Step~(\texttt{F}) was given in~(\ref{eq:F^z}).
%
This proves Theorem~\ref{thm:Lag-Hess-nonmag}.
\end{proof}

\boldmath
\section{Estimates for Hessian operator summands}
\label{sec:ests-summands-Hess-mag-Ham}
\unboldmath

In Appendix~\ref{sec:ests-summands-Hess-mag-Ham}
we analyze the summands of the Hessian operator~$B^z$ in
Lemma~\ref{le:Hessians-Lagrange}.
As a preparation we collect in~\ref{sec:util} a number of estimates.

In the following we suppose that $z$ and $\xi$ are smooth,
so the calculations and estimates are justified
and we can extract in each case the information in which Sobolev
spaces $z$ and $\xi$ must at least be.
By $\inner{\cdot}{\cdot}$ we denote the $L^2$-inner product,
by $\norm{\cdot}_{\rho}$ the $L^\rho=:H_\rho$ norm,
and by $\norm{\cdot}_{k,2}$ the Sobolev $W^{k,2}=:H_k$ norm.
We set $z_\tau:=z(\tau)$. 
For Banach spaces $X$ and $Y$ let
$
   \Ll(X,Y)
$,
and
$
   \mbf{\Cc(X,Y)}
$
be the space of bounded, respectively \textbf{compact}, linear operators.

\boldmath
\subsection{Utilities}
\label{sec:util}
\unboldmath

\textbf{Factor \boldmath$j_0$.}
Since $j_0$ is rotation by $\frac{\pi}{2}$ it
is an isometry, so it disappears in norms.

\boldmath
\subsubsection{Fractional Sobolev spaces}
\unboldmath

\begin{theorem}[Multiplication map, {\cite[Thm.\,7.4]{Behzadan:2021a}}]
\label{thm:Behz-Hol}
Assume $\rho_1, \rho_2, \rho\ge0$ are real and satisfy
  $\rho_1\ge \rho$ and $\rho_2\ge \rho$ and
  $\rho_1+\rho_2>\frac12 +\rho$.
Then the following is true.
If $v\in W^{\rho_1,2}(\SS^1)$ and $w\in W^{\rho_2,2}(\SS^1)$, then
$vw\in W^{\rho,2}(\SS^1)$ and pointwise multiplication of functions is a
continuous bi-linear map
\begin{equation*}
   W^{\rho_1,2}(\SS^1)\times W^{\rho_2,2}(\SS^1)\to W^{\rho,2}(\SS^1) .
\end{equation*}
\end{theorem}
Most important for us is the theorem in the form
\begin{align}
   \label{eq:Behz-Hol-2}
   W^{r,2}\times L^2\stackrel{r>\frac12}{\longrightarrow} L^2
   ,\qquad
   W^{1,2}\times W^{1,2}\to W^{1,2}
   .
\end{align}
The \textbf{Sobolev inequality} {\color{gray}2}
on $\SS^1$ asserts,\footnote{
  for a proof with constant $1$ see e.g.~\cite[\S A.6 Sobolev
  embedding]{Frauenfelder:2025e}
  }
respectively implies, that
$$
   \norm{z}_2
   \le\norm{z}_\infty
   \stackrel{{\color{gray}2}}{\le}\norm{z}_{1,2}
   ,\qquad
   \norm{z^\prime}_\infty\le\norm{z^\prime}_{1,2}
   \le\norm{z}_{2,2}
   .
$$
All norms are on the domain $\SS^1$, unless
indicated differently.
In fact, not only $H_1$ embedds into $C^0$, but even the larger
spaces $H_r$ whenever $r$ is strictly larger then the borderline
value $\frac12$.
This improvement is crucial for us, see e.g. term $C_{43}$
further below.

\begin{proposition}\label{prop:Taylor}
For any $\alpha\in(0,1)$ the inclusion map
$H_{\frac{1}{2}+\alpha}(\R)\subset C^{0,\alpha}(\R)$ into the space of
\textbf{\boldmath$\alpha$-H\"older continuous} maps is continuous.
Here
$$
   u\in C^{0,\alpha}(\R)\quad\Leftrightarrow\quad
   \text{$u$ bounded and $\exists C\colon
   \abs{u(x+y)-u(x)}\le C\abs{y}^\alpha$ $\forall x,y$}.
$$
\end{proposition}

\begin{proof}
See e.g.~\cite[Ch.\,4 Prop.\,1.5]{Taylor:1996a}.
\end{proof}

By Proposition~\ref{prop:Taylor}, given $r\in(\frac12,1)$, there
exists a constant $c_r$ such that
$
   \norm{\cdot}_{C^0}
   \le \norm{\cdot}_{C^{0,r-1/2}}
   \le c_\alpha \norm{\cdot}_{H_{1/2+\alpha}}
$
on $\SS^1$.
Writing $\norm{\cdot}_\infty=\norm{\cdot}_{C^0}$
and $r=\frac12+\alpha$, then for $r\in(\frac12,1]$ it holds that
\begin{equation}\label{eq:H_r-est}
   \norm{\cdot}_2
   \le\norm{\cdot}_\infty
   \le c_\alpha \norm{\cdot}_{r,2}
   ,\qquad
   \norm{\cdot}_{r,2}
   \le\norm{\cdot}_{1,2}
   .
\end{equation}

\boldmath
\subsubsection{Factor $t_z$}\label{sec:t_z}
\unboldmath

Recall from~(\ref{eq:class-time-tau-C})
the Barutello-Ortega-Verzini reparametrization
\begin{equation}\label{eq:class-time-tau-deriv-C}
   t_z(\tau)
   :=\frac{\int_0^\tau \Abs{z(s)}^2\, ds}{\norm{z}^2}
   ,\qquad
   t_z^\prime(\tau)
   =\frac{\Abs{z(\tau)}^2}{\norm{z}^2}
   ,\qquad
   t_z\in C^1
   .
\end{equation}
which takes a loop $z$ to a map
$t_z:=t(z)\colon\SS^1\to\SS^1$, $\tau\mapsto t_z(\tau)$.
There are two obvious identities
\begin{equation}\label{eq:dt|_z}
\begin{split}
   \tfrac{d}{d\tau} t_z(\tau)
   &=
   \frac{\Abs{z(\tau)}^2}{\norm{z}^2}
\\
   (dt|_z\xi)_\tau
   &=
   \frac{2}{\norm{z}^2}\int_0^\tau\INNER{z_\sigma}{\xi_\sigma} d\sigma
   -\frac{2\INNER{z}{\xi}}{\norm{z}^4}\int_0^\tau\Abs{z_\sigma}^2 d\sigma
   .
\end{split}
\end{equation}
We estimate
\begin{equation*}
\begin{split}
   \norm{t_z^\prime}_2
   &
   \stackrel{\text{(\ref{eq:class-time-tau-deriv-C})}}{=}
   \Norm{\tfrac{\Abs{z(\cdot)}^2}{\norm{z}_2^2}}_2
   \le\frac{\norm{z}_\infty\norm{z}_2}{\norm{z}_2^2}
   \le\frac{\norm{z}_{1,2}}{\norm{z}_2}
\end{split}
\end{equation*}
and
\begin{equation*}
\begin{split}
   \norm{t_z^\prime}_\infty:
   &=\sup_{\tau\in\SS^1}\Abs{\tfrac{d}{d\tau} t_z(\tau)}
   \stackrel{\text{(\ref{eq:class-time-tau-deriv-C})}}{=}
   \sup_{\tau\in\SS^1}\frac{\Abs{z(\tau)}^2}{\norm{z}_2^2}
   \le\frac{\norm{z}_\infty^2}{\norm{z}_2^2}
   \le\frac{\norm{z}_{1,2}^2}{\norm{z}_2^2}
   .
\end{split}
\end{equation*}
By~(\ref{eq:dt|_z}) we have and by the triangle and Cauchy-Schwarz
inequalities we get
\begin{equation}\label{eq:dt|_z-2-infty}
\begin{split}
   \norm{(dt|_z\xi)}_2
   &\le\norm{(dt|_z\xi)}_{{\color{red}\infty}}
\\
   &=\sup_{\tau\in\SS^1}\Abs{
   \tfrac{2}{\norm{z}_2^2}\int_0^\tau\INNER{z_\sigma}{\xi_\sigma}_0 d\sigma
   -\tfrac{2\INNER{z}{\xi}}{\norm{z}_2^4}\int_0^\tau\Abs{z_\sigma}^2 d\sigma
   }
\\
   &\le 
   \tfrac{2}{\norm{z}_2^2}
   \sup_{\tau\in\SS^1}\Abs{\INNER{z}{\xi}_{L^2_{0,\tau}}}
   +\tfrac{2\norm{z}_2\norm{\xi}_2}{\norm{z}_2^4}
   \sup_{\tau\in\SS^1}\Abs{\INNER{z}{z}_{L^2_{0,\tau}}}
\\
   &\le \tfrac{2}{\norm{z}_2^2}
   \sup_{\tau\in\SS^1}\left(\norm{z}_{L^2_{0,\tau}}\norm{\xi}_{L^2_{0,\tau}}\right)
   +\tfrac{2\norm{\xi}_2}{\norm{z}_2^3}
   \sup_{\tau\in\SS^1}\norm{z}_{L^2_{0,\tau}}^2
\\
   &\le
   \tfrac{2}{\norm{z}_2}\norm{\xi}_2
   +\tfrac{2}{\norm{z}_2}\norm{\xi}_2
\\
   &\le
   \tfrac{4}{\norm{z}_2}\norm{\xi}_{{\color{red} 2}}
   .
\end{split}
\end{equation}
Differentiate~(\ref{eq:dt|_z}) to get equality one here
\begin{equation*}
\begin{split}
   \norm{(dt|_z\xi)^\prime}_2
   &=\Norm{\tfrac{2}{\norm{z}_2^2}\INNER{z_\cdot}{\xi_\cdot}_0
   -\tfrac{2\INNER{z}{\xi}}{\norm{z}_2^4}\Abs{z_\cdot}^2}_2
\\
   &\le\tfrac{2}{\norm{z}_2^2}
   \Norm{\INNER{z_\cdot}{\xi_\cdot}_0}_2
   +\tfrac{2\norm{z}_\infty\norm{\xi}_2}{\norm{z}_2^4}
   \norm{\inner{z_\cdot}{z_\cdot}_0}_2
\\
   &\le \tfrac{2}{\norm{z}_2^2}\norm{z}_\infty\norm{\xi}_2
   +\tfrac{2\norm{z}_{1,2}\norm{\xi}_2}{\norm{z}_2^4}
   \norm{z}_\infty\norm{z}_2
\\
   &\le\left(
   \tfrac{2 \norm{z}_{1,2}}{\norm{z}_2^2}
   +\tfrac{2 \norm{z}_{1,2}^2}{\norm{z}_2^3}
   \right) \norm{\xi}_2
\\
  :&= d^\prime\tfrak^{z_{1,2}} \norm{\xi}_{{\color{red} 2}}
\end{split}
\end{equation*}
and to get equality one here
\begin{equation}\label{eq:dt_z-xi-infty}
\begin{split}
   \norm{(dt|_z\xi)^\prime}_{{\color{red} \infty}}
   &=
   \sup_{\tau\in\SS^1}\Abs{\tfrac{2}{\norm{z}_2^2}\INNER{z_\tau}{\xi_\tau}
   -\tfrac{2\INNER{z}{\xi}}{\norm{z}_2^4}\Abs{z_\tau}^2}
\\
   &\le \tfrac{2}{\norm{z}_2^2}\norm{z}_\infty\norm{\xi}_\infty
   +\tfrac{2\norm{\xi}_2}{\norm{z}_2^3}\norm{z}_\infty^2
\\
   &
   \le
   \tfrac{2 \norm{z}_{1,2}}{\norm{z}_2^2}\norm{\xi}_{1,2}
   +\tfrac{2 \norm{z}_{1,2}^2}{\norm{z}_2^3}\norm{\xi}_2
\\
   &\le d^\prime\tfrak^{z_{1,2}} \norm{\xi}_{{\color{red} 1,2}}
   .
\end{split}
\end{equation}
By definition~(\ref{eq:class-time-tau-C}) of $t_z$ and adding zero we
estimate
\begin{equation}\label{eq:t_z-diff}
\begin{split}
   \Abs{t_v(\tau)-t_w(\tau)}
   &=\Abs{
   \tfrac{\int_0^\tau \Abs{v_s}^2\, ds}{\norm{v}_2^2}
{\color{gray}\;
   -\;
   \tfrac{\int_0^\tau \Abs{w_s}^2\, ds}{\norm{v}_2^2}
   +
   \tfrac{\int_0^\tau \Abs{w_s}^2\, ds}{\norm{v}_2^2}
}
   -
   \tfrac{\int_0^\tau \Abs{w_s}^2\, ds}{\norm{w}_2^2}
   }
\\
   &\le\tfrac{1}{\norm{v}_2^2}
   \Abs{\int_0^\tau \Abs{v_s}^2
{\color{gray}\;
   -\langle v_s,w_s\rangle_0+\langle v_s,w_s\rangle_0
}
   -\Abs{w_s}^2\, ds}
   \\
   &\quad
   +\int_0^\tau \Abs{w_s}^2\, ds
   \Abs{\tfrac{\norm{w}_2^2
{\color{gray}\;
   -\langle v,w\rangle+\langle v,w\rangle
}
   -\norm{v}_2^2}{\norm{v}_2^2\norm{w}_2^2}}
\\
   &=\tfrac{1}{\norm{v}_2^2}
   \Abs{\int_0^\tau 
   \langle v_s,v_s-w_s\rangle_0+\langle v_s-w_s,w_s\rangle_0
   \, ds}
   \\
   &\quad
   +\int_0^\tau \Abs{w_s}^2\, ds
   \Abs{\tfrac{
   \langle v,v-w\rangle+\langle v-w,w\rangle
   }{\norm{v}_2^2\norm{w}_2^2}}
\\
   &\stackrel{{\color{gray}3}}{\le} 2
   \tfrac{\norm{v}_2\norm{v-w}_2+\norm{v-w}_2\norm{w}_2}{\norm{v}_2^2}
   =2\norm{v-w}_2
   \tfrac{\norm{v}_2+\norm{w}_2}{\norm{v}_2^2}
   .
\end{split}
\end{equation}
In step {\color{gray}3} , on summand one, we take the absolute value
inside the integral, then enlarge $\int_0^\tau$ to $\int_0^1$,
then we apply Cauchy-Schwarz.
On summand two we enlarge $\int_0^\tau$ to $\int_0^1$ in which case
$\norm{w}_2^2$ cancels.

\medskip
\noindent
We need in~(\ref{eq:D43-42}) the following estimate.
Differentiate~(\ref{eq:dt|_z}) to get step 1.
Step 2 is by the triangle inequality
and produces a sum of two absolute values.
In summand one add zero
and in summand two add two times zero, then apply the triangle
inequality to get step~3 where in addition we insert
{\color{gray}three zeroes}
\begin{equation}\label{eq:dt_z-deriv-difference}
\begin{split}
   &
   \tfrac12
   \abs{(dt_v\xi)^\prime_\tau-(dt_w\xi)^\prime_\tau}
\\
   &\stackrel{{\color{gray}1}}{=}
   \Abs{
   \tfrac{1}{\norm{v}_2^2}\INNER{v_\tau}{\xi_\tau}_0
   -
   \tfrac{\INNER{v}{\xi}}{\norm{v}_2^4}\Abs{v_\tau}^2
   -
   \tfrac{1}{\norm{w}_2^2}\INNER{w_\tau}{\xi_\tau}_0
   +
   \tfrac{\INNER{w}{\xi}}{\norm{w}_2^4}\Abs{w_\tau}^2
   }
\\
   &\stackrel{{\color{gray}2}}{\le}
   \Abs{
   \tfrac{1}{\norm{v}_2^2}\INNER{v_\tau}{\xi_\tau}_0
   -
   \tfrac{1}{\norm{w}_2^2}\INNER{w_\tau}{\xi_\tau}_0
   }
   +
   \Abs{
   \tfrac{\INNER{v}{\xi}}{\norm{v}_2^4}\Abs{v_\tau}^2
   -
   \tfrac{\INNER{w}{\xi}}{\norm{w}_2^4}\Abs{w_\tau}^2
   }
\\
   &\stackrel{{\color{gray}3}}{\le}
   \Abs{
   \tfrac{\norm{w}_2^2
{\color{gray}
   -\inner{w}{v}+\inner{w}{v}
}
   -\norm{v}_2^2}{\norm{v}_2^2 \norm{w}_2^2}
   }
   \Abs{\INNER{v_\tau}{\xi_\tau}_0}
   +
   \Abs{
   \INNER{v_\tau-w_\tau}{\xi_\tau}_0
   }
   \tfrac{1}{\norm{w}_2^2}
   \\
   &\quad
   +
   \Abs{\INNER{v-w}{\xi}}
   \tfrac{\Abs{v_\tau}^2}{\norm{v}_2^4}
   \\
   &\quad
   +
   \Abs{
   \tfrac{\norm{w}_2^4 
{\color{gray}
   -\norm{w}_2^2 \norm{v}_2^2 +\norm{w}_2^2 \norm{v}_2^2
}
   -\norm{v}_2^4}
   {\norm{v}_2^4 \norm{w}_2^4}
   }\Abs{v_\tau}^2 \Abs{\INNER{w}{\xi}}
   \\
   &\quad
   +\tfrac{\Abs{\INNER{w}{\xi}}}{\norm{w}_2^4}
   \Abs{
   \Abs{v_\tau}^2
{\color{gray}
   -\inner{w_\tau}{v_\tau}_0+\inner{w_\tau}{v_\tau}_0
}
   -
   \Abs{w_\tau}^2
   }
\\
   &\stackrel{{\color{gray}4}}{\le}
   \norm{v-w}_2
   \tfrac{\norm{v}_2+\norm{w}_2}
   {\norm{v}_2^2 \norm{w}_2^2}
   \cdot\abs{v_\tau}\cdot\abs{\xi_\tau}
   +
   \abs{v_\tau-w_\tau}\cdot
   \tfrac{\abs{\xi_\tau}}{\norm{w}_2^2}
   \\
   &\quad
   +
   \norm{v-w}_2\norm{\xi}_2
   \tfrac{\Abs{v_\tau}^2}{\norm{v}_2^4}
   \\
   &\quad
   +
   \norm{v-w}_2\left(\norm{v}_2+\norm{w}_2\right)
   \left(\norm{v}_2^2+\norm{w}_2^2\right)
   \Abs{v_\tau}^2 \norm{w}_2\norm{\xi}_2
   \\
   &\quad
   +\tfrac{\norm{\xi}_2}{\norm{w}_2^3}
   \Abs{v_\tau-w_\tau}\left(\Abs{v_\tau}+\Abs{w_\tau}\right)
\\
\end{split}
\end{equation}
pointwise for $\tau\in\SS^1$.
In step~4 we used the Cauchy-Schwarz inequality
for the Euclidean inner product $\inner{\cdot}{\cdot}_0$ on $\R^2$
and for the $L^2(\SS^1,\R^2)$ inner product $\inner{\cdot}{\cdot}$.
Take the supremum over $\tau$ and simplify
with $\norm{\cdot}_2\le\norm{\cdot}_\infty$
we get the estimate
\begin{equation}\label{eq:dt_z-deriv-difference-infty}
\begin{split}
   &
   \tfrac12
   \norm{(dt_v\xi)^\prime-(dt_w\xi)^\prime}_\infty
\\
   &\le
   \norm{v-w}_\infty\norm{\xi}_\infty\cdot
   \\
   &\quad
   \left(
   \tfrac{\left(\norm{v}_\infty+\norm{w}_\infty\right)^2}
      {\norm{v}_2^2 \norm{w}_2^2}
   +
   \tfrac{1}{\norm{w}_2^2}
   +
   \tfrac{\norm{v}_\infty}{\norm{v}_2^2}
   +
   \left(\norm{v}_\infty+\norm{w}_\infty\right)^6
   +
   \tfrac{\norm{v}_\infty+\norm{w}_\infty}{\norm{w}_2^3}
   \right)
\\
   &\le
   C_\rho^{\norm{z}_2}\norm{v-w}_\infty\norm{\xi}_\infty
      \quad
{\color{gray}
   ,\;
   C_\rho^{\norm{z}_2}
   :=
   \left(
   64\rho^6
   +\tfrac{4+\rho^4}{\norm{z}_2^2}
   +\tfrac{16\rho}{\norm{z}_2^3}
   +\tfrac{64\rho^2}{\norm{z}_2^4}
   \right)
   .
}
\end{split}
\end{equation}
where the last inequality uses hypothesis~(\ref{eq:V_z})
met in the applications in \S\,\ref{sec:ae}, namely
$\norm{v}_2,\norm{w}_2\ge\frac12\norm{z}_2$ and
$\norm{v}_\infty,\norm{w}_\infty\le \rho$.

\medskip
\noindent
We need in~(\ref{eq:42-D53}) an estimate for the following $L^\infty$
norm. Step~2 is by~(\ref{eq:dt|_z})
\begin{equation}\label{eq:dt_z-diff}
\begin{split}
   &\tfrac12\norm{dt_v\xi-dt_w\xi}_\infty
\\
   &=\tfrac12\sup_{\tau\in\SS^1}
   \abs{(dt_v\xi)_\tau-(dt_w\xi)_\tau}
\\
   &\stackrel{{\color{gray}2}}{=}
   \sup_{\tau\in\SS^1}
   \biggl|
   \int_0^\tau
   \left(
   \tfrac{\inner{v_\sigma}{\xi_\sigma}_0}{\norm{v}_2^2}
   -
     \tfrac{\inner{w_\sigma}{\xi_\sigma}_0}{\norm{w}_2^2}
   -
   \tfrac{\INNER{v}{\xi}}{\norm{v}_2^4}\Abs{v_\sigma}^2
   +
   \tfrac{\INNER{w}{\xi}}{\norm{w}_2^4}\Abs{w_\sigma}^2
   \right)
   d\sigma   
   \biggr|
\\
   &\stackrel{{\color{gray}3}}{\le}
   \int_0^1
   \left|
   \tfrac{\inner{v_\sigma}{\xi_\sigma}_0}{\norm{v}_2^2}
   -
     \tfrac{\inner{w_\sigma}{\xi_\sigma}_0}{\norm{w}_2^2}
   -
   \tfrac{\INNER{v}{\xi}}{\norm{v}_2^4}\Abs{v_\sigma}^2
   +
   \tfrac{\INNER{w}{\xi}}{\norm{w}_2^4}\Abs{w_\sigma}^2
   \right|
   d\sigma
\\
   &\stackrel{{\color{gray}4}}{\le}
   \sup_{\tau\in\SS^1}
   \left|
   \tfrac{\inner{v_\tau}{\xi_\tau}_0}{\norm{v}_2^2}
   -
     \tfrac{\inner{w_\tau}{\xi_\tau}_0}{\norm{w}_2^2}
   -
   \tfrac{\INNER{v}{\xi}}{\norm{v}_2^4}\Abs{v_\tau}^2
   +
   \tfrac{\INNER{w}{\xi}}{\norm{w}_2^4}\Abs{w_\tau}^2
   \right|
\\
   &\stackrel{{\color{gray}5}}{=}
   \tfrac12\norm{(dt_v\xi)^\prime-(dt_w\xi)^\prime}_\infty
\\
   &\stackrel{{\color{gray}6}}{\le}
   C_\rho^{\norm{z}_2}\norm{v-w}_\infty\norm{\xi}_\infty
   .
\end{split}
\end{equation}
Step 3 takes the absolute value inside the integral,
but then $\int_0^\tau\abs{\cdot}\le\int_0^1\abs{\cdot}$ and the
supremum over $\tau$ becomes void.
Step~4 estimates the integral by the supremum over the integrand
times the length of the integration interval $[0,1]$.
Step~5 is by equality 1 in~(\ref{eq:dt_z-deriv-difference}).
Step~6 is~(\ref{eq:dt_z-deriv-difference-infty}).

\boldmath
\subsubsection{Factor $\aaa$}
\unboldmath

Recall Remark~\ref{rmk:tw-per}\,(iii) that
$\aaa=(a^1, a^2)\colon\SS^1\times\Zfrak\to\R^2$ is the vector
potential of the twisted-periodic $1$-form $\vartheta$.
Restricted to the compact set $\SS^1\times\im z$, see factor $\bbb$, 
the following are finite constants 
(which depend on ${\norm{z}_\infty\le\norm{z}_{1,2}}$)
\begin{align}\label{eq:a-infty}
   \alpha_{\SS^1}^z
   &:=\bigl\|\aaa|_{\SS^1\times\im z}\bigr\|_\infty
   &
   \dot\alpha_{\SS^1}^z
   &:=\bigl\|\dot \aaa|_{\SS^1\times\im z}\bigr\|_\infty
\nonumber
\\
   d\alpha_{\SS^1}^z
   &:=\bigl\|d\aaa|_{\SS^1\times\im z}\bigr\|_\infty
   &
   \ddot\alpha_{\SS^1}^z
   &:=\bigl\|\ddot \aaa|_{\SS^1\times\im z}\bigr\|_\infty
\\
\nonumber
   d^2\alpha_{\SS^1}^z
   &:=\bigl\|d^2\aaa|_{\SS^1\times\im z}\bigr\|_\infty
   &
   d\dot\alpha_{\SS^1}^z
   &:=\bigl\|d\dot \aaa|_{\SS^1\times\im z}\bigr\|_\infty
   .
\end{align}

\boldmath
\subsubsection{Factor $b=\rot\,\aaa$}
\unboldmath

\begin{remark}[Estimates for $b=\rot\,\aaa$]\label{rem:b}
Recall that $b$ defined in~(\ref{eq:b}) abbreviates the rotational.
We further abbreviate
\begin{equation}\label{eq:b-infty}
\begin{gathered}
   \beta_{\SS^1}^z
   :=\bigl\|b|_{\SS^1\times\im z}\bigr\|_\infty
   ,\quad
   \dot\beta_{\SS^1}^z
   :=\bigl\|\dot b|_{\SS^1\times\im z}\bigr\|_\infty
   ,\quad
   d\beta_{\SS^1}^z
   :=\bigl\|db|_{\SS^1\times\im z}\bigr\|_\infty
   ,
\\
   d\dot\beta_{\SS^1}^z
   :=\bigl\|d\dot b|_{\SS^1\times\im z}\bigr\|_\infty
   ,\quad
   d^2\beta_{\SS^1}^z
   :=\bigl\|d^2b|_{\SS^1\times\im z}\bigr\|_\infty
   .
\end{gathered}
\end{equation}
These values are finite:
Firstly $b$ is a smooth map
$b\colon \SS^1\times \Zfrak\to \R$.
Secondly, since $z$ is in $U_2$ (could be even $U_1$),
it is in particular a continuous map
$\SS^1\to\Zfrak$ and therefore its image as a subset
$\im(z)\subset \Zfrak$ is compact.
Therefore the sup norm of the restriction of $b$
to $\SS^1\times\im(z)$ is finite, similarly for derivatives of $b$.
Hence the $\beta$-constants, analogously the $\beta$-constants above,
depend on ${\norm{z}_\infty\le\norm{z}_{1,2}}$.

With the chain rule we differentiate to obtain the estimate
\begin{equation*}
\begin{split}
   \norm{(db_{t_z}|_z)^\prime}_2
   &=\norm{d\dot b_{t_z}|_z\, t_z^\prime+d^2b_{t_z}|_z z^\prime}_2
\\
   &\le d\dot\beta_{\SS^1}^z\norm{t_z^\prime}_2
   +d^2\beta_{\SS^1}^z\norm{z}_{1,2}
\\
   &\le
   \left(d\dot\beta_{\SS^1}^z\tfrac{1}{\norm{z}_2}
   +d^2\beta_{\SS^1}^z
   \right) \norm{z}_{1,2}
   =: (d\beta)^{\prime,z_{1,2}}_2
   .
\end{split}
\end{equation*}
and the estimate
\begin{equation*}
\begin{split}
   \norm{(db_{t_z}|_z)^\prime}_\infty
   &=\norm{d\dot b_{t_z}|_z\, t_z^\prime+d^2b_{t_z}|_z z^\prime}_\infty
\\
   &\le d\dot\beta_{\SS^1}^z\norm{t_z^\prime}_\infty
   +d^2\beta_{\SS^1}^z\norm{z^\prime}_\infty
\\
   &\le
   d\dot\beta_{\SS^1}^z\tfrac{\norm{z}_{1,2}^2}{\norm{z}_2^2}
   +d^2\beta_{\SS^1}^z \norm{z}_{2,2}
   =:(d\beta)^{\prime,z_{2,2}}_\infty
   .
\end{split}
\end{equation*}
This concludes Remark~\ref{rem:b}.
\end{remark}

\boldmath
\subsection{Lagrangian scale 
$(\hfrak_0,\hfrak_1,\hfrak_2)=(L^2,W^{2,2},W^{4,2})$}
\label{sec:Lag-sc}
\unboldmath

For the notation $\hfrak_k$ and $\ufrak_k$
see Definition~\ref{def:analytic-setup-Lag}.

\boldmath
\subsubsection{Kinetic terms T1 and T2}
\unboldmath

\boldmath
\subsubsection*{Term T1 -- $T_{11}=0{\color{gray}\;=(\tilde T_{11}+m_1)}$,
$T_{12}$, $T_{13}$}
\unboldmath

\smallskip
\noindent
\textbf{\boldmath$T_{11}$.}
This term is zero.

\smallskip
\noindent
\textbf{\boldmath$T_{12}$.}
We set $T^z\xi:=8\inner{z^\prime}{\xi^\prime} z_\tau$ and estimate
\begin{equation*}
\begin{split}
   \norm{T^z\xi}_2
   &\le
   8\norm{z^\prime}_2\norm{\xi^\prime}_2\norm{z}_2
   \le 
   8\norm{z}_2\norm{z}_{1,2}\norm{\xi}_{1,2}
\\
{\color{gray}
   \norm{(T^z\xi)^\prime}_2
}
   &
{\color{gray}\;
   =\norm{8\inner{z^\prime}{\xi^\prime} z^\prime}_2
   \le
   8\norm{z^\prime}_2\norm{\xi^\prime}_2\norm{z^\prime}_2
}
{\color{gray}
   8\norm{z}_{1,2}^2\norm{\xi}_{1,2}
}
\\
   \norm{(T^z\xi)^{\prime\prime}}_2
   &=\norm{8\inner{z^{\prime}}{\xi^\prime} z^{\prime\prime}}_2
   \le
   8\norm{z^\prime}_2\norm{\xi^\prime}_2\norm{z^{\prime\prime}}_2
   \le
   8 {\color{blue}\norm{z}_{2,2}^2}\norm{\xi}_{1,2}
   .
\end{split}
\end{equation*}
This shows that
$$
   [z\mapsto T^z]
   \in C^0({\color{gray}\ufrak_1\subset\;} U_1,\Ll({\color{gray}\hfrak_1\INTO\;} W^{1,2},\hfrak_0))
   \cap C^0(\ufrak_1,\Ll({\color{gray}\hfrak_2\INTO \hfrak_1\INTO\;} W^{1,2}, \hfrak_1))
   .
$$
Hence on both levels $T=T_{12}$ takes values in the \emph{compact} linear
operators
$$
   T_{12}
   \in C^0(\ufrak_1,\Cc(\hfrak_1,\hfrak_0))
   \cap C^0({\color{blue}\ufrak_1},\Cc(\hfrak_2, \hfrak_1))
   .
$$

\smallskip
\noindent
\textbf{\boldmath$T_{13}$.}
We set $T^z\xi:=4\norm{z^\prime}_2^2\xi_\tau$ and estimate
\begin{equation*}
\begin{split}
   \norm{T^z\xi}_2
   &=
   4\norm{z^\prime}_2^2\norm{\xi}_2
   \le 
   4\norm{z}_{1,2}^2\norm{\xi}_2
\\
{\color{gray}
   \norm{(T^z\xi)^\prime}_2
}
   &
{\color{gray}\;
   =
   4\norm{z^\prime}_2^2\norm{\xi^\prime}_2
   \le 
   4\norm{z}_{1,2}^2\norm{\xi}_{1,2}
}
\\
   \norm{(T^z\xi)^{\prime\prime}}_2
   &=
   4\norm{z^\prime}_2^2\norm{\xi^{\prime\prime}}_2
   \le 
   4 {\color{blue}\norm{z}_{1,2}^2}\norm{\xi}_{2,2}
   .
\end{split}
\end{equation*}
This shows that
$$
   [z\mapsto T^z]
   \in C^0({\color{gray}\ufrak_1\subset\;} U_1,\Ll({\color{gray}\hfrak_1\INTO\;} \hfrak_0,\hfrak_0))
   \cap C^0({\color{gray}\ufrak_1\subset\;}
   U_1,\Ll({\color{gray}\hfrak_2\INTO\;} \hfrak_1,\hfrak_1))
   .
$$
Hence on both levels $T=T_{13}$ takes values in the \emph{compact} linear
operators
$$
   T
   \in C^0(\ufrak_1,\Cc(\hfrak_1,\hfrak_0))
   \cap C^0({\color{blue}\ufrak_1},\Cc(\hfrak_2, \hfrak_1))
   .
$$

\boldmath
\subsubsection*{Term T2 -- $C_{21}{\color{gray}\;=(\tilde T_{21}+m_2)}$,
$F_{22}$}
\unboldmath

\smallskip
\noindent
\textbf{\boldmath\color{red}$C_{21}$.}
We set
$C^z\xi:=-8 \inner{z}{\xi} {\color{red}z^{\prime\prime}_\tau}$ and estimate
\begin{equation*}
\begin{split}
   \norm{C^z\xi}_2
   &
   \le 8\norm{z}_2\norm{z^{\prime\prime}}_2\norm{\xi}_2
   \le 
   8\norm{z}_2\norm{z}_{2,2}\norm{\xi}_{\color{magenta}2}
\\
{\color{gray}\;
   \norm{(C^z\xi)^\prime}_2
}
   &
{\color{gray}\; 
  \le 8\norm{z}_2\norm{z^{\prime\prime\prime}}_2\norm{\xi}_2
   \le
   8\norm{z}_2\norm{z}_{3.2}\norm{\xi}_2
}
\\
   \norm{(C^z\xi)^{\prime\prime}}_2
   &
   \le 8\norm{z}_2\norm{{\color{red}z^{\prime\prime\prime\prime}}}_2\norm{\xi}_2
   \le
   8\norm{z}_2{\color{red}\norm{z}_{4,2}}\norm{\xi}_2
   .
\end{split}
\end{equation*}
This shows that, given any $r\in(\frac12,1)$, we have
$$
   [z\mapsto C^z]
   \in C^0(\ufrak_1,\Ll({\color{gray}\hfrak_r\INTO\;} \hfrak_{\color{magenta}0},\hfrak_0))
   \cap C^0({\color{red}\ufrak_2},\Ll({\color{gray}\hfrak_1\INTO\;}\hfrak_0,\hfrak_1))
   .
$$
Hence on both levels $C=C_{21}$ takes values in the \emph{compact} linear
operators
$$
   C_{21}
   \in C^0(\ufrak_1,\Cc({\color{magenta}\hfrak_r},\hfrak_0))
   \cap C^0({\color{red}\ufrak_2},\Cc(\hfrak_1))
   .
$$
\begin{remark}
Since on level two $C_{21}$ does not extend to $\ufrak_1$
this operator necessarily goes into the $C$-part of
the decomposition~(\ref{eq:Lag-Hess-nonmag}).
This is in general very undesirable since for the 
$C$-part one must show the scale Lipschitz estimate~(\ref{eq:(C)}).
For $C_{21}$, due to its extremely simple formula,
this estimate is still relatively short,
see Step~(\texttt{C}) in the proof of
Theorem~\ref{thm:Lag-Hess-nonmag}.

The following operator $F_{22}$ extends on level two to $\ufrak_1$,
it goes in the $F$-part of the decomposition~(\ref{eq:Lag-Hess-nonmag}).
In fact, it is Fredholm of index zero.
The previous operators $T_{12}$ and $T_{13}$, as well as $T_{31}$ below,
are compact perturbations of $F_{22}$.
\end{remark}

\smallskip
\noindent
\textbf{\boldmath\color{orange}$F_{22}$.}
We set $F^z\xi:=-4\norm{z}^2 {\color{orange}\xi^{\prime\prime}_\tau}$ and estimate
\begin{equation*}
\begin{split}
   \norm{F^z\xi}_2
   &
   \le 4\norm{z}_2^2\norm{{\color{orange}\xi^{\prime\prime}}}_2
   \le 
   4\norm{z}_2^2 {\color{orange}\norm{\xi}_{2,2}}
\\
{\color{gray}
   \norm{(F^z\xi)^\prime}_2
}
   &
{\color{gray}\;
   \le 4\norm{z}_2^2\norm{\xi^{\prime\prime\prime}}_2
   \le 
   4\norm{z}_2^2\norm{\xi}_{3,2}
}
\\
   \norm{(F^z\xi)^{\prime\prime}}_2
   &
   \le 4\norm{z}_2^2\norm{{\color{orange}\xi^{\prime\prime\prime\prime}}}_2
   \le 
   4{\color{blue}\norm{z}_2^2}{\color{orange}\norm{\xi}_{4,2}}
   .
\end{split}
\end{equation*}
This shows that
$$
   [z\mapsto F^z]
   \in C^0({\color{gray}\ufrak_0\subset\;} \ufrak_1,\Ll(\hfrak_1,\hfrak_0))
   \cap C^0({\color{gray}\ufrak_0\subset\;} {\color{blue}\ufrak_1},\Ll(\hfrak_2,\hfrak_1))
   .
$$
Hence on both levels $z\mapsto F^z=F_{22}^z$ takes values in the
bounded linear operators and on level two it {\color{blue}extends to $\ufrak_1$}
$$
   F_{12}
   \in C^0(\ufrak_1,\Ll({\color{orange} \hfrak_1},\hfrak_0))
   \cap C^0({\color{blue}\ufrak_1},\Ll({\color{orange} \hfrak_2}, \hfrak_1))
   .
$$

\boldmath
\subsubsection{Potential term T3}
\unboldmath

\smallskip
\noindent
\textbf{\boldmath$T_{31}=\tilde T_{31}+m_3$.}
We set
$T^z\xi:=-\tfrac{2\xi_\tau}{\norm{z}_2^4}
+\tfrac{8\inner{z}{\xi}z_\tau}{\norm{z}_2^6}$ and estimate
\begin{equation}\label{eq:T31-Lag}
\begin{split}
   \norm{T^z\xi}_2
   &
   \le 
   \tfrac{10}{\norm{z}_2^4}\norm{\xi}_2
\\
{\color{gray}
   \norm{(T^z\xi)^\prime}_2
}
   &
{\color{gray}\;
   \le
   \tfrac{2}{\norm{z}_2^4}
   \left(1+4\tfrac{\norm{z}_{1,2}}{\norm{z}_2}\right)
   \norm{\xi}_{1,2}
}
\\
   \norm{(T^z\xi)^{\prime\prime}}_2
   &\le
   \tfrac{2}{\norm{z}_2^4}
   \left(1+4\tfrac{{\color{blue}\norm{z}_{2,2}}}{\norm{z}_2}\right)
   \norm{\xi}_{2,2}
   .
\end{split}
\end{equation}
This shows that
$$
   [z\mapsto T^z]
   \in C^0({\color{gray}\ufrak_0\subset\;} \ufrak_1,\Ll({\color{gray}\hfrak_1\INTO\;} \hfrak_0,\hfrak_0))
   \cap C^0({\color{blue}\ufrak_1},
   \Ll({\color{gray}\hfrak_2\INTO\;} \hfrak_1,\hfrak_1))
   .
$$
Hence on both levels $T=T_{13}$ takes values in the \emph{compact} linear
operators
$$
   T
   \in C^0(\ufrak_1,\Cc(\hfrak_1,\hfrak_0))
   \cap C^0({\color{blue}\ufrak_1},\Cc(\hfrak_2, \hfrak_1))
   .
$$

\boldmath
\subsection{Hamiltonian scale
$( h_0, h_1, h_2)=(L^2,W^{1,2},W^{2,2})$}
\label{sec:Ham-sc}
\unboldmath

In this section we analyze the potential contribution $z\mapsto U^z$
and, as described in Remark~\ref{rem:M_ij} and most importantly,
the magnetic contribution $z\mapsto M^z$ to the
Hessian field~$z\mapsto B^z$; see Lemma~\ref{le:Hessians-Lagrange}.

For the notation $h_k\supset u_k$
see Definition~\ref{def:analytic-setup}.

\boldmath
\subsubsection{Potential term T3}
\unboldmath

\smallskip
\noindent
\textbf{\boldmath$T_{31}=\tilde T_{31}+m_3$.}
We set
$
T^z\xi:=-\tfrac{2\xi_\tau}{\norm{z}_2^4}
+\tfrac{8\inner{z}{\xi}z_\tau}{\norm{z}_2^6}
$.
\newline
The first two estimates in~(\ref{eq:T31-Lag})
show that
$$
   [z\mapsto T^z]
   \in
   C^0({\color{gray} u_0\subset\;}  u_1,\Ll({\color{gray} h_1\INTO\;}  h_0, h_0))
   \cap
   C^0({\color{blue} u_1},\Ll({\color{gray} h_2\INTO\;}  h_1, h_1))
   .
$$
Hence on both levels $T=T_{13}$ takes values in the \emph{compact} linear
operators
$$
   T_{31}
   \in C^0( u_1,\Cc( h_1, h_0))
   \cap C^0({\color{blue} u_1},\Cc( h_2,  h_1))
   .
$$

\boldmath
\subsubsection{Magnetic terms T4-T7}
\unboldmath

Concerning notation $h_k\supset u_k$ see
Definition~\ref{def:analytic-setup}.

To carry out the following $17+{\color{orange}1}$ estimates we
use the utility estimates prepared in \S\,\ref{sec:util},
in particular for the vector potential $\aaa$, the function
$b=\rot\,\aaa$, and the Barutello-Ortega-Verzini reparametrization
$t_z$.

\boldmath
\subsubsection*{Term T4 -- $T_{41}=0{\color{gray}\;=\tilde T_{41}+m_4}$,
${\color{red}C_{42}}$, ${\color{red}C_{43}}$, ${\color{orange}F_{44}}$}
\label{sec:T4}
\unboldmath

\smallskip
\noindent
\textbf{\boldmath$T_{41}+m_4$.} 
This term is zero.

\smallskip
\noindent
\textbf{\boldmath\color{red}$C_{42}$.}
We define
$
C^z\xi:=\dot b_{t_z(\tau)}|_{z_\tau}(dt_z\xi)_\tau\, j_0 z_\tau^\prime
$
and estimate
\begin{equation*}
\begin{split}
   \norm{C^z\xi}_2
   &\le\norm{\dot b|_{\SS^1\times\im z}}_\infty
   \norm{dt_z\xi}_\infty
   \norm{z^\prime}_2
   \le 
   \dot\beta_{\SS^1}^z
   \tfrac{4 {\color{blue}\norm{z}_{1,2}}}{\norm{z}_2}\,\norm{\xi}_2
\\
   \norm{(C^z\xi)^\prime}_2
   &\le\norm{(\dot b_{t_{z}}|_{z})^\prime}_2
   \norm{dt_z\xi}_\infty
   \norm{z^\prime}_\infty
   +\norm{\dot b|_{\SS^1\times\im z}}_\infty
   \norm{(dt_z\xi)^\prime}_{{\color{cyan} 2}}
   \norm{z^\prime}_\infty
   \\
   &\quad
   +\norm{\dot b|_{\SS^1\times\im z}}_\infty
   \norm{dt_z\xi}_\infty
   \norm{{\color{red} z^{\prime\prime}}}_2
\\
   &\le
   \left(
   (d\beta)^{\prime,z_{1,2}}_2\, \tfrac{4}{\norm{z}_2}
   +\dot\beta_{\SS^1}^z\, d^\prime\tfrak^{z_{1,2}}_\infty 
   +\dot\beta_{\SS^1}^z\, \tfrac{4}{\norm{z}_2}
   \right)
   {\color{red} \norm{z}_{2,2}}\norm{\xi}_{{\color{cyan} 2}}
   .
\end{split}
\end{equation*}
This shows that, given any $r\in(\frac12,1)$, we have
$$
   [z\mapsto C^z]
   \in C^0( u_1,\Ll({\color{gray} h_r\INTO\;}  h_0, h_0))
   \cap C^0({\color{red} u_2},\Ll({\color{gray} h_1\INTO\;} h_0, h_1))
   .
$$
Hence on both levels $C=C_{42}$ takes values in the \emph{compact} linear
operators
\begin{equation}\label{eq:C_42}
\begin{split}
   C_{42}
   &\in C^0( u_1,\Cc({\color{red} h_r}, h_0))
   \cap C^0({\color{red} u_2},\Cc( h_1))
\\
   &\subset
   C^0( u_1,\Cc( h_1, h_0))\cap C^0({\color{red} u_2},\Cc( h_2, h_1))
   .
\end{split}
\end{equation}
\begin{remark}
Since on level two $C_{42}$ does not extend to $ u_1$
this operator necessarily goes into the $C$-part of
the decomposition~(\ref{eq:A=F+C}).
This is in general very undesirable, since for the 
$C$-part one must show the scale Lipschitz estimate~(\ref{eq:(C)}).
For $C_{42}$ this cumbersome task is carried out
in Proposition~\ref{prop:C_42}
and for the following operator $C_{43}$ in Proposition~\ref{prop:C_43}.

As a rule of thumb, all operators which extend on level two to $ u_1$
should be put into the $F$-part of
the decomposition~(\ref{eq:A=F+C}).
One of them should be Fredholm of index zero and
others compact perturbations.
In the case at hand $f_{44}$ below contributes to the Fredholm
operator in Lemma~\ref{le:f_44-sFred}.
Luckily all the many summands in T5 T6 T7 further below play the role
of compact perturbations.
So these do not require any further work to prove almost extendable.
\end{remark}

\smallskip
\noindent
\textbf{\boldmath\color{red}$C_{43}$.}
We define 
$
C^z\xi:=(db_{t_z(\tau)}|_{z_\tau}\xi_\tau)\, j_0 z_\tau^\prime
$,
pick $r\in(\frac12,1)$, and estimate
\begin{equation*}
\begin{split}
   \norm{C^z\xi}_2
   &\le\norm{db_{t_{z}}|_{z}}_\infty \norm{\xi}_{\color{red}\infty}
   \norm{z^\prime}_{\color{blue}2}
   \stackrel{\text{(\ref{eq:H_r-est})}}{\le}
   d\beta_{\SS^1}^z
{\color{blue}
   \norm{z}_{1,2}
}
{\color{red}
   \norm{\xi}_{r,2}
}
\\
   \norm{(C^z\xi)^\prime}_2
   &\le
   \norm{(db_{t_{z}}|_{z})^\prime}_2
   \norm{\xi}_\infty\norm{z^\prime}_\infty
   +\norm{db_{t_{z}}|_{z}}_\infty\norm{\xi^\prime}_2\norm{z^\prime}_\infty
   \\
   &\quad
   +\norm{db_{t_{z}}|_{z}}_\infty\norm{\xi}_\infty\norm{{\color{red} z^{\prime\prime}}}_2
\\
   &\le
   \bigl(
    (d\beta)^{\prime,z_{1,2}}_2
   +2d\beta_{\SS^1}^z
   \bigr)\,
{\color{red}
   \norm{z}_{2,2}
}
\norm{\xi}_{1,2}
   .
\end{split}
\end{equation*}
This shows that, given any $r\in(\frac12,1)$, we have
\begin{equation}\label{eq:C_43}
\begin{split}
   C_{43}
   &\in C^0( u_1,\Ll({\color{red} h_r}, h_0))
   \cap C^0({\color{red} u_2},\Ll( h_1))
\\
   &\subset
   C^0( u_1,\Cc( h_1, h_0))\cap C^0({\color{red} u_2},\Cc( h_2, h_1))
   .
\end{split}
\end{equation}

\smallskip
\noindent
\textbf{\boldmath\color{orange}$F_{44}=f_{44}\p_\tau$.}
Set $F_{44}^z=f_{44}^z\p_\tau$
and $f^z_{44}\xi:=-4b_{t_z}|_z\, j_0\xi$.
With the estimates for $t_z$ in \S\,\ref{sec:t_z}
and the $\beta$-constants in~(\ref{eq:b-infty})
we get the estimates
\begin{equation}\label{eq:f_44-individual}
\begin{split}
   \norm{f^z_{44}\xi}_2
   &\le
   \norm{4b_{t_z}|_z\, j_0 \xi}_2
   \le
   4 \beta_{\SS^1}^z \norm{\xi}_{2}
\\
   \norm{(f^z_{44}\xi)^\prime}_2
   &\le
   4\norm{(b_{t_{z}}|_{z})^\prime}_2\norm{\xi}_\infty
   +4\norm{b_{t_{z}}|_{z}}_\infty\norm{\xi^{\prime}}_2
\\
   &\le
   4 \left(\norm{\dot b_{t_{z}}|_{z}\, t_z^\prime}_2
   +\norm{db_{t_{z}}|_{z} z^\prime}_2
   + \beta_{\SS^1}^z\right) \norm{\xi}_{1,2}
\\
   &\le
   \underbrace{
   4 \left(
   \dot\beta_{\SS^1}^z\tfrac{\norm{z}_{1,2}}{\norm{z}_2}
   +d\beta_{\SS^1}^z\norm{z}_{1,2}
   + \beta_{\SS^1}^z\right)
   }_{=:\gamma_{\norm{z}_{1,2}}}
   \norm{\xi}_{1 ,2}
   .
\end{split}
\end{equation}
This shows the following estimate
\begin{equation}\label{eq:f_44}
\begin{split}
   \norm{f^z_{44}\xi}_{1,2}
   \le \norm{f^z_{44}\xi}_2+\norm{(f^z_{44}\xi)^\prime}_2
   &\le2 {\color{blue}\gamma_{\norm{z}_{1,2}}}\norm{\xi}_{1 ,2}
\\
   [z\mapsto f_{44}^z]
   &\in C^0({\color{blue}  u_1},\Ll( h_1))
   ,
\end{split}
\end{equation}
while the continuity assertion follows by writing out the
formula for ${f^z_{44}}^\prime$ together with smoothness of
$b_{t_z}|_z$ and $t_z(\tau)$ in both variables.
Since $F^z_{44}=f_{44}^z\p_\tau=-4b_{t_z}|_z\, j_0\p_\tau$
the estimates above immediately tell that
\begin{equation*}
\begin{split}
   \norm{F^z_{44}\xi}_2
   &=\norm{f^z_{44}\xi^\prime}_2
   \le
   4 \beta_{\SS^1}^z \norm{\xi}_{1,2}
\\
   \norm{(F^z_{44}\xi)^\prime}_2
   &=\norm{(f^z_{44}\xi^\prime)^\prime}_2
   \le
{\color{blue}\;
   \gamma_{\norm{z}_{1,2}}
}
   \norm{\xi}_{2,2}
   .
\end{split}
\end{equation*}
This shows that on both levels $z\mapsto F^z=F_{44}^z$ takes values in the
bounded linear operators and on level two it {\color{blue}extends to $ u_1$}
$$
   F_{44}
   \in C^0( u_1,\Ll( h_1, h_0))
   \cap C^0({\color{blue} u_1},\Ll( h_2, h_1))
   .
$$

\boldmath
\subsubsection*{Term T5 -- $T_{51}$, $T_{52}=\tilde T_{52}+m_5$,
$T_{53}$, $T_{54}$}
\unboldmath

\smallskip
\noindent
\textbf{\boldmath$T_{51}$.} 
We define $T^z\xi:=-\tfrac{2\inner{z_\tau}{\xi_\tau}_0}{\norm{z}_2^2}\,
\dot\aaa_{t_z(\tau)}|_{z_\tau}$
and estimate
\begin{equation*}
\begin{split}
   \norm{T^z\xi}_2
   &\le\tfrac{2\norm{z}_2\norm{\xi}_2}{\norm{z}_2^2}\,
   \dot\alpha_{\SS^1}^z
   \le 2\dot\alpha_{\SS^1}^z\tfrac{1}{\norm{z}_2}\norm{\xi}_2
\\
   \norm{(T^z\xi)^\prime}_2
   &=\Bigl\|
   \tfrac{2\inner{z_\tau^\prime}{\xi_\tau}_0}{\norm{z}_2^2}\,
   \dot\aaa_{t_z(\tau)}|_{z_\tau}
   +
   \tfrac{2\inner{z_\tau}{\xi_\tau^\prime}_0}{\norm{z}_2^2}\,
   \dot\aaa_{t_z(\tau)}|_{z_\tau}
   \\
   &\quad\;
   +
   \tfrac{2\inner{z_\tau}{\xi_\tau}_0}{\norm{z}_2^2}\,
   \ddot\aaa_{t_z(\tau)}|_{z_\tau} t_z^\prime(\tau)
   +
   \tfrac{2\inner{z_\tau}{\xi_\tau}_0}{\norm{z}_2^2}\,
   d\dot\aaa_{t_z(\tau)}|_{z_\tau} z^\prime_\tau
   \Bigr\|_2
   \\
   &\le
   \tfrac{2}{\norm{z}_2^2}
   \Bigl(
   \dot\alpha_{\SS^1}^z\norm{\xi}_\infty\norm{z^\prime}_2
   +\dot\alpha_{\SS^1}^z\norm{z}_\infty\norm{\xi^\prime}_2
   \\
   &\quad\;
   +\ddot\alpha_{\SS^1}^z\norm{z}_\infty\norm{\xi}_2
   \tfrac{\norm{z}_{1,2}^2}{\norm{z}_2^2}
   +d\dot\alpha_{\SS^1}^z\norm{z}_\infty\norm{\xi}_\infty\norm{z^\prime}_2
   \Bigr)
   \\
   &\le
   \tfrac{2}{\norm{z}_2^2}
   \left(
   2 \dot\alpha_{\SS^1}^z
   +\ddot\alpha_{\SS^1}^z\tfrac{\norm{z}_{1,2}^2}{\norm{z}_2^2}
   +d\dot\alpha_{\SS^1}^z \norm{z}_{1,2}
   \right)
{\color{blue}
   \norm{z}_{1,2}
}
   \norm{\xi}_{1,2}
   .
\end{split}
\end{equation*}
This shows that
$$
   [z\mapsto T^z]
   \in C^0(  u_1,\Ll({\color{gray} h_1\INTO\;}  h_0, h_0))
   \cap C^0({\color{blue}  u_1},\Ll({\color{gray} h_2\INTO\;}  h_1, h_1))
   .
$$
Hence on both levels $T=T_{51}$ takes values in the
\emph{compact} linear operators and on level two it
{\color{blue}extends to $ u_1$}
$$
   T_{51}
   \in C^0( u_1,\Cc( h_1, h_0))
   \cap C^0({\color{blue} u_1},\Cc( h_2,  h_1))
   .
$$

\smallskip
\noindent
\textbf{\boldmath$T_{52}+m_5$.} 
We define $T^z\xi:=\tfrac{2\abs{z_\tau}^2\inner{z}{\xi}}{\norm{z}_2^4}\,
\dot\aaa_{t_z(\tau)}|_{z_\tau}$ and estimate
\begin{equation*}
\begin{split}
   \norm{T^z\xi}_2
   &\le\tfrac{2\norm{z}_\infty\norm{z}_2^2\norm{\xi}_2}{\norm{z}_2^4}\,
   \dot\alpha_{\SS^1}^z
   \le 2\dot\alpha_{\SS^1}^z\tfrac{\norm{z}_{1,2}}{\norm{z}_2^2}\norm{\xi}_2
\\
   \norm{(T^z\xi)^\prime}_2
   &=\Bigl\|
   \tfrac{4\inner{z_\tau}{z_\tau^\prime}\inner{z}{\xi}}{\norm{z}_2^4}\,
   \dot\aaa_{t_z(\tau)}|_{z_\tau}
   \\
   &\quad\;
   +
   \tfrac{2\abs{z_\tau}^2\inner{z}{\xi}}{\norm{z}_2^4}\,
   \ddot\aaa_{t_z(\tau)}|_{z_\tau} t_z^\prime(\tau)
   +
   \tfrac{2\abs{z_\tau}^2\inner{z}{\xi}}{\norm{z}_2^4}\,
   d\dot\aaa_{t_z(\tau)}|_{z_\tau} z^\prime_\tau
   \Bigr\|_2
   \\
   &\le
   4\dot\alpha_{\SS^1}^z
   \tfrac{\norm{z}_\infty\norm{z^\prime}_2\norm{z}_2\norm{\xi}_2}{\norm{z}_2^4}
   +
   2\ddot\alpha_{\SS^1}^z
   \tfrac{\norm{z}_\infty\norm{z}_2\norm{z}_2\norm{\xi}_2}{\norm{z}_2^4}
   \tfrac{\norm{z}_{1,2}^2}{\norm{z}_2^2}
   \\
   &\quad\;
   +2d\dot\alpha_{\SS^1}^z
   \tfrac{\norm{z}_\infty\norm{z}_\infty\norm{z}_2\norm{\xi}_2}{\norm{z}_2^4}
   \norm{z^\prime}_2
   \\
   &\le
   \tfrac{\norm{z}_{1,2}^2}{\norm{z}_2^3}
   \left(
   4\dot\alpha_{\SS^1}^z
   +2\ddot\alpha_{\SS^1}^z \tfrac{\norm{z}_{1,2}}{\norm{z}_2}
   +2d\dot\alpha_{\SS^1}^z \norm{z}_{1,2}
   \right)\norm{\xi}_2
   .
\end{split}
\end{equation*}
This shows that on both levels $T=T_{52}$ takes values in the
\emph{compact} linear operators and on level two it
{\color{blue}extends to $ u_1$}
$$
   T_{52}
   \in C^0( u_1,\Cc( h_1, h_0))
   \cap C^0({\color{blue} u_1},\Cc( h_2,  h_1))
   .
$$

\smallskip
\noindent
\textbf{\boldmath$T_{53}$.} 
We define $T^z\xi:=-\tfrac{\abs{z_\tau}^2}{\norm{z}_2^2}\,
\ddot\aaa_{t_z(\tau)}|_{z_\tau} \,(dt_z\xi)_\tau$ and estimate
\begin{equation*}
\begin{split}
   \norm{T^z\xi}_2
   &\le\tfrac{\norm{z}_\infty\norm{z}_2}{\norm{z}_2^2}\,
   \norm{\ddot\aaa_{t_z}|_z}_\infty
   \norm{dt_z\xi}_\infty
   \le 4 \ddot\alpha_{\SS^1}^z
   \tfrac{\norm{z}_{1,2}}{\norm{z}_2^2} \norm{\xi}_2
\\
   \norm{(T^z\xi)^\prime}_2
   &=\Bigl\|
   \tfrac{2\inner{z_\tau}{z_\tau^\prime}}{\norm{z}_2^2}\,
   \ddot\aaa_{t_z(\tau)}|_{z_\tau} \,(dt_z\xi)_\tau
   +
   \tfrac{\abs{z_\tau}^2}{\norm{z}_2^2}\,
   \dddot\aaa_{t_z(\tau)}|_{z_\tau} t_z^\prime(\tau) \,(dt_z\xi)_\tau
   \\
   &\quad\;
   +
   \tfrac{\abs{z_\tau}^2}{\norm{z}_2^2}\,
   d\ddot\aaa_{t_z(\tau)}|_{z_\tau} z_\tau^\prime\,(dt_z\xi)_\tau
   +
   \tfrac{\abs{z_\tau}^2}{\norm{z}_2^2}\,
   \ddot\aaa_{t_z(\tau)}|_{z_\tau} \,(dt_z\xi)^\prime_\tau
   \Bigr\|_2
   \\
   &\le
   2\ddot\alpha_{\SS^1}^z
   \tfrac{\norm{z}_\infty\norm{z^\prime}_2}{\norm{z}_2^2}
   \tfrac{4\norm{\xi}_2}{\norm{z}_2}
  +\dddot\alpha_{\SS^1}^z
   \tfrac{\norm{z}_\infty^2}{\norm{z}_2^2}\tfrac{\norm{z}_{1,2}}{\norm{z}_2}
   \tfrac{4\norm{\xi}_2}{\norm{z}_2}
   \\
   &\quad\;
   +
   d\ddot\alpha_{\SS^1}^z
   \tfrac{\norm{z}_\infty^2}{\norm{z}_2^2}\norm{z^\prime}_2
   \tfrac{4\norm{\xi}_2}{\norm{z}_2}
  +
   \ddot\alpha_{\SS^1}^z
   \tfrac{\norm{z}_\infty^2}{\norm{z}_2^2}
   2\tfrac{\norm{z}_{1,2}}{\norm{z}_2^2}
   \left(1+\tfrac{\norm{z}_{1,2}}{\norm{z}_2}\right)
   \norm{\xi}_2
   \\
   &\le const(\norm{z}_{1,2})\,\norm{\xi}_2
   .
\end{split}
\end{equation*}
This shows that on both levels $T=T_{53}$ takes values in the
\emph{compact} linear operators and on level two it
{\color{blue}extends to $ u_1$}
$$
   T_{53}
   \in C^0( u_1,\Cc( h_1, h_0))
   \cap C^0({\color{blue} u_1},\Cc( h_2,  h_1))
   .
$$

\smallskip
\noindent
\textbf{\boldmath$T_{54}$.} 
We define $T^z\xi:=-\tfrac{\inner{z_\tau}{{\color{red}z_\tau}}}{\norm{z}_2^2}
\begin{pmatrix}
   d\dot a^1_{t_z(\tau)}|_{z_\tau} \xi_\tau
   \\
   d\dot a^2_{t_z(\tau)}|_{z_\tau} \xi_\tau
\end{pmatrix}$
and estimate
\begin{equation*}
\begin{split}
   \norm{T^z\xi}_2
   &\le\tfrac{\norm{z}_\infty^2}{\norm{z}_2^2}\,
   d\dot\alpha_{\SS^1}^z\norm{\xi}_2
   \le
   d\dot\alpha_{\SS^1}^z\tfrac{\norm{z}_{1,2}^2}{\norm{z}_2^2}\norm{\xi}_2
\\
   \norm{(T^z\xi)^\prime}_2
   &=\Biggl\|
   \tfrac{2\inner{z_\tau}{{\color{red}z_\tau^\prime}}}{\norm{z}_2^2}
\begin{pmatrix}
   d\dot a^1_{t_z(\tau)}|_{z_\tau} \xi_\tau
   \\
   d\dot a^2_{t_z(\tau)}|_{z_\tau} \xi_\tau
\end{pmatrix}
   \\
   &\quad\;
   +
   \tfrac{\abs{z_\tau}^2}{\norm{z}_2^2}
\begin{pmatrix}
   d\ddot a^1_{t_z(\tau)}|_{z_\tau} t_z^\prime(\tau) \xi_\tau
   +d^2\dot a^1_{t_z(\tau)}|_{z_\tau} z^\prime_\tau \xi_\tau
   +d\dot a^1_{t_z(\tau)}|_{z_\tau} \xi^\prime_\tau
   \\
   d\ddot a^2_{t_z(\tau)}|_{z_\tau} t_z^\prime(\tau) \xi_\tau
   +d^2\dot a^2_{t_z(\tau)}|_{z_\tau} z^\prime_\tau \xi_\tau
   +d\dot a^2_{t_z(\tau)}|_{z_\tau} \xi^\prime_\tau
\end{pmatrix}
   \Biggr\|_2
\\
   &\le const(\norm{z}_{1,2})\,\norm{\xi}_{1,2}
   .
\end{split}
\end{equation*}
This shows that on both levels $T=T_{54}$ takes values in the
\emph{compact} linear operators and on level two it
{\color{blue}extends to $ u_1$}
$$
   T_{54}
   \in C^0( u_1,\Cc( h_1, h_0))
   \cap C^0({\color{blue} u_1},\Cc( h_2,  h_1))
   .
$$

\boldmath
\subsubsection*{Term T6 -- $T_{61}$, $T_{62}=\tilde T_{62}+m_6$,
$T_{63}$, $T_{64}$, $T_{65}$}
\unboldmath

\smallskip
\noindent
\textbf{\boldmath$T_{61}$.} 
We define $T^z\xi:=\tfrac{2 \xi_\tau}{\norm{z}_2^2} {\textstyle\int_{\tau}^1}
\inner{\dot{\aaa}_{t_{z}(\sigma)}|_{z_{\sigma}}}{z^\prime_{\sigma}}_0 \; d\sigma$
and estimate
\begin{equation}\label{eq:C_61}
\begin{split}
   \norm{T^z\xi}_2
   &\le
   \tfrac{2\norm{\xi}_2}{\norm{z}_2^2}
   \sup_{\tau\in\SS^1}\Abs{{\textstyle\int_{\tau}^1}
   \inner{\dot{\aaa}_{t_{z}(\sigma)}|_{z_{\sigma}}}{z^\prime_{\sigma}}_0\;
   d\sigma}
   \\
   &
   \le\tfrac{2\norm{\xi}_2}{\norm{z}_2^2}
   \sup_{\tau\in\SS^1}\left(\norm{\dot{\aaa}_{t_z}|_z}_{{\color{red}L^2_{[\tau,1]}}}
   \norm{z^\prime}_{{\color{red}L^2_{[\tau,1]}}}\right)
   \\
   &
   \le 2\dot\alpha_{\SS^1}^z\tfrac{\norm{z}_{1,2}}{\norm{z}_2^2}\norm{\xi}_2
\\
   \norm{(T^z\xi)^\prime}_2
   &\le
   \tfrac{2\norm{\xi^\prime}_2}{\norm{z}_2^2}
   \sup_{\tau\in\SS^1}\Abs{
{\textstyle
   \int_{\tau}^1
}
\inner{\dot{\aaa}_{t_{z}(\sigma)}|_{z_{\sigma}}}{z^\prime_{\sigma}}_0
\; d\sigma}
   +\tfrac{2\norm{\xi}_\infty}{\norm{z}_2^2}\Norm{
   \inner{\dot{\aaa}_{t_{z}(\cdot)}|_{z_{\cdot}}}
   {z^\prime_{\cdot}}_0}_2
   \\
   &\le 2\dot\alpha_{\SS^1}^z\tfrac{\norm{z}_{1,2}}{\norm{z}_2^2}\norm{\xi}_{1,2}
   +2\dot\alpha_{\SS^1}^z\tfrac{\norm{z^\prime}_{2}}{\norm{z}_2^2}\norm{\xi}_{1,2}
   .
\end{split}
\end{equation}
This shows that on both levels $T=T_{61}$ takes values in the
\emph{compact} linear operators and on level two it
{\color{blue}extends to $ u_1$}
$$
   T_{61}
   \in C^0( u_1,\Cc( h_1, h_0))
   \cap C^0({\color{blue} u_1},\Cc( h_2,  h_1))
   .
$$

\smallskip
\noindent
\textbf{\boldmath$T_{62}$.} 
We define $T^z\xi:=-\tfrac{4 z_\tau\inner{z}{\xi}}{\norm{z}^4}
{\textstyle\int_{\tau}^1}
\inner{\dot{\aaa}_{t_{z}(\sigma)}|_{z_{\sigma}}}{z^\prime_{\sigma}}_0
\; d\sigma$
and estimate
\begin{equation*}
\begin{split}
   \norm{T^z\xi}_2
   &
   \le 4 \dot\alpha_{\SS^1}^z\tfrac{\norm{z}_{1,2}}{\norm{z}_2^2}
   \norm{\xi}_2
   \qquad \text{\small\color{gray} (see $T_{61}$)}
\\
   \norm{(T^z\xi)^\prime}_2
   &\le
   \tfrac{4\norm{z^\prime}_2\norm{z}_2\norm{\xi}_2}{\norm{z}_2^4}
   \sup_{\tau\in\SS^1}\Abs{
{\textstyle
   \int_{\tau}^1
}
\inner{\dot{\aaa}_{t_{z}(\sigma)}|_{z_{\sigma}}}{z^\prime_{\sigma}}_0
\; d\sigma}
   \qquad \text{\small\color{gray} (see $T_{61}$)}
   \\
   &\quad
   +\tfrac{4\norm{z}_\infty\norm{z}_2\norm{\xi}_2}{\norm{z}_2^4}\Norm{
   {\color{gray}\tau\mapsto\;}\inner{\dot{\aaa}_{t_{z}(\tau)}|_{z_{\tau}}}
   {z^\prime_{\tau}}_0}_2
   \\
   &\le
   8 \dot\alpha_{\SS^1}^z\tfrac{\norm{z}_{1,2}^2}{\norm{z}_2^3}\norm{\xi}_2
   .
\end{split}
\end{equation*}
This shows that on both levels $T=T_{62}$ takes values in the
\emph{compact} linear operators and on level two it
{\color{blue}extends to $ u_1$}
$$
   T_{62}
   \in C^0( u_1,\Cc( h_1, h_0))
   \cap C^0({\color{blue} u_1},\Cc( h_2,  h_1))
   .
$$

\smallskip
\noindent
\textbf{\boldmath$T_{63}$.}
We define $T^z\xi:=\tfrac{2z_\tau}{\norm{z}^2_2}
{\textstyle\int_{\tau}^1}
\inner{\ddot{\aaa}_{t_{z}(\sigma)}|_{z_{\sigma}}\,(dt_z\xi)_\sigma}{z^\prime_{\sigma}}_0\; d\sigma$
and estimate
\begin{equation*}
\begin{split}
   \norm{T^z\xi}_2
   &\le \tfrac{2\norm{z}_2}{\norm{z}^2_2}
   \left(\norm{\ddot{\aaa}_{t_z}|_z}_{L^\infty_{[\tau,1]}}
   \norm{dt_z\xi}_{L^2_{[\tau,1]}}
   \norm{z^\prime}_{L^2_{[\tau,1]}}\right)
   \qquad\, \text{\small\color{gray} (see $T_{61}$)}
   \\
   &\le\tfrac{2}{\norm{z}_2}\ddot\alpha_{\SS^1}^z
   \tfrac{4}{\norm{z}_2}\norm{\xi}_2 \norm{z^\prime}_2
   \\
   &\le
   8 \ddot\alpha_{\SS^1}^z
   \tfrac{\norm{z}_{1,2}}{\norm{z}_2^2}\norm{\xi}_2
\\
   \norm{(T^z\xi)^\prime}_2
   &\le \tfrac{2\norm{z^\prime}_2}{\norm{z}^2_2}
   \left(\norm{\ddot{\aaa}_{t_z}|_z}_{L^\infty_{[\tau,1]}}
   \norm{dt_z\xi}_{L^2_{[\tau,1]}}
   \norm{z^\prime}_{L^2_{[\tau,1]}}\right)
   \qquad \text{\small\color{gray} (see $\norm{T^z\xi}_2$)}
   \\
   &\quad
   +\tfrac{2\norm{z}_\infty}{\norm{z}_2^2}
   \norm{
   {\color{gray}\tau\mapsto\;}
   \inner{\underbrace{\ddot{\aaa}_{t_{z}(\tau)}|_{z_{\tau}}}_{L^\infty}
   \,\underbrace{(dt_z\xi)_\tau}_{L^\infty}}
   {\underbrace{z^\prime_{\tau}}_{L^2}}_0}_2
   \qquad\, \text{\small\color{gray} (see $T_{61}$)}
   \\
   &\le
   16 \ddot\alpha_{\SS^1}^z
   \tfrac{\norm{z}_{1,2}^2}{\norm{z}_2^3}\norm{\xi}_2
   .
\end{split}
\end{equation*}
This shows that on both levels $T=T_{63}$ takes values in the
\emph{compact} linear operators and on level two it
{\color{blue}extends to $ u_1$}
$$
   T_{63}
   \in C^0( u_1,\Cc( h_1, h_0))
   \cap C^0({\color{blue} u_1},\Cc( h_2,  h_1))
   .
$$

\smallskip
\noindent
\textbf{\boldmath$T_{64}$.} 
Define $T^z\xi:=\tfrac{2z_\tau}{\norm{z}_2^2}
   \int_{\tau}^1
   \inner{
      (d\dot a^1_{t_z(\sigma)}|_{z_\sigma} \xi_\sigma,
      d\dot a^2_{t_z(\sigma)}|_{z_\sigma} \xi_\sigma)
   }
   {z^\prime_{\sigma}}_0
   \; d\sigma$
and estimate
\begin{equation*}
\begin{split}
   \norm{T^z\xi}_2
   &\le \tfrac{2\norm{{\color{red}z}}_2}{\norm{z}_2^2}
   \sup_{\tau\in\SS^1}\Abs{\int_{\tau}^1
   \INNER{
   \begin{pmatrix}
      d\dot a^1_{t_z(\sigma)}|_{z_\sigma} \xi_\sigma
      \\
      d\dot a^2_{t_z(\sigma)}|_{z_\sigma} \xi_\sigma
   \end{pmatrix}
   }
   {z^\prime_{\sigma}}_0
   \; d\sigma
   }
   \\
   &=
   \tfrac{2}{\norm{z}_2}\sup_{\tau\in\SS^1}\Abs{
   \INNER{
   \begin{pmatrix}
      d\dot a^1_{t_z}|_z \xi
      \\
      d\dot a^2_{t_z}|_z \xi
   \end{pmatrix}
   }
   {z^\prime}_{L^2_{[\tau,1]}}
   }
   \\
   &\le
   \tfrac{2}{\norm{z}_2}
   \Norm{\begin{pmatrix}
      d\dot a^1_{t_z}|_z \xi
      \\
      d\dot a^2_{t_z}|_z \xi
   \end{pmatrix}}_2
   \norm{z^\prime}_2
   \\
   &\le \tfrac{2}{\norm{z}_2}\sqrt{2}
   \bigl\|d\dot \aaa|_{\SS^1\times\im z}\bigr\|_\infty
   \norm{\xi}_2\norm{z^\prime}_2
   \\
   &
   \le 2\sqrt{2}d\dot\alpha_{\SS^1}^z
   \tfrac{\norm{z}_{1,2}}{\norm{z}_2} \norm{\xi}_2
\\
   \norm{(T^z\xi)^\prime}_2
   &\le \tfrac{2\norm{{\color{red}z^\prime}}_2}{\norm{z}_2^2}
   \sup_{\tau\in\SS^1}\Abs{\int_{\tau}^1
   \INNER{
   \begin{pmatrix}
      d\dot a^1_{t_z(\sigma)}|_{z_\sigma} \xi_\sigma
      \\
      d\dot a^2_{t_z(\sigma)}|_{z_\sigma} \xi_\sigma
   \end{pmatrix}
   }
   {z^\prime_{\sigma}}_0
   \; d\sigma
   }
   \quad \text{\small\color{gray} (see $\norm{T^z\xi}_2$)}
   \\
   &\quad\;
   +\tfrac{2\norm{z}_{{\color{red}\infty}}}{\norm{z}_2^2}
   \Norm{
   {\color{gray}\tau\mapsto\;}
   \INNER{
   \begin{pmatrix}
      d\dot a^1_{t_z(\tau)}|_{z_\tau} \xi_\tau
      \\
      d\dot a^2_{t_z(\tau)}|_{z_\tau} \xi_\tau
   \end{pmatrix}
   }
   {z^\prime_{\tau}}_0}_2
   \\
   &\le
   2\sqrt{2}d\dot\alpha_{\SS^1}^z
   \tfrac{\norm{z}_{1,2}^2}{\norm{z}_2^2} \norm{\xi}_2
   %
   +2 \tfrac{\norm{z}_{1,2}   }{\norm{z}_2^2}
   \underbrace{
   \Norm{
   \begin{pmatrix}
     d\dot a^1_{t_z}|_z \xi
     \\
     d\dot a^2_{t_z}|_z \xi
   \end{pmatrix}}_\infty
    }_{\le \sqrt{2}d\dot\alpha_{\SS^1}^z\norm{\xi}_\infty}
   \norm{z^\prime}_2
   \\
   &
   \le 4\sqrt{2}d\dot\alpha_{\SS^1}^z
   \tfrac{\norm{z}_{1,2}^2}{\norm{z}_2^2} \norm{\xi}_{1,2}
   .
\end{split}
\end{equation*}
This shows that on both levels $T=T_{64}$ takes values in the
\emph{compact} linear operators and on level two it
{\color{blue}extends to $ u_1$}
$$
   T_{64}
   \in C^0( u_1,\Cc( h_1, h_0))
   \cap C^0({\color{blue} u_1},\Cc( h_2,  h_1))
   .
$$

\smallskip
\noindent
\textbf{\boldmath\color{red}$T_{65}$.}
We define
\begin{equation}\label{eq:T_65}
\begin{split}
   F^z\xi:
   &=\tfrac{2z_\tau}{\norm{z}_2^2}\int_{\tau}^1
   \inner{\dot{\aaa}_{t_{z}(\sigma)}|_{z_{\sigma}}}{{\color{orange}\xi^\prime_{\sigma}}}_0
   \; d\sigma
\end{split}
\end{equation}
and we estimate
\begin{equation}\label{eq:F_65-est}
\begin{split}
   \norm{F^z\xi}_2
   &\le
   2\dot\alpha_{\SS^1}^z\tfrac{1}{\norm{z}_2}\norm{\xi^\prime}_2
   \le
   2\dot\alpha_{\SS^1}^z\tfrac{1}{\norm{z}_2}\norm{\xi}_{1,2}
   \qquad \text{\small\color{gray} (see $C_{61}$)}
\\
   \norm{(F^z\xi)^\prime}_2
   &=\tfrac{2}{\norm{z}_2^2}
   \Norm{{\color{gray}\tau\mapsto\;} z^\prime_\tau
{\textstyle
   \int_{\tau}^1
}
\inner{\dot{\aaa}_{t_{z}(\sigma)}|_{z_{\sigma}}}{{\color{orange}\xi^\prime_{\sigma}}}_0
\; d\sigma
   -z_\tau \inner{\dot{\aaa}_{t_{z}(\tau)}|_{z_{\tau}}}
   {{\color{orange}\xi^\prime_{\tau}}}_0}_2
   \\
   &\le 
   4\dot\alpha_{\SS^1}^z\tfrac{\norm{z}_{1,2}}{\norm{z}_2^2}\norm{\xi}_{1,2}
   .
   \qquad \text{\small\color{gray} (see $C_{61}$)}
\end{split}
\end{equation}
As usual on level two $z\mapsto T_{65}^z$
{\color{blue}extends to $ u_1$} and induces a
compact operator
$$
    h_2\stackrel{\rm cp.}{\INTO}
    h_1\stackrel{\rm bd}{\longrightarrow}
    h_1
   .
$$
On level one we only get boundedness $ h_1\to  h_0$, to prove
compactness is harder; see Lemma~\ref{le:f_65_75-comp}.

\boldmath
\subsubsection*{Term T7 -- $T_{71}=\tilde T_{71}+m_7$,
$T_{72}$, $T_{73}$, $T_{74}$, $T_{75}$}
\unboldmath

\smallskip
\noindent
\textbf{\boldmath$T_{71}$.} 
For $T^z\xi:=\left(\tfrac{8\inner{z}{\xi}z_\tau}{\norm{z}_2^6}
-\tfrac{2\xi_\tau}{\norm{z}_2^4}\right)
\int_0^1
{\textstyle
   \int_0^s\Abs{z_{\sigma}}^2 d\sigma
}
   \cdot\inner{\dot{\aaa}_{t_{z}(s)}|_{z_{s}}}{z^\prime_{s}}_0\; ds$
we estimate
\begin{equation*}
\begin{split}
   \norm{T^z\xi}_2
   &\le \left(\tfrac{8\norm{\xi}_2}{\norm{z}_2^4}
   +\tfrac{2\norm{\xi}_2}{\norm{z}_2^4}\right)
   \biggl| \int_0^1
\underbrace{
{\textstyle
   \int_0^s\Abs{z_{\sigma}}^2 d\sigma}
}_{=: f(s)\le f(1)}
   \cdot\inner{\dot{\aaa}_{t_{z}(s)}|_{z_{s}}}{z^\prime_{s}}_0\; ds\biggr|
   \\
   &=
   \tfrac{10\norm{\xi}_2}{\norm{z}_2^4}
   \Abs{\INNER{f \dot{\aaa}_{t_z}|_z}{z^\prime}}
   \le \tfrac{10\norm{\xi}_2}{\norm{z}_2^4}
   \norm{(\dot{\aaa}_{t_z}|_z)}_\infty
   \norm{\underbrace{f}_{\le f(1)}}_2
   \norm{z^\prime}_2
   \\
   &\le \tfrac{10\norm{\xi}_2}{\norm{z}_2^4}
   \dot\alpha_{\SS^1}^z
   \underbrace{f(1)}_{=\norm{z}_2^2}\underbrace{\norm{1}_2}_{=1}
   \norm{z}_{1,2}
   \\
   &
   \le 10 \dot\alpha_{\SS^1}^z 
   \tfrac{\norm{z}_{1,2}}{\norm{z}_2^2}\,\norm{\xi}_2
\\
   \norm{(T^z\xi)^\prime}_2
   &\le
   8 \dot\alpha_{\SS^1}^z 
   \tfrac{\norm{z}_{1,2}}{\norm{z}_2^2}
   \left(\tfrac{\norm{z}_{1,2}}{\norm{z}_2}+1\right)\norm{\xi}_{1,2}
   .
   \qquad \text{\small\color{gray} (cf. $\norm{T^z\xi}_2$)}
\end{split}
\end{equation*}
This shows that on both levels $T=T_{71}$ takes values in the
\emph{compact} linear operators and on level two it
{\color{blue}extends to $ u_1$}
$$
   T_{71}
   \in C^0( u_1,\Cc( h_1, h_0))
   \cap C^0({\color{blue} u_1},\Cc( h_2,  h_1))
   .
$$

\smallskip
\noindent
\textbf{\boldmath$T_{72}$.} 
For $T^z\xi:=-z_\tau\tfrac{4}{\norm{z}_2^4} \int_0^1
{\textstyle
   \int_0^s\inner{z_\sigma}{\xi_\sigma}\, d\sigma
}
   \cdot\inner{\dot{\aaa}_{t_{z}(s)}|_{z_{s}}}{z^\prime_{s}}_0\; ds$
we estimate
\begin{equation*}
\begin{split}
   \norm{T^z\xi}_2
   &\le \tfrac{4\norm{z}_2}{\norm{z}_2^4}
{\textstyle
   \int_0^1
}
\underbrace{
{\textstyle
   \int_0^s\inner{z_\sigma}{\xi_\sigma}\, d\sigma
}
}_{=:g(s)}
   \cdot\inner{\dot{\aaa}_{t_{z}(s)}|_{z_{s}}}{z^\prime_{s}}_0\; ds
   \qquad \text{\small\color{gray} (cf. $T_{71}$)}
   \\
   &\le\tfrac{4}{\norm{z}_2^3}
   \norm{\dot \aaa|_{\SS^1\times\im z}}_\infty
   \norm{g}_2\norm{z^\prime}_2
   {\color{gray}\qquad\small
   \text{, $\norm{g}_2\le \norm{z}_2\norm{\xi}_2$}}
   \\
   &
   \le 4 \dot\alpha_{\SS^1}^z
   \tfrac{\norm{z}_{1,2}}{\norm{z}_2^2}\norm{\xi}_2
\\
   \norm{(T^z\xi)^\prime}_2
   &=\Norm{
   {\color{red}z^\prime}\tfrac{4}{\norm{z}_2^4} \int_0^1
{\textstyle
   \int_0^s\inner{z_\sigma}{\xi_\sigma}\, d\sigma
}
   \cdot\inner{\dot{\aaa}_{t_{z}(s)}|_{z_{s}}}{z^\prime_{s}}_0\; ds
   }_2
   \\
   &\le
   4 \dot\alpha_{\SS^1}^z
   \tfrac{\norm{z}_{1,2}^2}{\norm{z}_2^3}\norm{\xi}_2
   .
   \qquad \text{\small\color{gray} (cf. $\norm{T^z\xi}_2$)}
\end{split}
\end{equation*}
This shows that on both levels $T=T_{72}$ takes values in the
\emph{compact} linear operators and on level two it
{\color{blue}extends to $ u_1$}
$$
   T_{72}
   \in C^0( u_1,\Cc( h_1, h_0))
   \cap C^0({\color{blue} u_1},\Cc( h_2,  h_1))
   .
$$

\smallskip
\noindent
\textbf{\boldmath$T_{73}$.} 
For $T^z\xi:=-z_\tau\tfrac{2}{\norm{z}_2^4}\int_0^1
{\textstyle
   \int_0^s\Abs{z_{\sigma}}^2 d\sigma
}
   \cdot\inner{\ddot{\aaa}_{t_{z}(s)}|_{z_{s}}\, (dt_z\xi)_s}{z^\prime_{s}}_0\; ds$
we estimate
\begin{equation*}
\begin{split}
   \norm{T^z\xi}_2
   &\le \tfrac{2\norm{z}_2}{\norm{z}_2^4}
   \norm{\underbrace{f}_{\le f(1)}}_2\;
   \norm{(\ddot{\aaa}_{t_z}|_z)}_\infty
   \norm{dt_z\xi}_\infty
   \norm{z^\prime}_2
   \qquad \text{\small\color{gray} (see $T_{71}$)}
   \\
   &
   \le 8 \ddot\alpha_{\SS^1}^z \tfrac{\norm{z}_{1,2}}{\norm{z}_2^2}\norm{\xi}_2
\\
   \norm{(T^z\xi)^\prime}_2
   &=\Norm{
   {\color{red}z^\prime}\tfrac{2}{\norm{z}_2^4}
{\textstyle
   \int_0^1
}
{\textstyle
   \int_0^s\Abs{z_{\sigma}}^2 d\sigma
}
   \cdot\inner{\ddot{\aaa}_{t_{z}(s)}|_{z_{s}} (dt_z\xi)_s}{z^\prime_{s}}_0\; ds
   }_2
   \\
   &
   \le 8 \ddot\alpha_{\SS^1}^z \tfrac{\norm{z}_{1,2}^2}{\norm{z}_2^3}\norm{\xi}_2
   .
   \qquad \text{\small\color{gray} (cf. $\norm{T^z\xi}_2$)}
\end{split}
\end{equation*}
This shows that on both levels $T=T_{73}$ takes values in the
\emph{compact} linear operators and on level two it
{\color{blue}extends to $ u_1$}
$$
   T_{73}
   \in C^0( u_1,\Cc( h_1, h_0))
   \cap C^0({\color{blue} u_1},\Cc( h_2,  h_1))
   .
$$

\smallskip
\noindent
\textbf{\boldmath$T_{74}$.} 
For $T^z\xi:=-z_\tau\tfrac{2}{\norm{z}^4}\int_0^1
{\textstyle
   \int_0^s\Abs{z_{\sigma}}^2 d\sigma
}
   \cdot\INNER{
   \begin{pmatrix}
      d\dot a^1_{t_z(s)}|_{z_s} \xi_s
      \\
      d\dot a^2_{t_z(s)}|_{z_s} \xi_s
   \end{pmatrix}
   }{z^\prime_{s}}_0\; ds$
we estimate
\begin{equation*}
\begin{split}
   \norm{T^z\xi}_2
   &
   \le 2\sqrt{2}d\dot\alpha_{\SS^1}^z
   \tfrac{\norm{z}_{1,2}}{\norm{z}_2}\norm{\xi}_2
\\
   \norm{(T^z\xi)^\prime}_2
   &\le
   \le 2\sqrt{2}d\dot\alpha_{\SS^1}^z
   \tfrac{\norm{z}_{1,2}^2}{\norm{z}_2^2}\norm{\xi}_2
   .
   \qquad \text{\small\color{gray} (cf. $\norm{T^z\xi}_2$)}
\end{split}
\end{equation*}
This shows that on both levels $T=T_{74}$ takes values in the
\emph{compact} linear operators and on level two it
{\color{blue}extends to $ u_1$}
$$
   T_{74}
   \in C^0( u_1,\Cc( h_1, h_0))
   \cap C^0({\color{blue} u_1},\Cc( h_2,  h_1))
   .
$$

\smallskip
\noindent
\textbf{\boldmath\color{red}$T_{75}$.}
We define 
\begin{equation}\label{eq:T_75}
   T^z\xi:=-z_\tau\tfrac{2}{\norm{z}^4_2} \int_0^1
{\textstyle
   \int_0^s\Abs{z_{\sigma}}^2 d\sigma
}
   \cdot\inner{\dot{\aaa}_{t_{z}(s)}|_{z_{s}}}{{\color{orange}\xi^\prime_s}}_0\;
   ds
\end{equation}
and estimate
\begin{equation}\label{eq:F_75-est}
\begin{split}
   \norm{T^z\xi}_2
   &
   \le 2\dot\alpha_{\SS^1}^z\tfrac{1}{\norm{z}_2}\norm{\xi^\prime}_2
   \le
   2\dot\alpha_{\SS^1}^z\tfrac{1}{\norm{z}_2}
   \;{\color{red}\norm{\xi}_{1,2}}
   \qquad \text{\small\color{gray} (see $T_{71}$)}
\\
   \norm{(T^z\xi)^\prime}_2
   &\le\tfrac{2}{\norm{z}^4_2} 
   \norm{z^\prime_\tau}_2
   \Abs{\int_0^1
{\textstyle
   \int_0^s\Abs{z_{\sigma}}^2 d\sigma
}
   \cdot\inner{\dot{\aaa}_{t_{z}(s)}|_{z_{s}}}{{\color{orange}\xi^\prime_s}}_0\; ds
   }
   \\
   &\le
   2\dot\alpha_{\SS^1}^z\tfrac{\norm{z}_{1,2}}{\norm{z}_2^2}\norm{\xi}_{1,2}
   .
   \qquad \text{\small\color{gray} (see $T_{71}$)}
\end{split}
\end{equation}
As usual on level two $z\mapsto T_{75}^z$
{\color{blue}extends to $ u_1$} and induces a
compact operator
$$
    h_2\stackrel{\rm cp.}{\INTO}
    h_1\stackrel{\rm bd}{\longrightarrow}
    h_1
   .
$$
On level one we only get boundedness $ h_1\to  h_0$, to prove
compactness is harder; see Lemma~\ref{le:f_65_75-comp}.







\bibliographystyle{alpha}
\addcontentsline{toc}{section}{References}
\small
\bibliography{$HOME/Dropbox/0-Libraries+app-data/Bibdesk-BibFiles/library_math,$HOME/Dropbox/0-Libraries+app-data/Bibdesk-BibFiles/library_math_2020,$HOME/Dropbox/0-Libraries+app-data/Bibdesk-BibFiles/library_physics}{}

\begin{thebibliography}{FW26d}

\bibitem[AP93]{ambrosetti:1993a}
Antonio Ambrosetti and Giovanni Prodi.
\newblock {\em A primer of nonlinear analysis}, volume~34 of {\em Cambridge
  Studies in Advanced Mathematics}.
\newblock Cambridge University Press, Cambridge, 1993.

\bibitem[BH21]{Behzadan:2021a}
A.~Behzadan and M.~Holst.
\newblock Multiplication in {S}obolev spaces, revisited.
\newblock {\em Ark. Mat.}, 59(2):275--306, 2021.

\bibitem[BOV21]{Barutello:2021b}
Vivina Barutello, Rafael Ortega, and Gianmaria Verzini.
\newblock Regularized variational principles for the perturbed {K}epler
  problem.
\newblock {\em Adv. Math.}, 383:Paper No. 107694, 64, 2021.
\newblock \href{https://arxiv.org/abs/2003.09383}{arXiv:2003.09383}.

\bibitem[FW24]{Frauenfelder:2024c}
Urs {Frauenfelder} and Joa {Weber}.
\newblock {Growth of eigenvalues of Floer Hessians}.
\newblock {\em \href{https://vixra.org/author/joa_weber}{viXra e-prints}
  {\small science, freedom, dignity}}, pages 1--50, August 2024.
\newblock \href{https://vixra.org/abs/2411.0060}{viXra:2411.0060}.

\bibitem[FW25]{Frauenfelder:2025e}
Urs {Frauenfelder} and Joa {Weber}.
\newblock {Hilbert manifold structures on path spaces}.
\newblock {\em \href{https://vixra.org/author/joa_weber}{viXra e-prints}
  {\small science, freedom, dignity}}, pages 1--81, July 2025.
\newblock \href{https://vixra.org/abs/2507.0031}{viXra: 2507.0031}.

\bibitem[FW26a]{Frauenfelder:2026c}
Urs Frauenfelder and Joa Weber.
\newblock {Loop space blow-up and scale calculus}.
\newblock {\em Arch. Math. (Basel)}, 126:335--342, 2026.
\newblock \href{https://rdcu.be/eZpmp}{Open access}.

\bibitem[FW26b]{Frauenfelder:2026a}
Urs {Frauenfelder} and Joa {Weber}.
\newblock {Merry-go-round and time-dependent symplectic forms}.
\newblock {\em \href{https://vixra.org/author/joa_weber}{viXra e-prints}
  {\small science, freedom, dignity}}, pages 1--20, January 2026.
\newblock \href{https://vixra.org/abs/2601.0019}{viXra: 2601.0019}.

\bibitem[FW26c]{Frauenfelder:2025g}
Urs {Frauenfelder} and Joa {Weber}.
\newblock {The linearized Floer equation in a chart}.
\newblock {\em SIGMA}, 22(032):38 pages, 2026.
\newblock \href{https://sigma-journal.com/Merry.html}{Special Issue} on
  Geometry and Dynamics in memory of Will Merry.
  \href{https://doi.org/10.3842/SIGMA.2026.032}{Open access}.

\bibitem[FW26d]{Frauenfelder:2026b}
Urs {Frauenfelder} and Joa {Weber}.
\newblock {Towards a Floer theory for Mars~I -- Twisted Zeeman systems}.
\newblock {\em \href{https://vixra.org/author/joa_weber}{viXra e-prints}
  {\small science, freedom, dignity}}, May 2026.
\newblock \href{https://vixra.org/abs/2605.0112}{viXra: 2605.0112}.

\bibitem[FW26e]{Frauenfelder:2026e}
Urs {Frauenfelder} and Joa {Weber}.
\newblock {Towards a Floer theory for Mars~III -- Nonlocal Floer-Morse index
  correspondence}.
\newblock {\em \href{https://vixra.org/author/joa_weber}{In preparation}},
  2026.

\bibitem[MS04]{mcduff:2004a}
Dusa McDuff and Dietmar Salamon.
\newblock {\em {$J$}-holomorphic curves and symplectic topology}, volume~52 of
  {\em American Mathematical Society Colloquium Publications}.
\newblock American Mathematical Society, Providence, RI, 2004.

\bibitem[M{\"{u}}l07]{Muller:2007a}
Vladimir M{\"{u}}ller.
\newblock {\em {Spectral Theory of Linear Operators -- and Spectral Systems in
  Banach Algebras}}, volume 139 of {\em {Operator Theory: Advances and
  Applications}}.
\newblock Birkh\"{a}user Verlag, Basel, 2nd edition, 2007.

\bibitem[Tay96]{Taylor:1996a}
Michael~E. Taylor.
\newblock {\em Partial differential equations. Basic theory.}, volume~23 of
  {\em Texts in Applied Mathematics}.
\newblock Springer-Verlag, New York, 1996.

\end{thebibliography}

%


\end{document}